\documentclass[12pt]{article}

\usepackage{amsfonts}
\usepackage{amssymb}
\usepackage{amsmath}
\usepackage{mathrsfs}
\usepackage[colorlinks, linkcolor=blue, anchorcolor=blue, citecolor=blue]{hyperref}
\usepackage[normalem]{ulem}
\usepackage{color}
\usepackage{leftidx}
\usepackage{verbatim}
\usepackage{enumerate}
\usepackage{amsthm}
\usepackage{geometry}
\usepackage{authblk}
\usepackage[toc,page]{appendix}

\newcommand{\E}{\mathbb{E}}
\newcommand{\R}{\mathbb{R}}

\newcommand{\D}{\mathbb{D}}
\newcommand{\1}{\mathbf{1}}
\newcommand{\nt}{ \lfloor nt \rfloor}

\newtheorem{theorem}{Theorem}[section]
\newtheorem{proposition}[theorem]{Proposition}
\newtheorem{lemma}[theorem]{Lemma}

\newtheorem{definition}[theorem]{Definition}
\begin{document}
\numberwithin{equation}{section}
\title{H\"{o}lder continuity of the solutions to a class of SPDEs arising from multidimensional superprocesses in random environment
}

\author[1]{Yaozhong Hu \footnote{Supported by an NSERC discovery  grant.\  Email: yaozhong@ualberta.ca}}

\author[2]{David Nualart \footnote{Supported by the NSF grant DMS 1811181. \  Email: nualart@ku.edu}}
\author[2]{Panqiu Xia \footnote{Email: pqxia@ku.edu}}

\affil[1]{\small Department of Mathematical and Statistical Science, University of Alberta at Edmonton, Edmonton, Canada, T6G 2G1.}

\affil[2]{Department of Mathematics, University of Kansas, Lawrence, KS, 66044, USA.}
\date{}
\maketitle

\begin{abstract}	We consider a $d$-dimensional branching particle system in a random environment. Suppose that the initial measures converge weakly to a measure with bounded density. Under the Mytnik-Sturm branching mechanism, we prove that the corresponding empirical measure $X_t^n$  converges weakly in the Skorohod space $D([0,T];M_F(\R^d))$ 
 and the limit has a density $u_t(x)$, where $M_F(\R^d)$ is the space of finite measures on $\R^d$. We also derive a stochastic partial differential equation    $u_t(x)$ satisfies. By using the techniques of Malliavin calculus, we prove that $u_t(x)$ is jointly H\"{o}lder continuous in time with exponent $\frac{1}{2}-\epsilon$ and in space with exponent $1-\epsilon$ for any $\epsilon>0$.  
\end{abstract}

\section{Introduction}
Consider a $d$-dimensional branching particle system in a random environment. For any integer $n\geq 1$, the branching events happen at time $\frac{k}{n}$, $k=1,2,\dots$. The dynamics of each particle, labelled by a multi-index $\alpha$, is described by the stochastic differential equation (SDE):
\begin{align}\label{pmer}
dx^{\alpha,n}_t=dB^{\alpha}_t+\int_{\mathbb{R}^d} h(y-x^{\alpha,n}_t)W(dt,dy),
\end{align}
where $h$ is a $d\times d$ matrix-valued function on $\mathbb{R}^d$, whose entries $h^{ij} \in L^2 (\mathbb{R}^d)$, $B^{\alpha}$ are $d$-dimensional independent Brownian motions, and $W$ is a $d$-dimensional space-time white Gaussian random field on $\R_+\times \R^d$ independent of the family $\{B^{\alpha}\}$. The random field $W$ can be regarded as the random environment for the particle system.

At any   branching time each particle dies and it randomly generates offspring. The new particles are born at the death position of their parents, and inherit the  branching-dynamics mechanism.
The branching mechanism we use in this paper follows the one 
 introduced by Mytnik \cite{ap-96-mytnik}, and studied further  by Sturm \cite{ejp-03-sturm}. Let $X^n=\{X^n_t, t\geq 0\}$ denote the empirical measure of the particle system.   One of the main results  of this work is to
  prove that the empirical measure-valued processes converge weakly to a process $X=\{X_t, t\geq 0\}$, such that for almost every $t\geq 0$, $X_t$ has a density $u_t(x)$ almost surely. By using the techniques of Malliavin calculus, we also establish the almost surely joint H\"{o}lder continuity of $u$ with exponent $\frac{1}{2}-\epsilon$ in time and $1-\epsilon$ in space for any $\epsilon>0$.  

To compare our results with the classical ones.
Let us recall   briefly some existing work 
in the literature.    The one-dimensional model was initially introduced and 
studied by Wang (\cite{ptrf-97-wang, saa-98-wang}). In these papers, he proved that under the classical Dawson-Watanabe branching mechanism, the empirical measure $X^n$ converges weakly to a   process $X=\{X_t, t\geq 0\}$, which  is the unique solution to a martingale problem. 

For  the above one dimensional model Dawson et al. \cite{aihpps-00-dawson-vaillancourt-wang} proved that for almost every $t>0$, the limit measure-value process $X$ has a density $u_t(x)$ a.s. 
and  $u$ is the weak solution to the following  stochastic partial differential equation (SPDE):
\begin{align}\label{ldf1}
u_t(x)=&\mu(x)+\int_0^t \frac{1}{2}(1+\|h\|_2^2)\Delta u_s(x)ds-\int_0^t\int_{\mathbb{R}}\nabla_x[h(y-x)u_s(x)]W(ds,dy)\nonumber\\
&+\int_0^t\sqrt{u_s(x)}\frac{V(ds,dx)}{dx},
\end{align}
where $\|h\|_2$ is the $L^2$-norm of $h$, and $V$ is a space-time white Gaussian random field on $\mathbb{R}_+\times \mathbb{R}$ independent of $W$.

Suppose further that $h$ is in the Sobolev space $H_2^2(\mathbb{R})$ and the initial measure has a density $\mu\in H_2^1(\R)$. Then Li et al. \cite{ptrf-12-li-wang-xiong-zhou} proved  that $u_t(x)$ is almost surely jointly H\"{o}lder continuous. By using the techniques of Malliavin calculus, Hu et al. \cite{ptrf-13-hu-lu-nualart} improved their result to obtain the  sharp H\"older continuity: they proved that the H\"{o}lder exponents are $\frac{1}{4}-\epsilon$ in time and $\frac{1}{2}-\epsilon$ in space, for any $\epsilon>0$.

Our paper is concerned with higher dimension ($d>1$). 
However in this case, the super Brownian motion (a special case when $h=0$) does not have a density (see e.g.  Corollary 2.4 of Dawson and Hochberg \cite{ap-79-dawson-hochberg}). Thus 
in higher dimensional case  we   have to abandon the   
classical Dawson-Watanabe branching mechanism  and adopt 
the Mytnik-Sturm one.  As a consequence, the difficult  term 
$\sqrt{u_s(x)}$ in the   SPDE \eqref{ldf1}   becomes 
$ u_s(x)$ (see equation \eqref{dqmvp}  in Section \ref{s.3} for
the exact form of the equation).  In Section \ref{s.2} 
 we shall briefly 
describe the branching mechanism used in this paper.  

 
In  Section  \ref{s.3} we   state the main results obtained in this paper.  These include three theorems. The first one 
(Theorem \ref{unique}) is about the existence and uniqueness of a (linear) stochastic partial differential equation (equation \eqref{dqmvp}), which is proved (Theorem \ref{tmpd})
to be satisfied by the density of the limiting  empirical measure
process $X^n$  of the particle system (see \eqref{napem}).
The core result of this paper is Theorem \ref{tjhc} which 
intends to give sharp   H\"older continuity of the solution $u_t(x)$ 
to \eqref{dqmvp}.  

Section 4 presents the proofs for Theorems \ref{unique} and \ref{tmpd}.   
The proof of Theorem  \ref{tjhc} is the objective of  the remaining sections. 
First, 
in Section 5, we focus on the one-particle motion with no branching. By using the techniques from  Malliavin calculus, we obtain a Gaussian type estimates for the transition probability density of the particle motion conditional on $W$.  This estimate plays a crucial role in the proof of Theorem  \ref{tjhc}. 
In Section 6, we derive a conditional convolution representation of the weak solution to the SPDE (\ref{dqmvp}), which is used to establish the H\"{o}lder continuity.
In Section 7, we show that the solution $u$
to \eqref{dqmvp} is H\"{o}lder continuous.

\section{Branching particle system}\label{s.2}
In this section, we briefly construct the branching particle system. 
For further study  of this branching mechanism, we refer the readers to  Mytnik's  and Sturm's  papers (see \cite{ap-96-mytnik, ejp-03-sturm}).

We start this section by introducing some notation. For any integer $k\geq 0$, we denote by $C_b^k(\R^d)$ the space of $k$ times continuously differentiable functions on $\R^d$ which are bounded together with their derivatives up to the order $k$. Also, $H_2^k(\R^d)$ is the Sobolev space of square integrable functions on $\R^d$ which have square integrable derivatives up to the order $k$. For any differentiable function $\phi$ on $\R^d$, we make use of the notation $\partial_{i_1\cdots i_m} \phi(x)=\frac{\partial^m}{\partial x_{i_1}   \cdots \partial x_{i_m}}\phi(x)$. 

We write $M_F(\mathbb{R}^d)$ for the space of finite measures on $\mathbb{R}^d$. We denote by  \break $D([0,T], M_F(\mathbb{R}^d))$ the Skorohod space of c\`{a}dl\`{a}g functions on time interval $[0,T]$, taking values in $M_F(\mathbb{R}^d)$, and equipped with the weak topology. For any $\phi\in C_b(\R^d)$ and $\mu\in M_F(\R^d)$, we write 
\begin{align}\label{itg}
\langle \mu, \phi \rangle=\mu(\phi):=\int_{\R^d}\phi(x)\mu(dx).
\end{align}

Let $\mathcal{I}:=\{\alpha=(\alpha_0,\alpha_1,\dots,\alpha_N), \alpha_0\in\{1,2,3\dots\}, \alpha_i\in\{1,2\},\ \textrm{for}\  1\leq i\leq N\}$ be a set of multi-indexes. In our model $\mathcal{I}$ is the index set of all possible particles. In other words, initially there are a finite number of particles and each particle generates at most $2$ offspring. For any particle $\alpha=(\alpha_0,\alpha_1,\dots, \alpha_N)\in\mathcal{I}$, let  $\alpha -1=(\alpha_0\dots, \alpha_{N-1}), \alpha -2=(\alpha_0,\dots, \alpha_{N-2}), \dots, \alpha -N=(\alpha_0)$ be the ancestors of $\alpha$. Then, $|\alpha|=N$ is the number of the ancestors of the particle $\alpha$. It is easy to see that the ancestors of any particle $\alpha$ are uniquely determined.

Fix a time interval $[0,T]$. Let $(\Omega, \mathcal{F}, P)$ be a complete probability space, on which $\{B^{\alpha}_t, t\in[0,T]\}_{\alpha\in \mathcal{I}}$ are independent $d$-dimensional standard Brownian motions, and $W$ is a $d$-dimensional space-time white Gaussian random field on $[0,T]\times \mathbb{R}^d$ independent of the family $\{B^{\alpha}\}$.

 Let $x_t=x(x_0,B^{\alpha},r,t)$, $t\in[0,T]$, be the unique solution to the following SDE:
\begin{align}\label{pmern}
x_t=x_0+B^{\alpha}_t-B^{\alpha}_r+\int_r^t\int_{\mathbb{R}^d} h(y-x_s)W(ds,dy),
\end{align}
where $x_0\in\mathbb{R}^d$, and $h$ is a $d\times d$ matrix-valued function, with entries $h^{ij}\in H_2^3(\mathbb{R}^d)$. We denote by $H^3_2(\R^d;\R^d\otimes \R^d)$ the space of such functions $h$, and equip it with the Sobolev norm:
\begin{align*}
\|h\|_{3,2}^2:=\sum_{i,j=1}^d\|h^{ij}\|_{3,2}^2.
\end{align*}
Let $\rho:\R^d\to\R^d\otimes\R^d$ be given by
\begin{align}\label{rho}
\rho(x)=\int_{\mathbb{R}^d}h(z-x)h^*(z)dz.
\end{align}
Then, for any $1\leq i,j\leq d$, and $x\in\R^d$, by Cauchy-Schwarz's inequality, we have 
\begin{align*}
|\rho^{ij}(x)|\leq \sum_{k=1}^d\|h^{ik}\|_2\|h^{kj}\|_2.
\end{align*}
We denote by $\|\cdot\|_2$ the Hilbert Schmidt norm for matrices. Then, by Cauchy-Schwarz's inequality again, we have
\begin{align*}
\|\rho\|_{\infty}:=&\sup_{x\in \R^d}\|\rho(x)\|_2=\sup_{x\in\R^d}\Big(\sum_{i,j=1}^d|\rho^{ij}(x)|^2\Big)^{\frac{1}{2}}\\
\leq &\Big(\sum_{i,j=1}^d\Big|\sum_{k=1}^d\|h^{ik}\|_2\|h^{kj}\|_2\Big|^2\Big)^{\frac{1}{2}}\\
\leq& \Big(\sum_{i,k=1}^d\|h^{ik}\|_2^2\sum_{j,k=1}^d\|h^{kj}\|_2^2\Big)^{\frac{1}{2}}\leq \|h\|_2^2\le  \|h\|_{3,2}^2.
\end{align*}
Denote by $A$ the infinitesimal generator of $x_t$. That is, $A$ is a differential operator on $C_b^2(\mathbb{R}^d)$, with values in $C_b(\R^d)$, given by
\begin{align}\label{fpgt}
A\phi(x)=\frac{1}{2}\sum_{i,j=1}^d \left(\rho^{ij}(0)\partial_{ij}\phi(x)\right)+\frac{1}{2}\Delta \phi(x),\ x\in \R^d.
\end{align}

For any $t\in[0,T]$, let $t_n=\frac{\nt}{n}$ be the last branching time before $t$. For any $\alpha=(\alpha_0,\alpha_1,\dots, \alpha_N)$, if $nt_n=\nt\leq N$, let $\alpha_t=(\alpha_0,\dots, \alpha_{\nt})$ be the ancestor of $\alpha$ at time $t$. Suppose that each particle, which starts from the death place of its parent, moves in $\R^d$ following the motion described by the SDE (\ref{pmern}) during its lifetime. Then, the path of any particle $\alpha$ and all its ancestors, denoted by $x^{\alpha,n}_t$, is given by
  \begin{equation*}
x_t^{\alpha,n}=x_t^{\alpha_t,n}=\begin{cases}
x\left(x^n_{\alpha_0}, B^{(\alpha_0)}, 0, t\right),\quad &0\leq t<\frac{1}{n},\\
 x\big(x^{\alpha_t -1, n}_{t_n^-}, B^{\alpha_t}, t_n, t\big),\quad &\frac{1}{n}\leq t< \frac{N+1}{n},\\
 \partial, & \mathrm{otherwise}.
 \end{cases}
 \end{equation*}
Here $x_{\alpha_0}^n\in\R^d$ is the initial position of particle $(\alpha_0)$, $x^{\alpha_t-1,n}_{t_n^-}:=\lim_{s\uparrow t_n}x^{\alpha_t-1,n}_s$, and  $\partial$ denotes the ``cemetery''-state, that can be understood as a point at infinity.

Let $\xi=\{\xi(x),x\in\R^d\}$ be a real-valued random field on $\R^d$ with covariance 
\begin{align}\label{crlt}
\E\big(\xi(x)\xi(y)\big)=\kappa(x,y),
\end{align}
for all $x,y\in\R^d$. Assume that $\xi$ satisfies the following conditions:
\begin{enumerate}[\textbf{[H1]}]\label{h1}
\item \begin{enumerate}[(i)]
\item $\xi$ is symmetric, that is $\mathbb{P}(\xi(x)> z)=\mathbb{P}(\xi(x)<-z)$ for all $x\in\R^d$ and $z\in\R$.
\item $\displaystyle\sup_{x\in \R^d}\E\big(|\xi(x)|^p\big)<\infty$ for some $p>2$.
\item $\kappa$ vanishes at infinity, that is $\displaystyle\lim_{|x|+|y|\to\infty}\kappa(x,y)=0$.
\end{enumerate}
\end{enumerate}
For any $n\geq 1$, the random field $\xi$ is used to define the offspring distribution after a scaling $\frac{1}{\sqrt{n}}$. 
In order to make the offspring distribution  a probability measure, we introduce the truncation of the random field $\xi$, denoted by $\xi^n$, as follows:
\begin{equation}
\xi^n(x)=
\begin{cases}
\sqrt{n},& \text{if}\ \xi(x)>\sqrt{n},\\
-\sqrt{n},& \text{if}\ \xi(x)<-\sqrt{n},\\
\xi(x),& \text{otherwise}.
\end{cases}
\end{equation}
The correlation function of the truncated random field is then given by 
\[
\kappa_n(x,y)=\E\big(\xi^n(x)\xi^n(y)\big).
\]

Let $(\xi^n_i)_{i\geq 0}$ be independent copies of $\xi^n$. Denote by $\xi_i^{n+}$ and $\xi_i^{n-}$ the positive and negative part of $\xi^n_i$ respectively. Let $N^{\alpha, n}\in\{0,1,2\}$ be the offspring number of the particle $\alpha$ at the branching time $\frac{|\alpha|+1}{n}$. Assume that $\{N^{\alpha,n}, |\alpha|=i\}$ are conditionally independent given $\xi^n_i$ and the position of $\alpha$ at its branching time, with a distribution given by
\begin{align*}
&P\Big(\left.N^{\alpha,n}=2\right|\xi^n_i, x^{\alpha,n}_{\frac{i+1}{n}^-}\Big)=\frac{1}{\sqrt{n}}\xi_i^{n+}\Big(x^{\alpha,n}_{\frac{i+1}{n}^-}\Big),\\
&P\Big(\left.N^{\alpha,n}=0\right|\xi^n_i, x^{\alpha,n}_{\frac{i+1}{n}^-}\Big)=\frac{1}{\sqrt{n}}\xi_i^{n-}\Big(x^{\alpha,n}_{\frac{i+1}{n}^-}\Big),\\
&P\Big(\left.N^{\alpha,n}=1\right|\xi^n_i, x^{\alpha,n}_{\frac{i+1}{n}^-}\Big)=1-\frac{1}{\sqrt{n}}|\xi_i^{n}|\Big(x^{\alpha,n}_{\frac{i+1}{n}^-}\Big).
\end{align*}

For any particle $\alpha=(\alpha_0,\dots, \alpha_N)$, $\alpha$ is called to be alive at time $t$, denoted by $\alpha\sim_n t$, if the following conditions are satisfied:
\begin{enumerate}[(i)]
\item There are exactly  $N$ branching before or at $t$: $\nt=N$.
\item $\alpha$ has an unbroken ancestors line: $\alpha_{N-i+1}\leq N^{\alpha-i,n}$, for all $i=1, 2, \dots, N$.
\end{enumerate}
[Introduction of $N^{\alpha, n}$   allows 
 the particle $\alpha$ produce one more
 generation, namely, produce   new particle  $(\alpha, N^{\alpha, n})$.  However,
 $(\alpha, 0)$ is considered no longer alive and will not produce offspring any more.]
For any $n$, denote by $X^n=\{X^n_t, t\in[0,T]\}$ the 
empirical measure-valued process of the particle system. Then, $X^n$ is a discrete measure-valued process, given by
\begin{align}\label{napem}
X^n_t=\frac{1}{n}\sum_{\alpha\sim_n t}\delta_{x^{\alpha,n}_t},
\end{align}
where $\delta_x$ is the Dirac measure at $x\in\mathbb{R}^d$, and the sum is among all alive particles at time $t\in[0,T]$. Then, for any $\phi\in C^2_b(\R^d)$, with the notation (\ref{itg}), we have
\[
X^n_t(\phi)=\int_{\R^d}\phi(x)X^n_t(dx)=\frac{1}{n}\sum_{\alpha\sim n}\phi(x^{\alpha,n}_t).
\]

\section{Main results}\label{s.3}
Let $(\Omega, \mathcal{F}, \{F_t\}_{t\in[0,T]}, P)$ be a complete filtered probability space that satisfies the usual conditions. Suppose that $W$ is a $d$-dimensional space-time white Gaussian random field on $[0,T]\times \R^d$, and $V$ is a one-dimensional Gaussian  random field on $[0,T]\times \R^d$ independent of $W$, that is time white and spatially colored with correlation $\kappa$ defined in (\ref{crlt}):
\begin{align*}
\E\big(V(t,x)V(s,y)\big)=(t\wedge s)\kappa(x,y),
\end{align*}
for all $s,t\in[0,T]$ and $x,y\in\R^d$. Assume that $\{W(t,x), x\in \R^d\}$, $\{V(t,x), x\in\R^d\}$ are $\mathcal{F}_t$-measurable for all $t\in [0,T]$, and $\{W(t,x)-W(s,x), x\in \R^d\}$, $\{V(t,x)-V(s,x), x\in \R^d\}$ are independent of $\mathcal{F}_s$ for all $0\leq s< t\leq T$ .

Denote by $A^*$ the adjoint of $A$, where $A$ is the generator defined in (\ref{fpgt}). Consider the following SPDE:
\begin{align}\label{dqmvp}
u_t(x)=&\mu(x)+\int_0^t A^* u_s(x)ds-\sum_{i,j=1}^d\int_0^t\int_{\mathbb{R}^d}\frac{\partial}{\partial x_i} \left[h^{ij}(y-x)u_s(x)\right]W^j(ds,dy)\nonumber\\
&+\int_0^t u_s(x)\frac{V(ds,dx)}{dx}.
\end{align}

\begin{definition}\label{def}
Let $u=\{u_t(x),t\in[0,T],x\in \R^d\}$ be a random field. Then,
\begin{enumerate}[(i)]
\item $u$ is said to be a strong solution to the SPDE (\ref{dqmvp}), if $u$ is jointly measurable on $[0,T]\times \R^d\times \Omega$, adapted to $\{\mathcal{F}_t\}_{t\in[0,T]}$ and for any $\phi\in C_b^2(\R^d)$, the following equation holds  for almost every $t\in[0,T]$:
\begin{align}
\int_{\R^d}\phi(x)u_t(x)dx=&\int_{\R^d}\phi(x)\mu(x)dx+\int_0^t\int_{\R^d} A \phi(x)u_s(x)dxds\nonumber\\
&+\int_0^t\int_{\R^d}\Big[\int_{\R^d}\nabla \phi(x)^*h(y-x)u_s(x)dx\Big]W(ds,dy)\nonumber\\
&+\int_0^t \int_{\R^d}\phi(x)u_s(x)V(ds,dx),
\end{align}
almost surely, where the last two stochastic integrals are Walsh's integral (see e.g. Walsh \cite{springer-86-walsh}).
\item $u$ is said to be a weak solution to the SPDE (\ref{dqmvp}), if there exists a filtered probability space, on which $W$ and $V$ are independent random fields that satisfy the above properties, such that $u$ is a strong solution with this probability space.
\end{enumerate}
\end{definition}

In Section 4, we prove the following two theorems.
\begin{theorem}\label{unique}
The SPDE (\ref{dqmvp}) has a unique strong solution in the sense of Definition \ref{def}.
\end{theorem}

Let $X^n=\{X^n_t,0\leq t\leq T\}$ be defined in (\ref{napem}). In order to show the convergence of $X^n$ in $D([0,T];M_F(\R^d))$, we make use of the following hypothesis on the initial measures $X^n_0$:
\begin{enumerate}[{\textbf{[H2]}}]
\item \begin{enumerate}[(i)]\label{h2}
\item $\displaystyle \sup_{n\geq 1}|X_0^{(n)}(1)|<\infty$.
\item $X_0^n\Rightarrow X_0$ in $M_F(\mathbb{R}^d)$ as $n\to\infty$.
\item $X_0$ has a bounded density $\mu$.
\end{enumerate}
\end{enumerate}

\begin{theorem}\label{tmpd}
Let $X^n$ be defined in (\ref{napem}). Then, under hypotheses \textbf{[H1]} and \textbf{[H2]}, we have the following results:
\begin{enumerate}[(i)]
  \item $X^n \Rightarrow X$ in $D([0,T], M_F(\mathbb{R}^d))$ as $n\to\infty$.

\item For almost every $t\in[0,T]$, $X_t$ has a density $u_t(x)$ almost surely.

\item $u=\{u_t(x), t\in[0,T], x\in\R^d\}$ is a weak solution to the SPDE (\ref{dqmvp}) in the sense of Definition \ref{def}.
\end{enumerate}
\end{theorem}

The last main result in this paper is the following theorem
concerning the H\"older continuity of the solution to the SPDE \eqref{dqmvp}. 
 
\begin{theorem}\label{tjhc}
Let $u=\{u_t(x),t\in[0,T],x\in\R^d\}$ be the strong solution to the SPDE (\ref{dqmvp}) in the sense of Definition \ref{def}. Then, for any $\beta_1,\beta_2\in (0,1)$, $p>1$, $x,y\in\mathbb{R}^d$, and $0< s<t\leq T$, there exists a constant $C$ that depends on $T$, $d$, $h$, $p$, $\beta_1$, and $\beta_2$, such that
\begin{align*}
\left\|u_t(x)-u_s(y)\right\|_{2p}\leq Cs^{-\frac{1}{2}}\big(|x-y|^{\beta_1}+|t-s|^{\frac{1}{2}\beta_2}\big).
\end{align*}
Hence by  Kolmogorov's criteria, $u_t(x)$ is almost surely jointly H\"{o}lder continuous, with exponent $\beta_1\in(0,1)$ in space and $\beta_2\in(0,\frac{1}{2})$ in time.
\end{theorem}

\section{Proof of Theorems \ref{unique} and \ref{tmpd}}
We prove Theorems \ref{unique} and \ref{tmpd} in the following steps:
\begin{enumerate}[(i)]
  \item In Section 4.1, we show that  $\{X^n\}_{n\geq 1}$ is a tight sequence in $D([0,T];M_F(\R^d))$, and the limit of any convergent subsequence in law solves a martingale problem.

  \item In Section 4.2, we show that any solution to the martingale problem has a density almost surely.

  \item In Section 4.3, we show the equivalence between  martingale problem (see e.g. (\ref{mprf}) - (\ref{qvmp}) below) and the SPDE  (\ref{dqmvp}). Finally, we prove Theorems  \ref{unique} and \ref{tmpd}.
\end{enumerate}

\subsection{Tightness and martingale problem}
Recall the empirical measure-valued process $X^n=\{X^n_t, t\in[0,T]\}$ given by (\ref{napem}). Let $\phi\in C_b^2(\mathbb{R}^d)$, then similar to the identity  (49) of Sturm  \cite{ejp-03-sturm}, we can decompose $X^n_t$ as follows:
\begin{align}\label{dfx}
X_t^n(\phi)=X_0^n(\phi)+Z^n_t(\phi)+M^{b,n}_{t}(\phi)+B^n_t(\phi)+U^n_t(\phi),
\end{align}
where
\begin{align*}
Z^n_t(\phi)=\int_0^t X_u^n(A\phi)du,
\end{align*}
\begin{align*}
M_{t}^{b,n}(\phi)=M_{t_n}^{b,n}(\phi)=\frac{1}{n}\sum_{s_n<t_n}\sum_{\alpha\sim_n s_n}\phi \big(x^{\alpha, n}_{s_n+\frac{1}{n}}\big)(N^{\alpha,n}-1),
\end{align*}
\begin{align*}
B^n_t(\phi)=\frac{1}{n}\Big(\sum_{s_n<t_n}\sum_{\alpha\sim_n s_n}\int_{s_n}^{s_n+\frac{1}{n}}\nabla\phi(x^{\alpha,n}_u)^* dB^{\alpha}_u+\sum_{\alpha\sim_n t}\int_{t_n}^t\nabla\phi(x^{\alpha,n}_u)^*dB^{\alpha}_u\Big),
\end{align*}
and
\begin{align*}
U_t^n(\phi)=&\frac{1}{n}\Big(\sum_{s_n<t_n}\sum_{\alpha\sim_n s_n}\int_{s_n}^{s_n+\frac{1}{n}}\int_{\mathbb{R}^d}\nabla\phi(x^{\alpha,n}_u)^* h(y-x^{\alpha,n}_u)W(du,dy)\\
&\qquad+\sum_{\alpha\sim_n t}\int_{t_n}^t\int_{\mathbb{R}^d}\nabla\phi(x^{\alpha,n}_u)^* h(y-x^{\alpha,n}_u)W(du,dy)\Big)\\
=&\int_0^t\int_{\mathbb{R}^d}\Big(\int_{\mathbb{R}^d}\nabla\phi(x)^* h(y-x)X_u(dx)\Big)W(du,dy).
\end{align*}
As in Sturm \cite{ejp-03-sturm}, consider the natural filtration, generated by the process $X^n$
\[
\mathcal{F}^n_t=\sigma\left(\{x^{\alpha,n},N^{\alpha,n}\big||\alpha|<\nt\}\cup\{x^{\alpha,n}_s, s\leq t, |\alpha|=\nt\}\right),
\]
and a discrete filtration at branching times
\[
\widetilde{\mathcal{F}}^n_{t_n}=\sigma\big(\mathcal{F}^n_{t_n}\cup\{x^{\alpha,n}\big||\alpha|=nt_n\}\big)=\mathcal{F}^n_{(t_n+n^{-1})^-}.
\]
Then, $B^n_t(\phi)$ and $U^n_t(\phi)$ are continuous $\mathcal{F}^n_t$-martingales, while $M^{b,n}_{t}(\phi)$ is a discrete $\widetilde{\mathcal{F}}^n_{t_n}$-martingale.

\begin{lemma}\label{ublxmu}
Assume hypotheses \textbf{[H1]}, \textbf{[H2]} (i) and (ii). Let $p>2$ be given in  hypothesis \textbf{[H1]}. Then, for all $\phi\in C_b^2(\R^d)$, 
\begin{enumerate}[(i)]
\item $\displaystyle\E\Big(\sup_{0\leq t\leq T}|X^n_t(\phi)|^p\Big)$, $\displaystyle\E\Big(\sup_{0\leq t\leq T}|M^{b,n}_{t}(\phi)|^p\Big)$, and $\displaystyle\E\Big(\sup_{0\leq t\leq T}|U^{n}_{t}(\phi)|^p\Big)$ are bounded uniformly in $n\geq 1$.
\item $\displaystyle\E\Big(\sup_{0\leq t\leq T}|B^n_t(\phi)|^p\Big)\to 0$, as $n\to\infty$.
\end{enumerate}
\end{lemma}
\begin{proof}
By the same argument as that for  Lemma 3.1 of Sturm 
\cite{ejp-03-sturm}, we can show that
\[
\E\Big(\sup_{0\leq t\leq T}|X^n_t(\phi)|^p\Big)+\E\Big(\sup_{0\leq t\leq T}|M^{b,n}_{t}(\phi)|^p\Big)\leq C \E\Big(\sup_{0\leq t\leq T}|X^{n}_{t}(1)|^p\Big),
\]
where the constant $C>0$ does not depend on $n$. Notice that $U^n_t(1)\equiv 0$, therefore $X^n_t(1)$ here is not different from the variable in Sturm's model. Thus we can simply refer to her result: $\displaystyle\E\Big(\sup_{0\leq t\leq T}|X^{n}_{t}(1)|^p\Big)$ is bounded uniformly in $n$. Therefore, $\displaystyle\E\Big(\sup_{0\leq t\leq T}|X^n_t(\phi)|^p\Big)$ and $\displaystyle\E\Big(\sup_{0\leq t\leq T}|M^{b,n}_t(\phi)|^p\Big)$ are also uniformly bounded in $n$.

For $U^{n}_t(\phi)$, by using the stochastic Fubini theorem, we have
\begin{align}\label{qvwnr}
\langle U^n(\phi)\rangle_t=&\Big\langle  \int_0^{\cdot}\int_{\mathbb{R}^d}\Big(\int_{\mathbb{R}^d}\nabla\phi(x)^*h(y-x)X^n_u(dx)\Big)W(du,dy)\Big\rangle_t\nonumber\\
=&\sum_{j=1}^d \int_0^t\int_{\mathbb{R}^d}\Big(\sum_{i=1}^d\int_{\mathbb{R}^d}\partial_i\phi(x)h^{ij}(y-x)X^n_u(dx)\Big)^2dydu\nonumber\\
=&\int_0^t\int_{\mathbb{R}^d\times\mathbb{R}^d}\nabla \phi(x)^*\rho(x-z)\nabla\phi(z)X^n_u(dx)X^n_u(dz)du\nonumber\\
\leq &\|\rho\|_{\infty}\|\phi\|_{1,\infty}^2\int_0^T \left|X_u^n(1)\right|^2 du.
\end{align}
Thus by (\ref{qvwnr}), Burkholder-Davis-Gundy's and Jensen's inequalities,  we have
\begin{align}\label{ubpu}
\E\Big(\sup_{0\leq t\leq T}\left|U^n_t(\phi)\right|^p\Big)\leq &c_p \E\big(\left\langle U^n(\phi)\right\rangle_T^\frac{p}{2}\big)\leq c_p\|\rho\|_{\infty}^{\frac{p}{2}}\|\phi\|_{1,\infty}^pT^{\frac{p}{2}-1}\E\Big(\int_0^T|X_u^n(1)|^p du\Big)\nonumber\\
\leq &c_p\|\rho\|_{\infty}^{\frac{p}{2}}\|\phi\|_{1,\infty}^pT^{\frac{p}{2}}\E\Big(\sup_{0\leq t\leq T}\left|X_t^n(1)\right|^p\Big),
\end{align}
that is also uniformly bounded in $n$.

Note that $\{B^{\alpha}\}$ are independent Brownian motions. Then, by Burkholder-Davis-Gundy's inequality, we have
\begin{align*}
&\E\Big(\sup_{0\leq t\leq T}|B^n_t(\phi)|^2\Big)\leq\frac{c_2}{n^2}\Big[\sum_{s_n<T_n}\sum_{\alpha\sim_n s_n}\E\Big(\int_{s_n}^{s_n+\frac{1}{n}}|\nabla\phi(x^{\alpha,n}_u) |^2du\Big)\\
&\hspace{35mm}+\sum_{\alpha\sim_n t}\E\Big(\int_{T_n}^T|\nabla\phi(x^{\alpha,n}_u)|^2du\Big)\Big]\\
&\hspace{10mm}=\frac{c_2}{n}\E\Big(\int_0^T\int_{\R^d}|\nabla\phi(x)|^2X_u(dx)du\Big)\leq \frac{c_2}{n}\|\phi\|_{1,\infty}^2T\E\Big(\sup_{0\leq t\leq T}|X^{n}_{t}(1)|^p\Big)\to 0,
\end{align*}
because $\displaystyle\E\Big(\sup_{0\leq t\leq T}|X^{n}_{t}(1)|^p\Big)$ is uniformly bounded in $n$.
\end{proof}
As a consequence of Lemma \ref{ublxmu}, the collection 
\[
\Big\{\sup_{0\leq t\leq T}|X_t^n(\phi)|^2, \sup_{0\leq t\leq T}|M_t^{b,n}(\phi)|^2,\sup_{0\leq t\leq T}|U_t^n(\phi)|^2\Big\}_{n\geq 1}
\]
is uniformly integrable.

\begin{definition}
Let $\{X^{\alpha}\}$ be a collection of real-valued stochastic processes. A family of stochastic processes 
 $\{X^{\alpha}\}$ is said to be C-tight, if it is tight, and the limit of any subsequence is continuous.
\end{definition}

\begin{lemma}\label{ctwqv}
For all $\phi\in C^2_b(\mathbb{R}^d)$, $M^{b,n}(\phi)$, $Z^{n}(\phi)$, and $U^n(\phi)$ are C-tight sequences in $D([0, T],\mathbb{R})$. As a consequence, $X^{n}(\phi)$ is C-tight in $D([0, T],\mathbb{R})$.
\end{lemma}

\begin{proof}
Note that the branching martingale $M^{b,n}(\phi)$ and drift term $Z^n(\phi)$ are the same as in Sturm \cite{ejp-03-sturm}. Thus by the proof of Lemma 3.6  in \cite{ejp-03-sturm}, we deduce the C-tightness of $M^{b,n}(\phi)$ and $Z^{n}(\phi)$.

We prove the tightness of $U^n_t(\phi)$ by checking Aldous's conditions (see Theorem 4.5.4 of Dawson \cite{springer-92-dawson}). By Chebyshev's inequality, for any fixed $t\in [0, T]$, and $N>0$, we have
\begin{align*}
\mathbb{P}\left(|U^n_t(\phi)|>N\right)\leq \frac{1}{N^p}\E\big(\left|U^n_t(\phi)\right|^p\big)\leq \frac{1}{N^p}\E\Big(\sup_{0\leq t\leq T}\left|U^n_t(\phi)\right|^p\Big)\to 0,
\end{align*}
uniformly in $n$ as $N\to\infty$ by Lemma \ref{ublxmu} (i). 

On the other hand, let $\{\tau_n\}_{n\geq 1}$ be any collection of stopping times bounded by $T$, and let $\{\delta_n\}_{n\geq 1}$ be any positive sequence that decreases to $0$. Then, due to (\ref{ubpu}) and the strong Markov property of It\^{o}'s diffusion $U^n(\phi)$, we have
\begin{align*}
\mathbb{P}\left(\big|U^n_{\tau_n+\delta_n}(\phi)-U^n_{\tau_n}(\phi)\big|>\epsilon\right)=&\mathbb{P}\left(\left|U^n_{\delta_n}(\phi)-U^n_0(\phi)\right|>\epsilon\right)\leq  \frac{1}{\epsilon^p}\E\big(\left|U^n_{\delta_n}(\phi)\right|^p\big)\\
\leq &\frac{\delta_n^{\frac{p}{2}}}{\epsilon^p}c_p\|\rho\|_{\infty}^{\frac{p}{2}}\|\phi\|_{1,\infty}^p\E\Big(\sup_{0\leq t\leq T}\left|X_t^n(1)\right|^p\Big)\to 0,
\end{align*}
as $n\to 0$. Thus both of Aldous's conditions are satisfied, and then it follows that $U^n_t(\phi)$ is tight in $D([0, T],\mathbb{R})$.

Notice that for any $n\geq 1$, $U^n_t(\phi)$ is a continuous martingale, then by Proposition VI.3.26 of Jacod and Shiryaev \cite{springer-13-jacod-shiryaev}, every limit of a tight sequence of continuous processes is also continuous.

Recall the decomposition formula (\ref{dfx}):
\begin{align*}
X_t^n(\phi)=X_0^n(\phi)+Z^n_t(\phi)+M^{b,n}_{t}(\phi)+B^n_t(\phi)+U^n_t(\phi).
\end{align*}
Notice that the first term converges weakly by assumption. The second, third, and last terms are C-tight as proved just now. The fourth term tends to $0$ in $L^2(\Omega)$ uniformly in $t\in[0,T]$ by Lemma  \ref{ublxmu},   (ii). As a consequence, $X^n(\phi)$ is C-tight in $D([0,T],\R)$.
\end{proof}

Let $\mathscr{S}=\mathscr{S}(\R^d)$ be the Schwartz space on $\R^d$, and let $\mathscr{S}'$ be the Schwartz dual space. Then, we have the following lemma:
\begin{lemma}\label{tight}
Assume  hypotheses \textbf{[H1]} and \textbf{[H2]} (i), (ii).  Then,
\begin{enumerate}[(i)]
\item $\{X^n\}_{n\geq 1}$ is a tight sequence in $D([0,T]; M_F(\R^d))$.
\item $\{B^n\}_{n\geq 1}$, $\{M^{b,n}\}_{n\geq 1}$, and $\{U^n\}_{n\geq 1}$ are C-tight in $D([0,T]; \mathscr{S}')$.
\end{enumerate}
\end{lemma}
\begin{proof}
Let $\widehat{\R}^d=\R^d\cup \{\partial\}$ be the one point compactification of $\R^d$. Then, by Theorem 4.6.1 of Dawson \cite{springer-92-dawson} and Lemma \ref{ctwqv}, $\{X^n\}_{n\geq 1}$ is a tight sequence in $D([0,T]; M_F(\widehat{\R}^d))$.

On the other hand, by the same argument as in Lemma 3.9 of Sturm \cite{ejp-03-sturm}, we can show that any limit of a weakly convergent subsequence $X^{n_k}$ in $D([0,T]; M_F(\widehat{\R}^d))$ belongs to $C([0,T]; M_F(\R^d))$, the space of continuous $M_F(\R^d)$-valued functions on $[0,T]$. Therefore, $\{X^n\}_{n\geq 1}$ is a tight sequence in $D([0,T]; M_F(\R^d))$.

To show the 
 property (ii), notice that $\mathscr{S}\subset C_b^2(\R^d)$. Then, by Theorem 4.1 of Mitoma \cite{ap-83-mitoma}, $\{B^n\}_{n\geq 1}$, $\{M^{b,n}\}_{n\geq 1}$, and $\{U^n\}_{n\geq 1}$ are C-tight in $D([0,T]; \mathscr{S}')$. 
\end{proof}

\begin{proposition}\label{propmp}
Assume hypotheses \textbf{[H1]}, \textbf{[H2]} (i) and (ii).  Let $X$ be the limit of a weakly convergent subsequence $\{X^{n_k}\}_{k\geq 1}$ in $D([0,T]; M_F(\R^d))$. Then, $X$ is a solution to the following martingale problem: for any $\phi \in C^2_b(\mathbb{R}^d)$, the process $M(\phi)=\{M_t(\phi):0\leq t\leq T\}$, given by
\begin{align}\label{mprf}
M_t(\phi):=&X_t(\phi)-X_0(\phi)-\int_0^tX_s(A\phi)ds,
\end{align}
is a continuous and  square integrable $\mathcal{F}^X_t$-adapted martingale with quadratic variation:
\begin{align}\label{qvmp}
\langle M(\phi)\rangle_t=&\int_0^t\int_{\mathbb{R}^d\times\mathbb{R}^d}\nabla \phi(x)^*\rho(x-y)\nabla\phi(y)X_s(dx)X_s(dy)ds\nonumber\\
&+\int_0^t\int_{\mathbb{R}^d\times\mathbb{R}^d}\kappa(x,y)\phi(x)\phi(y)X_s(dx)X_s(dy)ds.
\end{align}
\end{proposition}

\begin{proof}

Let $\{X^{n_k}\}_{k\geq 1}$  be a weakly convergent subsequence in $D([0,T]; M_F(\R^d))$. By taking further subsequences, we can assume, in view of Lemma \ref{tight} (ii), that  $\{B^{n_k}\}_{k\geq 1}$, $\{M^{b, n_k}\}_{k\geq 1}$, and $\{U^{n_k}\}_{k\geq 1}$ are weakly convergent subsequences in $D([0,T]; \mathscr{S}')$.

Then, by Skorohod's representation theorem, there exists a probability space $(\widetilde{\Omega}, \widetilde{\mathcal{F}}, \widetilde{\mathbb{P}})$, on which $(\widetilde{X}^{n_k}, \widetilde{M}^{b,n_k}, \widetilde{B}^{n_k}, \widetilde{U}^{n_k})$  has the same joint distribution as \break $(X^{n_k},  M^{b,n_k}, B^{n_k}, U^{n_k})$ for all $k\geq 1$, and converge a.s. to $(\widetilde{X},  \widetilde{M}^b, \widetilde{B}, \widetilde{U})$ in the product space $D([0,T],M_F(\widehat{\R}^d))\times D([0,T],\mathscr{S}')^3$.

Then, for any $\phi\in \mathscr{S}'$, $(\widetilde{X}^{n_k}(\phi), \widetilde{M}^{b,n_k}(\phi), \widetilde{B}^{n_k}(\phi), \widetilde{U}^{n_k}(\phi))$ converges a.s. in $D([0,T],\R)$. Since $\displaystyle\Big\{\sup_{0\leq t\leq T}|X_t^n(\phi)|^2, \sup_{0\leq t\leq T}|M_t^{b,n}(\phi)|^2, \sup_{0\leq t\leq T}|U_t^n(\phi)|^2\Big\}_{n\geq 1}$ is uniformly integrable, the convergence is also in $L^2([0,T]\times \Omega)$.

For any $t\in[0,T]$, let 
\[
\widetilde{M}_t^{n_k}(\phi):=\widetilde{X}^{n_k}_t(\phi)-\widetilde{X}^{n_k}_0(\phi)-\int_0^t\widetilde{X}^{n_k}_s(A\phi)ds=\widetilde{M}^{b,n_k}_t(\phi)+\widetilde{B}^{n_k}_t(\phi)+\widetilde{U}^{n_k}_t(\phi).
\]
Then, it converges to a continuous and  square integrable martingale $\widetilde{M}(\phi)=\widetilde{M}^{b}(\phi)+\widetilde{U}(\phi)$ in $L^2(\widetilde{\Omega})$. It suffices to compute its quadratic variation.

Notice that $W$ and $\{B^{\alpha}\}$ are independent, then $U^{n}$ and $B^{n}$ are orthogonal. As a consequence, $\widetilde{U}^{n_k}$ and $\widetilde{B}^{n_k}$ are also orthogonal. On the other hand, $\widetilde{M}^{b,n}(\phi)$ is a pure jump martingale, while $\widetilde{U}^{n_k}(\phi)$ and $\widetilde{B}^{n_k}(\phi)$ are continuous martingales. Due to Theorem 43 on page 353 of Dellacherie and Meyer \cite{northholland-82-dellacherie-meyer}, $\widetilde{M}^{b,n}(\phi)$,  $\widetilde{B}^{n_k}(\phi)$ and $\widetilde{U}^{n_k}(\phi)$ are mutually orthogonal. By the same argument as in Lemma \ref{ublxmu}, we can show that $\langle\widetilde{M}^{b,n_k}(\phi)+\widetilde{B}^{b,n_k}(\phi)+\widetilde{U}^{n_k}(\phi)\rangle_t=\langle\widetilde{M}^{b,n_k}(\phi)\rangle_t+\langle\widetilde{B}^{b,n_k}(\phi)\rangle_t+\langle\widetilde{U}^{n_k}(\phi)\rangle_t$ are uniformly integrable. Then, by Theorem II.4.5 of Perkins \cite{springer-02-perkins}, we have
\begin{align*}
\langle\widetilde{M}^{b,n_k}(\phi)+\widetilde{B}^{b,n_k}(\phi)+\widetilde{U}^{n_k}(\phi)\rangle_t&=\langle\widetilde{M}^{b,n_k}(\phi)\rangle_t+\langle\widetilde{B}^{b,n_k}(\phi)\rangle_t+\langle\widetilde{U}^{n_k}(\phi)\rangle_t \\
&\to\langle\widetilde{M}^{b}(\phi)\rangle_t+\langle\widetilde{U}(\phi)\rangle_t= \langle \widetilde{M}(\phi)\rangle_t
\end{align*}
as $k\to\infty$ in $D([0,T],\R)$ in probability. 

On the other hand, by the same argument of Lemma 3.8 of Sturm \cite{ejp-03-sturm}, we have
\begin{align*}
\langle \widetilde{M}^b(\phi)\rangle_t=\lim_{k\to\infty}\langle \widetilde{M}^{b,n_k}(\phi)\rangle_t=\int_0^t\int_{\mathbb{R}^d\times\mathbb{R}^d}\kappa(x,y)\phi(x)\phi(y)\widetilde{X}_s(dx)\widetilde{X}_s(dy)ds,\ a.s.
\end{align*}
For $\langle \widetilde{U}(\phi)\rangle_t$, by (\ref{qvwnr}), since $\widetilde{X}^{n_k}(\phi)\to \widetilde{X}(\phi)$ in $L^2([0,T]\times \Omega)$, it follows that
\begin{align*}
\lim_{k\to\infty}\langle\widetilde{U}^{n_k}(\phi)\rangle_t=&\lim_{k\to\infty}\int_0^t\int_{\mathbb{R}^d\times\mathbb{R}^d}\nabla \phi(x)^*\rho(x-z)\nabla\phi(z)\widetilde{X}^n_u(dx)\widetilde{X}^n_u(dz)du\\
=&\int_0^t\int_{\mathbb{R}^d\times\mathbb{R}^d}\nabla \phi(x)^*\rho(x-z)\nabla\phi(z)\widetilde{X}_u(dx)\widetilde{X}_u(dz)du.
\end{align*}
As a consequence, $\widetilde{M}=\{\widetilde{M_t},t\in[0,T]\}$, where
\[
\widetilde{M}_t(\phi)=\widetilde{X}_t(\phi)-\widetilde{X}_0(\phi)-\int_0^t\widetilde{X}_s(A\phi)ds=\widetilde{M}^b_t(\phi)+\widetilde{B}_t(\phi)+\widetilde{U}_t(\phi),
\]
is a continuously square integrate martingale with the quadratic variation given by  the expression (\ref{qvmp}) .

Finally, by the same argument as in Theorem II.4.2 of Perkins \cite{springer-02-perkins}, we can show $\widetilde{M}(\phi)$ is an  $\mathcal{F}^{\widetilde{X}}$-adapted martingale.
\end{proof}

\subsection{Absolute continuity} Let $X_t$ be a solution to the martingale problem (\ref{mprf}) - (\ref{qvmp}).  In this section, we show that for almost every $t\in[0,T]$, as an $M_F(\R^d)$-valued random variable, $X_t$ has a density almost surely. 

 For any $n \geq 1$, $f\in C_b^2(\R^{nd})$, and $\mu\in M_F(\R^d)$, we define
\[
\mu^{\otimes n}(f):=\int_{\R^{d}}\cdots\int_{\R^d}f(x_1,\dots,x_n)\mu(dx_1)\cdots \mu(dx_n).
\]
We derive the moment formula $\E(X^{\otimes n}_t(f))$ of the process $X$. In the one dimensional case, Skoulakis and Adler \cite{aap-01-skoulakis-adler} obtained the formula by computing the limit of particle approximations. An alternative approach by Xiong \cite{ap-04-xiong} consists in differentiating a conditional stochastic log-Laplace equation. In the present paper we use the techniques of moment duality to derive the moment formula. It can be also formulated by computing the limit of particle approximations.

For any integers $n\geq 2$ and $k\leq n$, we make use of the notation $x_k=(x_k^1,\dots, x_k^d)\in \R^d$ and $x=(x_1,\dots, x_n)\in\R^{nd}$. Let $\Phi_{ij}^{(n)}: C_b^2(\R^{nd}) \to C_b^2(\R^{nd})$, and $F^{(n)}, G^{(n)}: C_b^{2}(\R^{nd})\times M_F(\R^d)\to \R$ be given by
\[
(\Phi_{ij}^{(n)}f)(x_1,\dots,x_n):=\kappa(x_i,x_j)f(x_1,\dots,x_n),\quad i,j\in\{1,2,\dots, n\},
\]
\[
F^{(n)}(f, \mu):=\mu^{\otimes n}(f),
\]
and
\begin{align*}
G^{(n)}(f, \mu):=\mu^{\otimes n} (A^{(n)}f)+\frac{1}{2}\sum_{\mbox{\tiny$\begin{array}{c}
1\le i,j \le n\\
i\neq j\end{array}$}}\mu^{\otimes n} (\Phi^{(n)}_{ij}f),
\end{align*}
where $\kappa\in C_b^2(\R^{2d})$ is the correlation of the random field $\xi$ given by (\ref{crlt}), and $A^{(n)}$ is the generator of the $n$-particle-motion described by (\ref{pmern}). More precisely,
\begin{align*}
A^{(n)}f(x_1,\dots, x_n)=\frac{1}{2}(\Delta +B^{(n)})f(x_1,\dots, x_n),
\end{align*}
where $\Delta$ is the Laplace operator in $\R^{nd}$ and 
\begin{align*}
B^{(n)}f(x_1,\dots, x_n)=\sum_{k_1,k_2=1}^n\sum_{i,j=1}^d\rho^{ij}(x_{k_1}-x_{k_2})\frac{\partial^2 f}{\partial x_{k_1}^{i}\partial x_{k_2}^{j}}(x_1,\dots,x_n).
\end{align*}
\begin{lemma}\label{mxf}
Let $X_t$ be a solution to the martingale problem (\ref{mprf}) - (\ref{qvmp}). Then, for any $n\geq 1$ and $n\geq 2$, $f\in C_b^2(\R^{nd})$, the following process
\[
F^{(n)}(f, X_t)-\int_0^t G^{(n)}(f, X_s)ds
\]
is a martingale.
\end{lemma}

\begin{proof}
See Lemma 1.3.2 of Xiong \cite{ws-13-xiong}.
\end{proof}

Let $\{T_t^{(n)}\}_{t\geq 0}$ be the semigroup generated by $A^{(n)}$, that is, $T^{(n)}_t: C_b^2 (\R^{nd})\to C_b^2 (\R^{nd})$, given by
\[
T_t^{(n)}f(x_1,\dots, x_n)=\int_{\R^{nd}}p(t, (x_1,\dots, x_n), (y_1,\dots, y_n))f(y_1,\dots, y_n)dy_1\dots dy_n,
\]
where $p$ is the transition density of $n$-particle-motion. 

Let $\{S_k^{(n)}\}_{k\geq 1}$ be i.i.d. uniformly distributed random variables taking values in the set $\{\Phi_{ij}, 1\leq i,j\leq n, i\neq j\}$. Let $\{\tau_k\}_{k\geq 1}$ be i.i.d exponential random variables independent of $\{S_k^{(n)}\}_{k\geq 1}$, with rate $\lambda_n=\frac{1}{2}n(n-1)$. Let $\eta_0\equiv 0$, and $\eta_j=\sum_{i=1}^j\tau_i$ for all $j\geq 1$.  For any $f\in C_b^2(\R^{nd})$, we define a $C_b^2(\R^{nd})$-valued random process $Y^{(n)}=\{Y_t^{(n)}, 0\leq t\leq T\}$ as follows: for any $j\geq 0$ and $t\in[\eta_j, \eta_{j+1})$,
\begin{align}\label{dual}
Y^{(n)}_t:=T^{(n)}_{t-\eta_j}S^{(n)}_jT^{(n)}_{\tau_j}\cdots S^{(n)}_2T^{(n)}_{\tau_2}S^{(n)}_1T^{(n)}_{\tau_1}f.
\end{align}
Then, $Y^{(n)}$ is a Markov process with $Y^{(n)}_0=f$. It involves countable many i.i.d. jumps $S^{(n)}_k$, controlled by i.i.d. exponential clocks $\tau_k$. In between jumps, the process evolves deterministically by the continuous semigroup $T^{(n)}_t$. Notice that the exponential clock is memoryless, and the semigroup $T^{(n)}_t$ is generated by a time homogeneous Markov process. Therefore, $Y^{(n)}$ is also time homogeneous. 
\begin{lemma}\label{ingbty}
For any $f\in C_b^2(\R^{nd})$, let $Y^{(n)}_t$ be defined in (\ref{dual}). Then
\begin{align}\label{ifyny}
\E\Big(\sup_{x\in\R^{nd}}\big|Y^{(n)}_t(x)\big|\Big) \leq \|f\|_{\infty}\exp\left(\|\kappa\|_{\infty}\lambda_nt\right).
\end{align}
\end{lemma}
\begin{proof}
Since $T^{(n)}_t$ is the semigroup generated by a Markov process, for any $t>0$ and $f\in C_b^{2}(\R^{nd})$, $\|T^{(n)}_tf\|_{\infty}\leq \|f\|_{\infty}$. By definition of the jump operators $\{S^{(n)}_j\}_{j\geq 1}$, it is easy to see that $\|S^{(n)}_jf\|_{\infty}\leq \|\kappa\|_{\infty}\|f\|_{\infty}$. Thus we have
\begin{align}\label{ifyny1}
\E\Big(\sup_{x\in\R^{nd}}\big|Y^{(n)}_t(x)\big|\Big) \leq\|f\|_{\infty} \sum_{j=0}^{\infty}\big[\|\kappa\|_{\infty}^j\mathbb{P}(\eta_j<t)\big].
\end{align}
Notice that $\eta_j$ is the sum of i.i.d. exponential random variables. Then, we have
\begin{align}\label{ifyny2}
\mathbb{P}(\eta_j<t)=1-\exp\left(-\lambda_nt\right)\sum_{k=0}^{j-1}\frac{(\lambda_nt)^k}{k!}=\exp(\lambda_n(t'-t))\frac{(\lambda_nt)^j}{j!},
\end{align}
for some $t'\in (0,t)$. Therefore, (\ref{ifyny}) follows from (\ref{ifyny1}) and (\ref{ifyny2}).
\end{proof}
Let $H^{(n)}:C_b^2(\R^{nd})\times M_F(R^d)\to \R$ be given by
\[
H^{(n)}(f,\mu):=G^{(n)}(f,\mu)-\lambda_nF^{(n)}(f,\mu).
\]
\begin{lemma}\label{mxdual}
Let $\mu\in M_F(\R^d)$, Then, the process
\begin{align}\label{mrgdual}
F^{(n)}(Y^{(n)}_t, \mu)-\int_0^t H^{(n)}(Y^{(n)}_s, \mu)ds
\end{align}
is a martingale.
\end{lemma}
\begin{proof}
Let $\mu^{(n)}$ be any finite measure on $\R^{nd}$. Then, we have
\begin{align}\label{eexclc}
&\E \big(\mu^{(n)}(Y^{(n)}_t)\big)=\E\big(\mu^{(n)}(Y^{(n)}_t)\1_{\{\tau_1>t\}}\big)+\E\big(\mu^{(n)}(Y^{(n)}_t)\1_{\{\eta_1\leq t<\eta_2\}}\big)+o(t).
\end{align}
For the first term, we have
\begin{align}\label{eexclc0}
\E\big(\mu^{(n)}(Y^{(n)}_t)\1_{\{\tau_1>t\}}\big)=&\mu^{(n)}(T^{(n)}_tf) \mathbb{P}(\tau_1>t)=\mu^{(n)}(T^{(n)}_tf) \exp(-\lambda_nt)).
\end{align}
For the second term, since $\tau_1$, $\tau_2$ are independent, then for any $0\leq s\leq t$, we have
\begin{align}\label{p2iide}
\mathbb{P}(\tau_1+\tau_2>t,\tau_1\leq s)=\int_0^s\int_{t-s_1}^{\infty}\lambda_n^2\exp(-\lambda_n(s_1+s_2))ds_2ds_1=\lambda_nse^{-\lambda_nt}.
\end{align}
Note that by Lemma \ref{ingbty}, $|Y^{(n)}_{\cdot}|$ is integrable on $[0,T]\times \R^{nd}\times \Omega$ with respect to the product measure $dt \times \mu^{(n)}(dx)\times P(d\omega)$. Then, by (\ref{p2iide}), Fubini's theorem, and the mean value theorem, we have
\begin{align}\label{eexclc1}
&\E\big(\mu^{(n)}(Y^{(n)}_t)\1_{\{\eta_1\leq t<\eta_2\}}\big)\nonumber\\
=&\frac{1}{2}\sum_{\mbox{\tiny$\begin{array}{c}
1\le i,j \le n\\
i\neq j\end{array}$}}\int_0^t\int_{\R^{nd}}\big(T^{(n)}_{t-s}\Phi^{(n)}_{ij}T^{(n)}_sf\big)(x)\exp\left(-\lambda_nt\right)\mu^{(n)}(dx)ds\nonumber\\
=&\frac{t}{2}\exp\left(-\lambda_nt\right)\sum_{\mbox{\tiny$\begin{array}{c}
1\le i,j \le n\\
i\neq j\end{array}$}}\int_{\R^{nd}}\big(T^{(n)}_{t-t'}\Phi_{ij}^{(n)}T^{(n)}_{t'}f\big)(x)\mu^{(n)}(dx),
\end{align}
for some $t'\in(0,t)$. Combining (\ref{eexclc}), (\ref{eexclc0}), and (\ref{eexclc1}), we have
\[
\lim_{t\downarrow 0}\frac{\E\big( \mu^{(n)}(Y^{(n)}_t)\big)-\mu^{(n)}(f)}{t}=\mu^{(n)}(A^{(n)}f)+\frac{1}{2}\sum_{\mbox{\tiny$\begin{array}{c}
1\le i,j \le n\\
i\neq j\end{array}$}}\mu^{(n)}\big(\Phi^{(n)}_{ij}f-f\big).
\]
By Proposition 4.1.7 of Ethier and Kurtz \cite{wiley-86-ethier-kurtz}, the following process:
\begin{align}\label{mtgyg}
\mu^{(n)}(Y^{(n)}_t)-\int_0^t \Big[\mu^{(n)}(A^{(n)} Y_s^{(n)})+\frac{1}{2}\sum_{\mbox{\tiny$\begin{array}{c}
1\le i,j \le n\\
i\neq j\end{array}$}}\mu^{(n)}(\Phi^{(n)}_{ij}Y_s^{(n)}-Y_s^{(n)})\Big]ds,
\end{align}
is a martingale. Then, the lemma follows by choosing $\mu^{(n)}=\mu^{\otimes n}$.
\end{proof}

By Lemma \ref{mxf}, \ref{mxdual} and Corollary 3.2 of Dawson and Kurtz \cite{springer-82-dawson-kurtz}, we have the following moment identity:
\begin{align}\label{mmtidt}
\E \big(X_t^{\otimes n}(f)\big)=\E\bigg[X^{\otimes n}_0(Y^{(n)}_t)\exp\Big(\int_0^t\lambda_nds\Big)\bigg]=\exp\Big(\frac{1}{2}n(n-1)t\Big)\E\big(X^{\otimes n}_0(Y^{(n)}_t)\big).
\end{align}
\begin{lemma}\label{lmfif}
Fix $f\in C_b^2(\R^{nd})$.
\begin{enumerate}[(i)] 
\item The following PDE
\begin{align}\label{dualpde}
\partial_tv_t(x)=A^{(n)}v_t(x)+\frac{1}{2}\sum_{\mbox{\tiny$\begin{array}{c}
1\le i,j \le n\\
i\neq j\end{array}$}}\kappa(x_i,x_j)v(t,x),
\end{align}
with initial value $v_0(x)=f(x)$, has a unique solution.
\item Let $X=\{X_t,t\in[0,T]\}$ be a solution to the martingale problem (\ref{mprf}) - (\ref{qvmp}). Then,
 \begin{align}\label{dualid}
\E \big(X_t^{\otimes n}(f)\big)=X_0^{\otimes n}(v_t).
\end{align}
\end{enumerate}
\end{lemma}
\begin{proof}
Firstly, we claim that the operator $A^{(n)}=\frac{1}{2}(\Delta+B^{(n)})$ is uniformly parabolic in the sense of Friedman (see Section 1.1 of \cite{courier-13-friedman}). Because $\Delta$ is uniformly parabolic, then it suffices to analyse the properties of $B^{(n)}$. For any $k=1,\dots, n$, $i=1,\dots,  d$, and $\xi_k^i\in\R$, let $\xi_k=(\xi_k^1,\dots,\xi_k^d)$. Then, we have
\begin{align*}
\sum_{k_1,k_2=1}^n\sum_{i,j=1}^d\rho^{ij}(x_{k_1}-x_{k_2})\xi_{k_1}^{i}\xi_{k_2}^{j}=\int_{\R^d}\bigg|\sum_{k=1}^nh^*(z-x_k)\xi_k\bigg|^2dz\geq 0.
\end{align*}
Thus $B^{(n)}$ is parabolic. On the other hand, by Jensen's inequality, we have
\[
\sum_{k_1,k_2=1}^n\sum_{i,j=1}^d\rho^{ij}(x_{k_1}-x_{k_2})\xi_{k_1}^{i}\xi_{k_2}^{j}=\int_{\R^d}\bigg|\sum_{k=1}^nh^*(z-x_k)\xi_k\bigg|^2dz\leq n\|\rho\|_{\infty}\sum_{k=1}^n|\xi_{k}|^2.
\]
It follows that $A^{(n)}=\frac{1}{2}(\Delta+B^{(n)})$ is uniformly parabolic.

Since $h\in H_2^3(\R^d;\R^d\otimes \R^d)$, $\rho(x-y)=\int_{\R^d}h(z-x)h^*(z-y)dz$ has bounded derivatives up to order three, then by Theorem 1.12 and 1.16 of Friedman \cite{courier-13-friedman}, the PDE (\ref{dualpde}) has a unique solution. 

In order to show (ii), let 
\[
\widetilde{v}_t(x)=\E \big(Y^{(n)}_t(x)\big),
\]
where $Y^{(n)}$ is defined by (\ref{dual}). By the same argument as we did in the proof of Lemma \ref{ingbty}, we can show that for any $t\in [0,T]$ and $x\in\R^{nd}$
\[
\E \Big(\sup_{x\in\R^d}\big|A^{(n)} Y_t^{(n)}(x)\big|\Big)< \infty.
\]
Then, by the dominated convergence theorem, we have 
\[
\E\big( A^{(n)} Y_t^{(n)}(x)\big)=A^{(n)} \E (Y_t^{(n)}(x)).
\]
Let $\mu^{(n)}$ be any finite measure on $\R^{nd}$. Recall that the process defined by  (\ref{mtgyg}) is a martingale, then the following equality follows from Fubini's theorem:
\begin{align*}
\mu^{(n)}(\widetilde{v}_t)=&\E\big(\mu^{(n)}(Y^{(n)}_t)\big)=\mu^{(n)}(f)+\int_0^t \big\langle \mu^{(n)}, \E \big(A^{(n)} Y_s^{(n)}\big)\big\rangle ds\\
&\hspace{25mm}+\frac{1}{2}\sum_{\mbox{\tiny$\begin{array}{c}
1\le i,j \le n\\
i\neq j\end{array}$}}\int_0^t\big\langle\mu^{(n)}, [k(\cdot_i,\cdot_j)-1]\E (Y_s^{(n)})\big\rangle ds\\
=&\mu^{(n)}(f)+\int_0^t \big\langle \mu^{(n)}, A^{(n)} \widetilde{v}_s\big\rangle ds+\frac{1}{2}\sum_{\mbox{\tiny$\begin{array}{c}
1\le i,j \le n\\
i\neq j\end{array}$}}\int_0^t\big\langle \mu^{(n)}, [k(\cdot_i,\cdot_j)-1]\widetilde{v}_s\big\rangle ds.
\end{align*}
In other words,
\[
\bigg\langle \mu^{(n)}, \widetilde{v}_t-f-\int_0^t \Big[ A^{(n)} \widetilde{v}_s -\frac{1}{2}\sum_{\mbox{\tiny$\begin{array}{c}
1\le i,j \le n\\
i\neq j\end{array}$}} (k(\cdot_i,\cdot_j)-1)\widetilde{v}_s \Big]ds\bigg\rangle=0,
\]
for all $\mu^{(n)}\in M_F(\R^{nd})$. It follows that $\widetilde{v}=\{\widetilde{v}_t(x),t\in[0,T],x\in\R^d\}$ solves the following PDE 
\begin{align}\label{dualst}
\partial_t\widetilde{v}_t(x)=A^{(n)}\widetilde{v}_t(x)+\frac{1}{2}\sum_{\mbox{\tiny$\begin{array}{c}
1\le i,j \le n\\
i\neq j\end{array}$}}[\kappa(x_i,x_j)-1]\widetilde{v}_t(x),
\end{align}
with the initial value $\widetilde{v}_0(x)=f(x)$. This solution is unique by the same argument as in part (i). Observe that 
\begin{align}\label{pdead}
v_t(x)=\widetilde{v}_t(x)\exp\Big(\frac{1}{2}n(n-1)t\Big).
\end{align}
Therefore, (\ref{dualid}) follows from (\ref{pdead}) and the moment duality (\ref{mmtidt}).
\end{proof}

In Lemma \ref{lmfif}, we derived the moment formula for $\E \big(X_t^{(n)}(f)\big)$ in the case when $n\geq 2$. If $n=1$, the dual process only involves a deterministic evolution driven by the semigroup of one particle motion, which makes things much simpler. We write the formula below and skip the proof. Let $p(t,x,y)$ be the transition density of the one particle motion, then for any $\phi\in C_b^2(\R^d)$,
\[
\E (X_t(\phi))=X_0(T^{(1)}_t\phi)=\int_{\R^d}\int_{\R^d}p(t,x,y)\phi(y)dyX_0(dx).
\]

The existence of the density of $X_t$ will be derived following Wang's idea (see Theorem 2.1 of \cite{ptrf-97-wang}). For any $\epsilon>0$, $x\in \R^d$, let $p_{\epsilon}$ be the heat kernel on $\R^d$, that is
\[
p_{\epsilon}(x)=(2\pi \epsilon)^{-\frac{d}{2}}\exp\Big(-\frac{|x|^2}{2\epsilon}\Big).
\]
\begin{lemma}\label{lcsprt}
Let $X=\{X_t,t\in[0,T]\}$ be a solution to the martingale problem (\ref{mprf}) - (\ref{qvmp}). Assume that the initial measure $X_0\in M_F(\R^d)$ has a bounded density $\mu$. Then,
\begin{align}\label{bdprt}
\int_0^T\int_{\R^d}\E \big(\big|X_t(p_{\epsilon}(x-\cdot))\big|^2\big)dxdt< \infty,
\end{align}
and
\begin{align}\label{csprt}
\lim_{\epsilon_1,\epsilon_2\downarrow 0}\int_0^T\int_{\R^d}\E\big(\big|X_t(p_{\epsilon_1}(x-\cdot))-X_t(p_{\epsilon_2}(x-\cdot))\big|^2\big)dxdt=0.
\end{align}
\end{lemma}
\begin{proof}
Let $\Gamma (t,(y_1,y_2); r, (z_1,z_2))$ be the fundamental solution to the PDE (\ref{dualpde}) when $n=2$. We write $y=(y_1,y_2)$ and $ z=(z_1,z_2)\in\R^{2d}$. Then, for $f\in C_b^2(\R^{2d})$,
\[
v(t,y)=\int_{\R^{2d}}\Gamma(t,y;0,z)f(z)dz,
\]
is the unique solution to the PDE (\ref{dualpde}) with initial condition $v_0=f$. Thus by Lemma \ref{lmfif}, we have
\begin{align}\label{bdps}
&\E \big[X_t(p_{\epsilon_1}(x-\cdot))X_t(p_{\epsilon_2}(x-\cdot))\big]\nonumber\\
&\qquad=\int_{\R^{2d}}\int_{\R^{2d}}\Gamma (t,y;0,z) p_{\epsilon_1}(x-z_1)p_{\epsilon_2}(x-z_2)dz X_0^{\otimes 2}(dy).
\end{align}
By the inequality (6.12) of Friedman \cite{courier-13-friedman} on page 24, we know that there exists $C_{\Gamma},\lambda>0$, such that for any $0\leq r<t\leq T$,
\begin{align}\label{eubfs}
|\Gamma(t,y;r;z)|\leq C_{\Gamma} p_{\frac{t-r}{\lambda}}(y_1-z_1) p_{\frac{t-r}{\lambda}}(y_2-z_2).
\end{align}
Therefore, by the semigroup property of heat kernels and Fubini's theorem, we have
\begin{align}\label{410est1}
&\int_0^T\int_{\R^d} \E\big[X_t(p_{\epsilon_1}(x-\cdot))X_t(p_{\epsilon_2}(x-\cdot))\big]dxdt\nonumber\\
&\qquad=\int_0^T\int_{\R^{2d}}\int_{\R^{2d}}\Gamma (t,y;0,z) p_{\epsilon_1+\epsilon_2}(z_1-z_2)dz X_0^{\otimes 2}(dy)dt
\end{align}
From (\ref{eubfs}), (\ref{410est1}) and the fact that $X_0\in M_F(\R^d)$ has a bounded density $\mu$, it follows that (\ref{bdprt}) is true.

Let $\mathcal{M}$ be the function on $\R^{2d}$, given by
\[
\mathcal{M}(z)= \int_0^T\int_{\R^{2d}}\Gamma(t,y;0,z)X_0^{\otimes 2}(dy)dt.
\]
Notice that fix $0\leq r<t\leq T$, $\Gamma(t,y;r,x)$ is uniformly continuous in the spatial argument (see (6.13) of Friedman \cite{courier-13-friedman} on page 24). As a consequence $\mathcal{M}$ is continuous.  Therefore, by (\ref{eubfs}) and the continuity of $\mathcal{M}$, the function $\mathcal{N}$ on $\R^d$ given by
\[
\mathcal{N}(x):=\int_{\R^d}\mathcal{M}(z_1, z_1-x)dz_1, 
\]
is integrable  and continuous everywhere. It follows that
\begin{align}\label{410est2}
&\lim_{\epsilon_1,\epsilon_2\to 0}\int_0^t\int_{\R^d}\E \big[X_t(p_{\epsilon_1}(x-\cdot))X_t(p_{\epsilon_2}(x-\cdot))\big]dxdt\nonumber\\
&\qquad=\lim_{\epsilon_1,\epsilon_2\to 0}\int_{\R^{2d}}\mathcal{M}(z)p_{\epsilon_1+\epsilon_2}(z_1-z_2)dz\nonumber\\
&\qquad=\lim_{\epsilon_1,\epsilon_2\to 0}\int_{\R^d}\mathcal{N}(y)p_{\epsilon_1+\epsilon_2}(y)dy\nonumber\\
&\qquad=\mathcal{N}(0)= \int_0^T\int_{\R^d}\int_{\R^{2d}}\Gamma(t,y;0,(x,x))X_0^{\otimes 2}(dy)dxdt
\end{align}
Therefore, (\ref{csprt}) is a consequence of (\ref{410est2}).
\end{proof}

\begin{proposition}\label{extdst}
Let $X=\{X_t,t\in[0,T]\}$ be a solution to the martingale problem (\ref{mprf}) - (\ref{qvmp}). Assume that the initial measure $X_0\in M_F(\R^d)$ has a bounded density $\mu$.  Then, for almost every $t\in (0,T]$, $X_t$ is absolutely continuous with respect to the Lebesgue measure almost surely.
\end{proposition}

\begin{proof}
As proved in Lemma \ref{lcsprt}, for any $x\in \R^d$ and $\epsilon_n\downarrow 0$, the sequence $\{X_t(p_{\epsilon_n}^x)\}_{n\geq 1}$ is Cauchy in $L^2(\Omega \times \R^d \times [0,T])$. Then, it converges to some square integrable random field. By the same argument as in Theorem 2.1 of Wang \cite{ptrf-97-wang}, we can show that the limit random field is the density of $X_t$ almost surely.
\end{proof}

\textbf {Remark:} The assumption in Proposition \ref{extdst}, that the initial measure has a bounded density, cannot be removed. Actually, if we choose $X_0=\delta_0$, the Dirac delta mass at $0$, then $\int_0^T\int_{\R^d}\Gamma (t,0;0,(x,x)) dxdt$ behaves like $\int_0^Tt^{-\frac{d}{2}}dt$, which is finite only if $d=1$. This is another difference from the one dimensional situation, in which case $X_0(1)<\infty$ is enough to prove the existence of the density (see Theorem 2.1 Wang \cite{ptrf-97-wang}).

\subsection{Proof of Theorems \ref{unique} and \ref{tmpd}}
The proof of Theorems \ref{unique} and \ref{tmpd} is based on the equivalence of the martingale problem (\ref{mprf}) - (\ref{qvmp}) and the SPDE (\ref{dqmvp}). 

The equivalence between  martingale problems  and SDEs in finite dimensions was observed in the 1970s (see Stroock and Varachan \cite{psbs-72-stroock-varadhan}). An alternative  proof given by Kurtz \cite{sa-11-kurtz} consists of  the ``Markov mapping theorem''. In a recent paper  \cite{arxiv-18-biswas-etheridge-klimek}    Biswas et al.     generalized this result to the infinite dimensional cases with one noise   following Kurtz's idea. Here in the present paper, we establish a similar result with two noises by using the martingale representation theorem.

\begin{proposition}\label{propspde}
Let $\mu\in C_b(\R^d)\cap L^1(\R^d)$ be a nonnegative function on $\R^d$. Then, $u=\{u_t, t\in[0,T]\}$ is the density of a solution of the martingale problem (\ref{mprf}) - (\ref{qvmp}) with initial density $\mu$, if and only if $u$ is a weak solution to the SPDE (\ref{dqmvp}).
\end{proposition}

\begin{proof}
If $u$ is a weak solution to (\ref{dqmvp}), then, as a consequence of It\^{o}'s formula, $u$ is the density of a measure-valued process that solves the martingale problem  (\ref{mprf}) - (\ref{qvmp}). It suffices to show the converse statement.

Let $X=\{X_t, t\in[0,T]\}$ be a solution to the martingale problem (\ref{mprf}) - (\ref{qvmp}) with initial density $\mu$. Then, by Proposition \ref{extdst}, for almost every $t\in[0,T]$, $X_t$ has a density almost surely. We denote by $u_t$ the density of $X_t$.

Consider $M=\{M_t, t\in[0,T]\}$ defined by (\ref{mprf}) as an $\mathscr{S}'$-martingale (see Definition 2.1.2 of Kallianpur and Xiong \cite{ims-95-kallianpur-xiong}). Then, by Theorem 3.1.4 of \cite{ims-95-kallianpur-xiong}, there exists a Hilbert space $\mathcal{H}^*\supset L^2(\R^d)$, such that $M$ is an $\mathcal{H}^*$-valued martingale. Denote by $\mathcal{H}$ the dual space of $\mathcal{H}^*$. 

Let $\mathfrak{H}_1=L^2(\R^d;\R^d)$, and let $\mathfrak{H}_2$ be the completion of $\mathscr{S}$ with the inner product
\[
\langle \phi,\varphi\rangle_{\mathfrak{H}_1}:=\int_{\R^d\times R^d}\kappa(x,y)\phi(x)\varphi(y)dxdy.
\]
Consider the product space $\mathfrak{H}=\mathfrak{H}_1\times \mathfrak{H}_2$. Then, $\mathfrak{H}$ is a Hilbert space equipped with the inner product
\[
\big\langle (\phi_1,\phi_2),(\varphi_1,\varphi_2)\big\rangle_{\mathfrak{H}}:=\langle \phi_1, \varphi_1\rangle_{\mathfrak{H}_1}+\langle \phi_2,\varphi_2\rangle_{\mathfrak{H}_2}.
\]
For any $t\in [0,T]$, let $\Psi_t:\mathcal{H}\to \mathfrak{H}$ be given by $\Psi_t(\phi)(x,y)=\big(\Psi^1_t(\phi)(x),\Psi^2_t(\phi)(y)\Big)$, where
\[
\Psi^1_t(\phi)(x):=\int_{\R^d}\nabla \phi(y)^*h(x-y)u_t(y)dy,
\]
and
\[
\Psi^2_t(\phi)(x):=\phi(x)u_t(x).
\]
Then, for any $\phi,\varphi\in \mathcal{H}$, we have
\begin{align*}
\langle M(\phi), M(\varphi)\rangle_t=&\int_0^t\nabla \phi(x)^*\rho(x-y)\nabla \varphi(y)X_s(dx)X_s(dy)ds\\
&+\int_0^t\int_{\mathbb{R}^d\times\mathbb{R}^d}\kappa(x,y)\phi(x)\phi(y)X_s(dx)X_s(dy)ds\\
=&\int_0^t \langle \Phi_s (\phi), \Phi_s (\varphi)\rangle_{\mathfrak{H}} ds,
\end{align*}
Therefore, by the martingale representation theorem (see e.g. Theorem 3.3.5 of Kallianpur and Xiong \cite{ims-95-kallianpur-xiong}), there exists a $\mathfrak{H}$-cylindrical Brownian motion $\mathfrak{B}=\{\mathfrak{B}_t, 0\leq t\leq T\}$, such that
\begin{align*}
M_t(\phi)=&\int_0^t\big\langle\Psi_s(\phi), d\mathfrak{B}_s\big\rangle_{\mathfrak{H}}.
\end{align*}
Let $\mathfrak{B}^1=\{\mathfrak{B}^1_t(\phi),t\in [0,T],\phi\in \mathfrak{H}_1\}$ and $\mathfrak{B}^2=\{\mathfrak{B}^2_t(\varphi),t\in [0,T],\varphi\in\mathfrak{H}_2\}$  be given  by
\[
\mathfrak{B}^1_t(\phi)=\mathfrak{B}_t(\phi,0)\ \mathrm{and}\ \mathfrak{B}^2_t(\varphi)=\mathfrak{B}_t(0, \varphi).
\]
Then, $\mathfrak{B}^1$ and $\mathfrak{B}^2$ are $\mathfrak{H}^1$- and $\mathfrak{H}^2$-cylindrical Brownian motion respectively, and they are independent. As a consequence, we have
\begin{align}\label{ihs}
M_t(\phi)=\int_0^t\Big\langle\int_{\R^d}\nabla\phi(z)^*h(\cdot-z)X_s(dz), d\mathfrak{B}^1_s\Big\rangle_{\mathfrak{H}_1}
+\int_0^t\big\langle \phi(z)u_s(z), d\mathfrak{B}^2_s\big\rangle_{\mathfrak{H}_2}.
\end{align}
Let $\{e_j\}_{j\geq 1}$ be a complete orthonormal 
 basis of $\mathfrak{H}_2$. Then, by Theorem 3.2.5 of \cite{ims-95-kallianpur-xiong}, $V=\{V_t,t\in[0,T]\}$, defined by
\[
V_t:=\sum_{j=1}^{\infty}\mathfrak{B}^2_t(e_j)e_j,
\]
is a $\mathscr{S}'$-valued Wiener process with covariance
\[
\E\big[V_s(\phi)V_t(\varphi)\big]=s\wedge t\int_{\R^d\times \R^d}\kappa(x,y)\phi(x)\varphi(y)dxdy,
\]
for any $\phi,\varphi\in \mathscr{S}$. Therefore, by (\ref{ihs}) and the equivalence of stochastic integrals with Hilbert space valued Brownian motion and Walsh's integrals (see e.g. Proposition 2.6 of Dalang and Quer-Sardanyons \cite{em-11-dalang-sardanyons} for spatial homogeneous noises), $u$ is a weak solution to the SPDE (\ref{dqmvp}).
\end{proof}

\begin{proof}[Proof of Theorem \ref{unique}]
By Propositions \ref{propmp} and \ref{propspde}, the SPDE (\ref{dqmvp}) has a weak solution, that can be obtained by the branching particle approximation. Therefore, by Yamada-Watanabe argument (see Yamada and Watanabe \cite{jmku-71-yamada-watanabe} and Kurtz \cite{ejp-07-kurtz}), it suffices to show the pathwise uniqueness of the equation.  Assume that $u$, $\widetilde{u}$ be two strong solutions to the SPDE (\ref{dqmvp}). Let $d=u-\widetilde{u}$. Then, $d=\{d_t(x), t\in[0,T], x\in \R^d\}$ is a solution to (\ref{dqmvp}), with initial condition $\mu\equiv 0$. Thus $d$ is also the density of a solution to the martingale problem (\ref{mprf}) - (\ref{qvmp}), with initial measure $X_0\equiv 0$. By the moment duality (\ref{mmtidt}), for any $\phi\in C_b^2(\R^d)$, we have
\begin{align*}
\E \langle d_t, \phi\rangle^2=\exp(t)\E \big(X_0(Y^{(2)}_t)\big)\equiv 0,
\end{align*}
where $Y^{(2)}$ is the dual process defined by (\ref{dual}) in the case when $n=2$. If follows that $u=\widetilde{u}$ almost surely.
\end{proof}

\begin{proof}[Proof of Theorem \ref{tmpd}]
As a consequence of Yamada-Watanabe's argument, the weak solution to the SPDE (\ref{dqmvp}) is unique in distribution. Assume hypotheses \textbf{[H1]} and \textbf{[H2]}. It follows that  every weakly convergent subsequence of $\{X^n\}_{n\geq 1}$ converges to the same limit in $D([0,T]; M_F)$ in law. The limit has a density almost surely, that is a weak solution the SPDE (\ref{dqmvp}).
\end{proof}

{\bf Remark:} Assume that the initial measure has a bounded density. According to Theorem \ref{unique}, \ref{tmpd},  Proposition \ref{extdst} and \ref{ihs}, the martingale problem (\ref{mprf}) - (\ref{qvmp}) has a unique solution in distribution. If we allow the solution to the SPDE (\ref{dqmvp}) to be a distribution-valued process, the existence and uniqueness of (\ref{dqmvp}) are still true in the case when the initial value is only finite. This implies the uniqueness of the martingale problem (\ref{mprf}) - (\ref{qvmp}) can be generalized to the situation when the initial measure is only finite.

\section{Moment estimates for one-particle motion}

In this section, we focus on the one-particle motion without branching. By using the techniques of  Malliavin calculus, we will obtain moment estimates for the transition probability density of the particle motion conditional on the environment $W$. A brief introduction and several theorems on Malliavin calculus are stated in Appendix A. For a detailed account on this topic, we refer the readers to the book of Nualart \cite{springer-06-nualart}.

Fix a  time interval $[0, T]$. Let $B=\{B_t,0\leq t\leq T\}$ be a standard $d$-dimensional Brownian motion and let $W$ be a $d$-dimensional space-time white Gaussian random field on $[0,T]\times \mathbb{R}^d$ that is  independent of $B$. Assume that $h\in H_2^3(\R^d;\R^d\otimes \R^d)$. For any $0\leq r<t\leq T$, we denote by $\xi_t=\xi_t^{r,x}$, the path of one-particle motion, with initial position $\xi_r=x$. It satisfies the SDE:
\begin{align}\label{sde}
\xi_t=x+B_t-B_r+\int_r^t\int_{\mathbb{R}^d}h(y-\xi_u)W(du, dy).
\end{align}

We will apply the Malliavin calculus on $\xi_t$ with respect to the Brownian motion $B$. Let $H=L^2([0,T];\R^d)$ be the associated Hilbert space. By the Picard iteration scheme (see e.g. Theorem 2.2.1 of Nualart \cite{springer-06-nualart}), we can prove that for any $t\in(r, T]$, $\xi_t\in\cap_{p\geq 1}\mathbb{D}^{3,p}(\R^d)$. Particularly, $D\xi_t$ satisfies the following system of SDEs: 
\begin{align}\label{sdexi1}
D^{(k)}_{\theta}\xi_t^i=\delta_{ik}-\sum_{j_1,j_2=1}^d\int_{\theta}^t\int_{\mathbb{R}^d} \partial_{j_1} h^{ij_2}(y-\xi_s) D^{(k)}_{\theta}\xi_s^{j_1} W^{j_2}(ds, dy),\quad 1\leq i,k\leq d,
\end{align}
for any $\theta\in [r,t]$, and $D^{(k)}_{\theta}\xi_t^i=0$ for all $\theta>t$. 

In order to simplify the expressions, we rewrite the stochastic integrals in (\ref{sdexi1}) as integrals with respect to martingales. To this end, let $M=\{M_t, r\leq t\leq T\}$ be the $d\times d$ matrix-valued process given by
\begin{align*}
M_t=\sum_{k=1}^d\int_r^t\int_{\mathbb{R}^d}g_k(s,y) W^k(ds, dy),
\end{align*}
where $g_k:\Omega\times [r,T]\times \R^d\to \R^d\otimes \R^d$ is given by
\[
g_k^{ij}(t,y)=\partial_i h^{jk}(y-\xi_t),\quad 1\leq i,j,k\leq d.
\]
Notice that $M_t$ is the sum of stochastic integrals, so it is a matrix-valued martingale. The quadratic covariations of $\{M^{ij}\}_{i,j=1}^d$ are bounded and deterministic:
\begin{align}\label{rfqvm}
&\left\langle M^{i_1 j_1}, M^{i_2 j_2}\right\rangle_t=\sum_{k=1}^d\int_r^t\int_{\mathbb{R}^d}\partial_{i_1} h^{j_1k}(y-\xi_s)\partial_{i_2} h^{j_2k}(y-\xi_s)dyds\\
&\qquad=(t-r)\sum_{k=1}^d\int_{\mathbb{R}^d}\partial_{i_1} h^{j_1k}(y)\partial_{i_2} h^{j_2k}(y)dy:=Q^{i_1,j_1}_{i_2,j_2}(t-r)\leq \|h\|_{3,2}(t-r).\nonumber
\end{align}
Now the equation (\ref{sdexi1}) can be rewritten as follows:
\begin{align}\label{sdexi1m}
D^{(k)}_{\theta}\xi_t^i=\delta_{ik}-\sum_{j=1}^d\int_{\theta}^t\int_{\mathbb{R}^d} D^{(k)}_{\theta}\xi_s^{j} dM^{ji}_s,\quad 1\leq i,k\leq d.
\end{align}

\begin{lemma}
For any $0\leq r<t\leq T$, $x\in\R^d$, let $\gamma_t=\gamma_{\xi_t}$ be the Malliavin matrix of $\xi_t=\xi_t^{r,x}$, then $\gamma_t$ is nondegenerate almost surely.
\end{lemma}
\begin{proof}
We prove the lemma following Stroock's idea (see Chapter 8 of Stroock \cite{lnm-83-stroock}).  Let $\lambda_{\theta}(t)$ be the $d\times d$ symmetric random matrix given by
\begin{align*}
\lambda_{\theta}^{ij}(t)=\sum_{k=1}^d D_{\theta}^{(k)} \xi_t^i D_{\theta}^{(k)} \xi_t^j.
\end{align*}
Then, the Malliavin matrix of $\xi_t$ is the integral of $\lambda_{\theta}(t)$:
\[
\gamma_t=\int_r^t\lambda_{\theta}(t)d\theta.
\]
By (\ref{sdexi1}), (\ref{rfqvm}) and It\^{o}'s formula, we have
\begin{align*}
D^{(k)}_{\theta}\xi_t^iD^{(k)}_{\theta}\xi_t^j=&\delta_{ik}\delta_{kj}-\sum_{k_1=1}^d\int_{\theta}^t D^{(k)}_{\theta}\xi_s^i D^{(k)}_{\theta}\xi_s^{k_1}dM_s^{k_1j}-\sum_{k_2=1}^d\int_{\theta}^tD^{(k)}_{\theta}\xi_s^j D^{(k)}_{\theta}\xi_s^{k_2} dM_s^{k_2i}\nonumber\\
&+\sum_{k_1,k_2=1}^dQ^{k_1,j}_{k_2, i}\int_{\theta}^t D_{\theta}^{(k)}\xi_s^{k_1} D_{\theta}^{(k)}\xi_s^{k_2}ds.
\end{align*}
Therefore,
\begin{align}\label{lambda}
\lambda_{\theta}(t)=&I-\int_{\theta}^t \lambda_{\theta}(s)  dM_s-\int_{\theta}^t dM_s^*\cdot\lambda_{\theta}(s)\nonumber\\
&+\sum_{k=1}^d \int_{\theta}^t\int_{\mathbb{R}^d} g_k^*(s,y)\lambda_{\theta}(s) g_k(s,y) dy d s.
\end{align}
For any $\theta\in [r,t]$, we claim that $\lambda_{\theta}(t)$ is invertible almost surely, and its inverse $\beta_{\theta}(t)$ satisfies the following SDE:
\begin{align}\label{ivlambda}
\beta_{\theta}(t)=&I+\int_{\theta}^t \beta_{\theta}(s) dM_s^*+\int_{\theta}^t d M_s\cdot \beta_{\theta}(s)\\
&+\sum_{k=1}^d\int_{\theta}^t \int_{\mathbb{R}^n} \left(g_k(s,y)^2\beta_{\theta}(s)+g_k(s,y)\beta_{\theta}(s) g_k^*(s,y)+\beta_{\theta}(s)g_k^*(s,y)^2\right)dyds\nonumber.
\end{align}
Indeed, by It\^{o}'s formula, we have
\begin{align}\label{sdeplambda}
d[\lambda_{\theta}(t) \beta_{\theta}(t)]&=-d M^*_t\cdot [\lambda_{\theta}(t) \beta_{\theta}(t)]+[\lambda_{\theta}(t) \beta_{\theta}(t)] d M^*_t\\
+\sum_{k=1}^d &\Big(\int_{\mathbb{R}^d}\big([\lambda_{\theta}(t) \beta_{\theta}(t)] g^*_k(t,y)^2-g_k^*(t,y)[\lambda_{\theta}(t) \beta_{\theta}(t)] g_k^*(t,y)\big)dy\Big) dt\nonumber.
\end{align}
Note that $\lambda_{\theta}(t)\beta_{\theta}(t)\equiv I$ solves the SDE (\ref{sdeplambda}) with initial value $\lambda_{\theta}(\theta)\beta_{\theta}(\theta)=I$. Therefore, the strong uniqueness of the linear SDE (\ref{sdeplambda}) implies that $\lambda^{-1}_{\theta}(t)=\beta_{\theta}(t)$ almost surely.

Denote by $\|\cdot\|_2$ the Hilbert-Schmidt norm of matrices. By Jensen's inequality (see Lemma 8.14 of Stroock \cite{lnm-83-stroock}),   the following inequality holds almost surely
\begin{align}\label{hsniigamma}
\left\|\gamma^{-1}_t\right\|_2=\bigg\|\Big(\int_r^t \lambda_{\theta}(t)d\theta\Big)^{-1}\bigg\|_2\leq \frac{1}{(t-r)^2}\Big\|\int_r^t \beta_{\theta}(t)d \theta\Big\|_2.
\end{align}
It is easy to show that $\displaystyle\sup_{\theta\in[r,t]}\big\|\|\beta_{\theta}(t)\|_2\big\|_{2p}<\infty$ for all $p\geq 1$. Therefore, the right-hand side of (\ref{hsniigamma}) is finite a.s., and thus $\gamma_t$ is nondegenerate almost surely.
\end{proof}

We denote by $\sigma_t=\gamma^{-1}_t$ the inverse of the Malliavin matrix of $\xi_t$. In the following lemma, we obtain some moment estimates for the derivatives of $\xi_t$ and $\sigma_t$. Before estimates, we introduce the following generalized Cauchy-Schwarz's inequality.

\begin{lemma}\label{ttpcsi}
Let $n_1,n_2$ be nonnegative integers, $u_1\in L^{2p}(\Omega;( H^{\otimes n_1}))$, and $u_2\in L^{2p}(\Omega, ( H^{\otimes n_2}))$, then $u_1\otimes u_2 \in L^p(\Omega; (H^{\otimes (n_1+n_2)}))$, and
\begin{align}\label{tpcsi}
\big\|\|u_1\otimes u_2\|_{ H^{\otimes (n_1+n_2)}}\big\|_{p} \leq \big\|\|u_1\|_{H^{\otimes n_1}}\big\|_{2p}\big\|\|u_2\|_{H^{\otimes n_2}}\big\|_{2p}.
\end{align}
\end{lemma}

\begin{proof}
The lemma can be obtained by the classical Cauchy-Schwarz inequality.
\end{proof}

\begin{lemma}\label{elgamma}
For any $p\geq 1$ and $0 \leq r<t\leq T$, there exists a constant $C>0$, that depends on $T$, $d$, $\|h\|_{3,2}$, $p$, such that
\begin{align}
\max_{1\leq i\leq d}\left\|\|D \xi_t^i\|_{H}\right\|_{2p}\leq &C(t-r)^{\frac{1}{2}}. \label{medxih}\\
\max_{1\leq i,j\leq d}\left\|\sigma_t^{ij}\right\|_{2p}\leq &C(t-r)^{-1},\label{eigamma}\\
 \max_{1\leq i,j\leq d}\left\|\|D\sigma_t^{ij}\|_{H}\right\|_{2p}\leq &C,\label{edgamma}\\
 \max_{1\leq i\leq d}\left\|\|D^2\xi_t^i\|_{H^{\otimes 2}}\right\|_{2p}\leq &C(t-r)^{\frac{3}{2}}.\label{eddxi}\\
\max_{1\leq i,j\leq d}\left\|\|D^2\sigma_t^{ij}\|_{H^{\otimes 2}}\right\|_{2p}\leq &C(t-r)^{\frac{1}{2}},\label{ed2gamma}\\
\max_{1\leq i\leq d}\left\|\|D^3\xi_t^i\|_{H^{\otimes 3}}\right\|_{2p}\leq &C(t-r)^2.\label{edddxi}
\end{align}
\end{lemma}

\begin{proof}
{\bf (i)} By (\ref{rfqvm}), (\ref{sdexi1m}), Jensen's, 
Burkholder-Davis-Gundy's, and Minkowski's inequalities, we have
\begin{align}\label{lambdai}
\sum_{i,k=1}^d\big\|D^{(k)}_{\theta}\xi_t^i\big\|_{2p}^2\leq &\sum_{i,k=1}^d\bigg(\delta_{ik}+\sum_{j=1}^d\Big\|\int_{\theta}^t\int_{\mathbb{R}^d}  D^{(k)}_{\theta}\xi_s^j d M^{ji}_s\Big\|_{2p}\bigg)^2\nonumber\\
\leq &(d+1)\sum_{i,k=1}^d\bigg(\delta_{ik}+\sum_{j=1}^d\Big\|\int_{\theta}^t\int_{\mathbb{R}^d}  D^{(k)}_{\theta}\xi_s^j d M^{ji}_s\Big\|_{2p}^2\bigg)\nonumber\\
\leq &d(d+1)+(d+1)c_p\sum_{i,j,k=1}^d Q_{ji}^{ji}\Big\|\int_{\theta}^t \big| D^{(k)}_{\theta}\xi_s^j\big|^2 d s\Big\|_{p}\nonumber\\
\leq &d(d+1)+2c_p d(d+1) \|h\|_{3,2}^2\sum_{j,k=1}^d \int_{\theta}^t \big\| D^{(k)}_{\theta}\xi_s^j\big\|_{2p}^2 ds.
\end{align}
Thus by Gr\"{o}nwall's lemma, we have
\begin{align}\label{elambda}
\sum_{i,j=1}^d \big\|D^{(k)}_{\theta}\xi_t^j\big\|_{2p}^2\leq d(d+1)\exp\left(2c_pd(d+1)\|h\|_{3,2}^2T\right):= C.
\end{align}
Therefore,  by (\ref{elambda}) and Minkowski's inequality, we have
\begin{align*}
\left\|\|D \xi_t^i\|_{H}\right\|_{2p}^2=\bigg\|\sum_{k=1}^d\int_r^t |D^{(k)}_{\theta}\xi_t^{i}|^2d\theta\bigg\|_{p}\leq \sum_{k=1}^d\int_r^t \big\|D^{(k)}_{\theta}\xi_t^{i}\big\|^2_{2p}d\theta\leq C(t-r).
\end{align*}

{\bf (ii)} In order to prove (\ref{eigamma}), we rewrite the SDE (\ref{ivlambda}) in the following way:
\begin{align}\label{ivlembdai}
\beta^{ij}_{\theta}(t)=&\delta_{ij}+\sum_{k_1=1}^d\int_{\theta}^t\beta^{ik_1}_{\theta}(s) dM^{jk_1}_s+\sum_{k_2=1}^d\int_{\theta}^t\beta^{k_2j}_{\theta}(s) dM^{ik_2}_s\nonumber\\
&+\sum_{k_1,k_2=1}^d\Big(Q^{i,k_1}_{k_1,k_2}\int_{\theta}^t\beta^{k_2j}_{\theta}(s)ds\Big)+\sum_{k_1,k_2=1}^d\Big(Q^{i,k_1}_{j,k_2}\int_{\theta}^t\beta^{k_1k_2}_{\theta}(s)ds\Big)\nonumber\\
&+\sum_{k_1,k_2=1}^d\Big(Q^{k_2,k_1}_{j,k_2}\int_{\theta}^t\beta^{ik_1}_{\theta}(s)ds\Big).
\end{align}
Similarly as we did in step {\bf (i)}, by Burkholder-Davis-Gundy's, and  Minkowski's inequalities, we can show that the martingale terms satisfies the following inequality
\begin{align}\label{tdiffusioni}
\Big\|\int_{\theta}^t\beta^{ik_1}_{\theta}(s)d M_s^{jk_1}\Big\|_{2p}^2\leq 2c_p \|h\|_{3,2}^2  \int_{\theta}^t\big\|\beta^{ik_1}_{\theta}(s)\big\|_{2p}^2ds.
\end{align}
For the drift terms, by Minkowski's and Jensen's inequality, we have
\begin{align}\label{tdrifti}
\Big\|\int_{\theta}^t \beta^{k_1k_2}_{\theta}(s)ds\Big\|_{2p}^2\leq  (t-\theta) \int_{\theta}^t  \left\|\beta^{k_1k_2}_{\theta}(s)\right\|_{2p}^2 ds.
\end{align}
Then, by (\ref{ivlembdai}) - (\ref{tdrifti}), and Gr\"{o}nwall's lemma, we have
\begin{align*}
\sum_{i,j=1}^d \left\|\beta^{ij}_{\theta}(t)\right\|_{2p}^2\leq C.
\end{align*}
Thus by Minkowski's and Jensen's inequalities, we have
\begin{align}\label{ebbeta}
\bigg\|\Big\|\int_r^t \beta_{\theta}(t)d \theta\Big\|_2\bigg\|_{2p}\leq c_d\sum_{i,j=1}^d\int_r^t\left\|\beta^{ij}_{\theta}(t)\right\|_{2p} d\theta \leq C(t-r).
\end{align}
Therefore, (\ref{eigamma}) follows from (\ref{hsniigamma}), (\ref{ebbeta}), Minkowski's and Jensen's inequalities.

{\bf (iii)} By integrating equation (\ref{lambda}) on both sides with respect to $\theta$, and applying the stochastic Fubini theorem, we have
\begin{align}\label{gmsde}
\gamma_t=\int_r^t \lambda_{\theta}(t)d\theta=&I(t-r)-\int_r^t \gamma_s dM_s-\int_r^t d M^*_s\cdot \gamma_s\\
&+\sum_{m=1}^d\int_r^t \int_{\mathbb{R}^d} g_m^*(y, s)\gamma_s g_m(y, s) dy ds.\nonumber
\end{align}
Taking the Malliavin derivative on both sides of (\ref{gmsde}), we have the following SDE:
\begin{align}\label{dgme}
 D^{(k)}_{\theta}\gamma_t^{ij}=&- \sum_{k_1=1}^d\int_{\theta}^tD^{(k)}_{\theta}\gamma_s^{ik_1} d M_s^{k_1j}-\sum_{k_1=1}^d\int_{\theta}^t \gamma_s^{ik_1} d \left(D^{(k)}_{\theta} M^{k_1j}_s\right)\nonumber\\
&-\sum_{k_2=1}^d\int_{\theta}^tD^{(k)}_{\theta}\gamma^{k_2j}_s d M^{k_2i}_s-\sum_{k_2=1}^d\int_{\theta}^t\gamma^{k_2j}_s d\left(D^{(k)}_{\theta} M^{k_2i}_s\right)\nonumber\\
&+\sum_{k_1,k_2=1}^d\left(Q^{k_1,i}_{k_2,j}\int_{\theta}^t D^{(k)}_{\theta}\gamma^{k_1k_2}_s ds\right),
\end{align}
where
\begin{align}\label{drdm}
D^{(k)}_{\theta} M^{ij}_s=-\sum_{i_1,i_2=1}^d\int_{\theta}^s \int_{\mathbb{R}^d} \partial_{i,i_2}h^{ji_1}\left(y-\xi_r\right)D^{(k)}_{\theta}\xi^{i_2}_r W^{i_1}(dr, dy).
\end{align}
For the first and the third term, by similar arguments as in (\ref{lambdai}), we can show that
\begin{align}\label{diffusiondgamma1}
\Big\|\int_{\theta}^tD^{(k)}_{\theta}\gamma^{ik_1}_s d M_s^{k_1j}\Big\|_{2p}^2\leq c_{d,p} \|h\|_{3,2}^2  \int_{\theta}^t \big\|D^{(k)}_{\theta}\gamma^{ik_1}_s\big\|_{2p}^2ds.
\end{align}
To estimate the second and the fourth term, notice that by (\ref{medxih}), we have
\begin{align}\label{egamma}
\max_{1\leq i,j\leq d}\left\|\gamma_t^{ij}\right\|_{2p}=&\max_{1\leq i,j\leq d}\left\|\langle D\xi_t^i, D\xi_t^j\rangle_H\right\|_{2p}\nonumber\\
\leq &\max_{1\leq i\leq d}\left\|\| D\xi_t^i\|_H\right\|_{4p} \max_{1\leq j\leq d}\left\|\|D\xi_t^j\|_H\right\|_{4p}\leq C(t-r).
\end{align}
Therefore, by (\ref{elambda}), (\ref{drdm}), (\ref{egamma}), Jensen's, Burkholder-Davis-Gundy's, Minkowski's, and Cauchy-Schwarz's inequalities, we have
\begin{align}\label{diffusionddm1}
\Big\|\int_{\theta}^t\gamma_s^{ik_1} d \Big(D^{(k)}_{\theta} M^{k_1j}_s\Big)\Big\|_{2p}^2\leq &c_{d,p}\|h\|_{3,2}^2\sum_{k_2=1}^d\int_{\theta}^t\big\|\gamma_s^{ik_1}\|_{4p}^2\|D^{(k)}_{\theta}\xi^{k_2}_s \big\|_{4p}^2ds\nonumber\\
\leq& C(t-r)^3.
\end{align}
For the last term, by Minkowski's and Jensen's inequalities, we have
\begin{align}\label{driftdgamma1}
\Big\|\int_{\theta}^t D^{(k)}_{\theta}\gamma^{k_1k_2}_sds\Big\|_{2p}^2\leq (t-\theta) \int_{\theta}^t  \big\|D^{(k)}_{\theta}\gamma^{k_1k_2}_s\big\|_{2p}^2 ds\leq T \int_{\theta}^t  \big\|D^{(k)}_{\theta}\gamma^{k_1k_2}_s\big\|_{2p}^2 ds.
\end{align}
Combining (\ref{dgme}) - (\ref{driftdgamma1}), we obtain the following inequality
\begin{align}
\sum_{i,j=1}^d \big\|D^{(k)}_{\theta} \gamma_t^{ij}\big\|_{2p}^2\leq c_1 \int_{\theta}^t\sum_{i,j=1}^d \big\|D^{(k)}_{\theta} \gamma_s^{ij}\big\|_{2p}^2 ds+c_2(t-r)^3,
\end{align}
where $c_1,c_2$ depends on $T$, $d$, $\|h\|_{3,2}^2$, and $p$. Thus by Gr\"{o}nwall's lemma, we have
\begin{align}\label{idgamma}
\sum_{i,j=1}^d\big\|D^{(k)}_{\theta} \gamma_t^{ij}\big\|_{2p}^2\leq C(t-r)^3.
\end{align}
It follows that
\begin{align}\label{idgammah}
\big\|\|D\gamma_t^{ij}\|_H\big\|_{2p}\leq C(t-r)^2
\end{align}
Notice that $\gamma_t\sigma_t=I$, a.s., as a consequence, $D\left(\gamma_t\sigma_t\right)=DI\equiv 0$. That implies
\begin{align}\label{digmf}
D\sigma^{ij}_t=-\sum_{i_1,i_2=1}^d\sigma^{ii_1}_tD\gamma^{i_1i_2}_t\sigma^{i_2j}_t.
\end{align}
Then, (\ref{edgamma})  follows from (\ref{tpcsi}), (\ref{eigamma}), (\ref{idgammah}) and (\ref{digmf}).

{\bf (iv)} Fix $0\leq r<t\leq T$. For any $\theta_1,\theta_2\in [r,t]$, let $\theta=\theta_1\vee \theta_2$. Taking the Malliavin derivative on both sides of (\ref{sdexi1m}), we have the following SDE:
\begin{align}\label{d2xisde}
D^{(k_1,k_2)}_{\theta_1,\theta_2}\xi_t^i=&-\sum_{j_1=1}^d\int_{\theta}^t D^{(k_1,k_2)}_{\theta_1,\theta_2}\xi_s^{j_1} dM^{j_1i}_s\nonumber\\
&+\sum_{j_1,j_2,j_3=1}^d\int_{\theta}^t\int_{\mathbb{R}^d}\partial_{j_2,j_3} h^{ij_1}(y-\xi_s) D^{(k_1)}_{\theta_1}\xi_s^{j_2} D^{(k_2)}_{\theta_2}\xi_s^{j_3}W^{j_1}(ds, dy).
\end{align}
Similarly as in (\ref{lambdai}), we can show the following inequalities
\begin{align}\label{ed2xi1}
\Big\|\int_{\theta}^t D^{(k_1,k_2)}_{\theta_1,\theta_2}\xi_s^{j_1} dM^{j_1i}_s\Big\|_{2p}^2\leq c_{d,p}\|h\|_{3,2}^2\int_{\theta}^t \big\|D^{(k_1,k_2)}_{\theta_1,\theta_2}\xi_s^{j_1}\big\|_{2p}^2ds,
\end{align}
and
\begin{align}\label{ed2xi2}
&\Big\|\int_{\theta}^t\int_{\mathbb{R}^d}\partial_{j_2,j_3} h^{ij_1}(y-\xi_s) D^{(k_1)}_{\theta_1}\xi_s^{j_2} D^{(k_2)}_{\theta_2}\xi_s^{j_3}W^{j_1}(ds, dy)\Big\|_{2p}^2\nonumber\\
&\qquad \leq c_p\|h\|_{3,2}^2\int_{\theta}^t\big\|D^{(k_1)}_{\theta_1}\xi_s^{j_2}\big\|_{4p}^2\big\| D^{(k_2)}_{\theta_2}\xi_s^{j_3}\big\|_{4p}^2ds\leq  C (t-r).
\end{align}
Thus combining (\ref{d2xisde}) - (\ref{ed2xi2}), we have
\begin{align*}
\sum_{i=1}^d\big\|D^{(k_1,k_2)}_{\theta_1,\theta_2}\xi_t^i\big\|_{2p}^2\leq &c_1\sum_{i=1}^d\int_{\theta}^t \big\|D^{(k_1,k_2)}_{\theta_1,\theta_2}\xi_s^i\big\|_{2p}^2ds+c_2(t-r).
\end{align*}
Then, it follows from Gr\"{o}nwall's lemma that
\begin{align}\label{mdxi2}
\sum_{i=1}^d\big\|D^{(k_1,k_2)}_{\theta_1,\theta_2}\xi_t^i\big\|_{2p}^2\leq C(t-r).
\end{align}
The inequality (\ref{eddxi}) is a consequence of (\ref{mdxi2}), Jensen's and Minkowski's inequalities.

{\bf (v)} For any  $\theta_1,\theta_2\in [r,t]$ and $\theta=\theta_1\vee\theta_2$, by taking the Malliavin derivative on both sides of (\ref{dgme}), we have
\begin{align}\label{rfdgamma2}
&D^{(k_1,k_2)}_{\theta_1,\theta_2}\gamma_t^{ij}=- \sum_{i_1=1}^d\Big(\int_{\theta}^tD^{(k_1,k_2)}_{\theta_1,\theta_2}\gamma_s^{ii_1} d M_s^{i_1j}+\int_{\theta}^tD^{(k_1)}_{\theta_2}\gamma_s^{ii_1} d \left(D^{(k_2)}_{\theta_1}M_s^{i_1j}\right)\Big)\nonumber\\
&\quad -\sum_{i_1=1}^d\Big(\int_{\theta}^t D^{(k_2)}_{\theta_2}\gamma_s^{ii_1} d \left(D^{(k_1)}_{\theta_1} M^{i_1j}_s\right)+\int_{\theta}^t \gamma_s^{ii_1} d \left(D^{(k_1,k_2)}_{\theta_1,\theta_2} M^{i_1j}_s\right)\Big)\nonumber\\
&\quad -\sum_{i_2=1}^d\Big(\int_{\theta}^tD^{(k_1,k_2)}_{\theta_1,\theta_2}\gamma^{i_2j}_s d M^{i_2i}_s+\int_{\theta}^tD^{(k_1)}_{\theta}\gamma^{i_2j}_s d \left(D^{(k_2)}M^{i_2i}_s\right)\Big)\nonumber\\
&\quad -\sum_{i_2=1}^d\Big(\int_{\theta}^tD^{(k_2)}_{\theta_1}\gamma^{i_2j}_s d\left(D^{(k_1)}_{\theta_2} M^{i_2i}_s\right)+\int_{\theta}^t\gamma^{i_2j}_s d\left(D^{(k_1,k_2)}_{\theta_1,\theta_2} M^{i_2i}_s\right)\Big)\nonumber\\
&\quad +\sum_{i_1,i_2=1}^d\Big(Q^{i_1,i}_{i_2,j}\int_{\theta}^t D^{(k_1,k_2)}_{\theta_1,\theta_2}\gamma^{i_1i_2}_s ds\Big),
\end{align}
where
\begin{align*}
D^{(k_1,k_2)}_{\theta_1,\theta_2} M^{ij}_s=&-\sum_{j_1,j_2,j_3=1}^d\int_{\theta}^s \int_{\mathbb{R}^d} \partial_{i,j_2,j_3}h^{jj_1}\left(y-\xi_r\right)D^{(k_1)}_{\theta_1}\xi^{j_2}_rD^{(k_2)}_{\theta_2}\xi^{j_3}_rW^{j_1}(dr, dy)\\
&+\sum_{j_1,j_2=1}^d \int_{\theta}^s \int_{\mathbb{R}^d} \partial_{i,j_2}h^{jj_1}\left(y-\xi_r\right)D^{(k_1,k_2)}_{\theta_1,\theta_2}\xi^{j_2}_rW^{j_1}(dr, dy).
\end{align*}
By (\ref{elambda}), (\ref{egamma}), (\ref{idgamma}), (\ref{mdxi2}), Burkholder-Davis-Gundy's, Minkowski's and H\"{o}lder's inequalities, we have the following inequalities
\begin{align}\label{ddgdm}
\Big\|\int_{\theta}^tD^{(k_1,k_2)}_{\theta_1,\theta_2}\gamma_s^{ii_1} d M_s^{i_1j}\Big\|_{2p}^2\leq c_{d,p}\|h\|_{3,2}^2\int_{\theta}^t\big\|D^{(k_1,k_2)}_{\theta_1,\theta_2}\gamma_s^{ii_1}\big\|_{2p}^2ds,
\end{align}
\begin{align}\label{dgmddm}
&\Big\|\int_{\theta}^tD^{(k_1)}_{\theta_2}\gamma_s^{ii_1} d \left(D^{(k_2)}_{\theta_1}M_s^{i_1j}\right)\Big\|_{2p}^2\leq c_{d,p}\|h\|_{3,2}^2\sum_{i_2=1}^d \int_{\theta}^t\left\|D^{(k_1)}_{\theta_2}\gamma_s^{ii_1}D_{\theta_2}^{(k_2)}\xi_s^{i_2}\right\|_{2p}^2ds\nonumber\\
\leq &c_{d,p}\|h\|_{3,2}^2\sum_{i_2=1}^d \int_{\theta}^t\left\|D^{(k_1)}_{\theta_2}\gamma_s^{ii_1}\right\|_{4p}^2\left\|D_{\theta_2}^{(k_2)}\xi_s^{i_2}\right\|_{4p}^2ds\leq C(t-r)^4,
\end{align}
and
\begin{align*}
&\Big\|\int_{\theta}^t\gamma_t^{ii_1}d\left(D^{(k_1,k_2)}_{\theta_1,\theta_2}M_s^{i_1j}\right)\Big\|_{2p}^2\\
\leq&c_d\bigg(\sum_{j_1,j_2,j_3=1}^d\Big\|\int_{\theta}^t \int_{\mathbb{R}^d} \gamma_s^{ii_1}\partial_{i_1,j_2,j_3}h^{jj_1}\left(y-\xi_r\right)D^{(k_1)}_{\theta_1}\xi^{j_2}_sD^{(k_2)}_{\theta_2}\xi^{j_3}_sW^{j_1}(ds, dy)\Big\|_{2p}^2 \nonumber\\
&+\sum_{j_1,j_2=1}^d \Big\|\int_{\theta}^t \int_{\mathbb{R}^d} \gamma_s^{ii_1}\partial_{i_1,j_2}h^{jj_1}\left(y-\xi_s\right)D^{(k_1,k_2)}_{\theta_1,\theta_2}\xi^{j_2}_sW^{j_1}(ds, dy)\Big\|_{2p}^2\bigg):=c_d\left(I_1+I_2\right).\nonumber
\end{align*}
We estimate $I_1$, $I_2$ as follows:
\begin{align*}
I_1\leq& d\|h\|_{3,2}^2\sum_{j_2,j_3=1}^d\int_{\theta}^t\left\|\gamma_s^{ii_1}\right\|_{6p}^2\big\|D^{(k_1)}_{\theta_1}\xi^{j_2}_s\big\|_{6p}^2\big\|D^{(k_2)}_{\theta_2}\xi^{j_3}_s\big\|_{6p}^2 ds\leq C(t-r)^3,
\end{align*}
and
\begin{align*}
I_2\leq d\|h\|_{3,2}^2\sum_{j_2=1}^d\int_{\theta}^t\left\|\gamma_s^{ii_1}\right\|_{4p}^2\big\|D^{(k_1,k_2)}_{\theta_1,\theta_2}\xi^{j_2}_s\big\|_{4p}^2ds\leq C(t-r)^4\leq CT(t-r)^3.
\end{align*}
Thus we have
\begin{align}\label{gmd3m0}
\Big\|\int_{\theta}^t\gamma_t^{ii_1}d\left(D^{(k_1,k_2)}_{\theta_1,\theta_2}M_s^{i_1j}\right)\Big\|_{2p}^2\leq C(t-r)^3.
\end{align}
Therefore, combine (\ref{rfdgamma2}) - (\ref{gmd3m0}), we have
\begin{align*}
\sum_{i,j=1}^d\left\|D^{(k_1,k_2)}_{\theta_1,\theta_2}\gamma_t^{ij}\right\|_{2p}^2\leq c_1(t-r)^3+ c_2\sum_{i,j=1}^d\int_\theta^t\left\|D^{(k_1,k_2)}_{\theta_1,\theta_2}\gamma_s^{ij}\right\|_{2p}^2ds,
\end{align*}
By Gr\"{o}nwall's lemma, we have
\begin{align}\label{ddgm0}
\sum_{i,j=1}^d\big\|D^{(k_1,k_2)}_{\theta_1,\theta_2}\gamma_t^{ij}\big\|_{2p}^2\leq C(t-r)^3,
\end{align}
which implies
\begin{align*}
\left\|\|D^2\gamma_t^{ij}\|_{H^{\otimes 2}}\right\|_{2p}\leq C(t-r)^\frac{5}{2}.
\end{align*}
By taking the second Malliavin derivative of $\gamma_t\sigma_t\equiv I$, we have
\begin{align}\label{ddigm}
D^2\sigma^{ij}_t=&-\sum_{i_1,i_2=1}^d\sigma^{ii_1}_t\big(D^2\gamma^{i_1i_2}_t\sigma^{i_2j}_t+D\gamma^{i_1i_2}_t\otimes D\sigma^{i_2j}_t+D\sigma^{i_2j}_t\otimes D\gamma^{i_1i_2}_t\big).
\end{align}
Then, (\ref{ed2gamma}) can be deduced by  (\ref{tpcsi}), (\ref{eigamma}), (\ref{edgamma}), (\ref{idgammah}) and (\ref{ddigm}).

{\bf (vi)} For any $\theta_1,\theta_2,\theta_3\in [r,t]$, let $\theta=\theta_1\vee \theta_2\vee \theta_3$.  Taking the Malliavin derivative on both sides of (\ref{d2xisde}), we have
\begin{align}\label{d3xisde}
D^{(k_1,k_2,k_3)}_{\theta_1,\theta_2,\theta_3}&\xi_t^i=\sum_{j_1,j_2,j_3=1}^d\int_{\theta}^t\int_{\mathbb{R}^d} \partial_{j_2,j_3} h^{ij_1}(y-\xi_s) D^{(k_1,k_2)}_{\theta_1,\theta_2}\xi_s^{j_2}D^{(k_3)}_{\theta_3}\xi_s^{j_3} W^{j_1}(ds, dy)\nonumber\\
&-\sum_{j_1,j_2=1}^d\int_{\theta}^t\int_{\mathbb{R}^d} \partial_{j_2} h^{ij_1}(y-\xi_s) D^{(k_1,k_2,k_3)}_{\theta_1,\theta_2,\theta_3}\xi_s^{j_2} W^{j_1}(ds, dy)\nonumber\\
&-\sum_{j_1,j_2,j_3,j_4=1}^d\int_{\theta}^t\int_{\mathbb{R}^d}\partial_{j_2,j_3,j_4} h^{ij_1}(y-\xi_s) D^{(k_1)}_{\theta_1}\xi_s^{j_2} D^{(k_2)}_{\theta_2}\xi_s^{j_3}D^{(k_3)}_{\theta_3}\xi_s^{j_4}W^{j_1}(ds, dy)\nonumber\\
&+\sum_{j_1,j_2,j_3=1}^d\int_{\theta}^t\int_{\mathbb{R}^d}\partial_{j_2,j_3} h^{ij_1}(y-\xi_s) D^{(k_1,k_3)}_{\theta_1, \theta_3}\xi_s^{j_2} D^{(k_2)}_{\theta_2}\xi_s^{j_3}W^{j_1}(ds, dy)\nonumber\\
&+\sum_{j_1,j_2,j_3=1}^d\int_{\theta}^t\int_{\mathbb{R}^d}\partial_{j_2,j_3} h^{ij_1}(y-\xi_s) D^{(k_1)}_{\theta_1}\xi_s^{j_2} D^{(k_2,k_3)}_{\theta_2,\theta_3}\xi_s^{j_3}W^{j_1}(ds, dy).
\end{align}
By (\ref{elambda}), (\ref{mdxi2}), Burkholder-Davis-Gundy's, Minkowski's, and H\"{o}lder's inequalities, we have the following inequalities:
\begin{align}\label{d3xi1}
&\Big\|\int_{\theta}^t\int_{\mathbb{R}^d} \partial_{j_2,j_3} h^{ij_1}(y-\xi_s) D^{(k_1,k_2)}_{\theta_1,\theta_2}\xi_s^{j_2}D^{(k_3)}_{\theta_3}\xi_s^{j_3} W^{k_1}(ds, dy)\Big\|_{2p}^2\nonumber\\
\leq &c_p\|h\|_{3,2}^2\int_{\theta}^t \big\|D^{(k_1, k_2)}_{\theta_1,\theta_2}\xi_s^{j_2}\big\|_{4p}^2\big\|D^{(k_3)}_{\theta_3}\xi_s^{j_3}\big\|_{4p}^2ds\leq C(t-r)^2,
\end{align}
\begin{align}\label{d2xi2}
&\Big\|\int_{\theta}^t\int_{\mathbb{R}^d} \partial_{j_2} h^{ij_1}(y-\xi_s) D^{(k_1,k_2,k_3)}_{\theta_1,\theta_2,\theta_3}\xi_s^{j_2} W^{j_1}(ds, dy)\Big\|_{2p}^2\nonumber\\
\leq&c_p\|h\|_{3,2}^2\int_{\theta}^t\big\|D^{(k_1,k_2,k_3)}_{\theta_1,\theta_2,\theta_3}\xi_s^{j_2}\big\|_{2p}^2ds,
\end{align}
and
\begin{align}\label{d3xi3}
&\Big\|\int_{\theta}^t\int_{\mathbb{R}^d}\partial_{j_2,j_3,j_4} h^{ij_1}(y-\xi_s) D^{(k_1)}_{\theta_1}\xi_s^{j_2} D^{(k_2)}_{\theta_2}\xi_s^{j_3}D^{(k_3)}_{\theta_3}\xi_s^{j_3}W^{j_1}(ds, dy)\Big\|_{2p}^2\nonumber\\
\leq&c_p\|h\|_{3,2}^2\int_{\theta}^t\big\|D^{(k_1)}_{\theta_1}\xi_s^{j_2}\big\|_{6p}^2\big\|D^{(k_2)}_{\theta_2}\xi_s^{j_3}\big\|_{6p}^2\big\|D^{(k_3)}_{\theta_3}\xi_s^{j_3}\big\|_{6p}^2ds\leq C(t-r).
\end{align}
Thus combining (\ref{d3xisde}) - (\ref{d3xi3}), by Jensen's inequality, we have
\begin{align*}
\sum_{i=1}^d\big\|D^{(k_1,k_2,k_3)}_{\theta_1,\theta_2,\theta_3}\xi_t^i\big\|_{2p}^2\leq &c_1\sum_{i=1}^d\int_{\theta}^t \big\|D^{(k_1,k_2,k_3)}_{\theta_1,\theta_2,\theta_3}\xi_t^i\big\|_{2p}^2ds+c_2(t-r).
\end{align*}
Then, the following inequality follows from Gr\"{o}nwall's lemma
\begin{align}\label{d3xi0}
\sum_{i=1}^d\big\|D^{(k_1,k_2,k_3)}_{\theta_1,\theta_2,\theta_3}\xi_t^i\big\|_{2p}^2\leq C(t-r).
\end{align}
Therefore, (\ref{edddxi}) is a consequence of (\ref{d3xi0}).
\end{proof}

In the next lemma, we derive estimates for the moments of increments of the derivatives of $\xi_t$ and $\sigma_t$.
\begin{lemma}\label{medxigm}
For any $p\geq 1$, $0 \leq  r< s<t\leq T$, and $1\leq i,j\leq d$, there exists a constant $C>0$ depends on $T$, $d$, $p$, and $\|h\|_{3,2}$, such that
\begin{align}
\max_{1\leq i\leq d}\left\|\|D\xi_t^i-D\xi_s^i\|_{H}\right\|_{2p}\leq &C(t-s)^{\frac{1}{2}},\label{ddxi}\\
\max_{1\leq i,j\leq d}\left\|\sigma_t^{ij}-\sigma_s^{ij}\right\|_{2p}\leq &C(t-r)^{-\frac{1}{2}}(s-r)^{-1}(t-s)^{\frac{1}{2}},\label{dgamma}\\
\max_{1\leq i,j\leq d}\left\|\|D\sigma_t^{ij}-D\sigma_s^{ij}\|_{H}\right\|_{2p}\leq &C(t-r)^{-\frac{1}{2}}(t-s)^{\frac{1}{2}},\label{ddgamma}\\
\max_{1\leq i\leq d}\left\|\|D\xi_t^i-D^2\xi_s^i\|_{H^{\otimes 2}}\right\|_{2p}\leq &C(t-r)(t-s)^{\frac{1}{2}}.\label{ddxi2}
\end{align}
\end{lemma}
\begin{proof}
{\bf (i)} By (\ref{sdexi1m}), we have
\begin{align*}
D^{(k)}_{\theta}\xi_t^i-D^{(k)}_{\theta}\xi_s^i=\delta_{ik}\1_{[s,t]}(\theta)-\sum_{j=1}^d\int_{\theta\vee s}^tD^{(k)}_{\theta}\xi_u^j dM^{ji}_u.
\end{align*}
Thus by (\ref{elambda}), Burkholder-Davis-Gundy's, Jensen's, and Minkowski's inequalities, we have
\begin{align*}
\big\|D^{(k)}_{\theta}\xi_t^i-D^{(k)}_{\theta}\xi_s^i\big\|_{2p}^2\leq C\left[\delta_{ik}\1_{[s,t]}(\theta)+ (t-s)\right].
\end{align*}
Thus we can show (\ref{ddxi}) by Minkowski's inequality:
\begin{align*}
\left\|\|D\xi^i_t-D\xi^i_s\|_{H}\right\|_{2p}^2\leq &\sum_{k=1}^d\int_r^t\big\|D^{(k)}_{\theta}\xi_t^i-D^{(k)}_{\theta}\xi_s^i\big\|_{2p}^2 d\theta\nonumber\\
\leq&\sum_{k=1}^dC\Big(\int_s^t\delta_{ik}d\theta+\int_r^t(t-s) d\theta\Big)\leq C(t-s).
\end{align*}

{\bf (ii)} Note that $\sigma_t-\sigma_s=\sigma_t\left(\gamma_s-\gamma_t\right)\sigma_s$. Then, by (\ref{eigamma}) and H\"{o}lder's inequality, it suffices to estimate the moment of $\gamma_t-\gamma_s$. By (\ref{gmsde}), we have
\begin{align*}
\gamma_t^{ij}-\gamma_s^{ij}=&\delta_{ij}(t-s)-\sum_{k_1=1}^d\int_s^t\gamma_u^{ik_1}dM_u^{k_1j}-\sum_{k_2=1}^d\int_s^t\gamma_u^{jk_2}dM_u^{k_2i}\\
&+\sum_{k_1,k_2=1}^dQ^{i,k_1}_{k_2,j}\int_s^t \gamma_u^{k_1k_2}du.
\end{align*}
Then, by (\ref{egamma}), Minkowski's, Jensen's, and Burkholder-Davis-Gundy's inequalities, for all $1\leq i,j\leq d$, we have
\begin{align}\label{dgamma1}
\big\|\gamma_t^{ij}-\gamma_s^{ij}\big\|_{2p}^2\leq &C\left((t-s)^2+(t-r)^2(t-s)+(t-r)^2(t-s)^2\right)\nonumber\\
\leq &C(1+T)^2(t-r)(t-s).
\end{align}
Then, (\ref{dgamma}) is a consequence of (\ref{eigamma}) and (\ref{dgamma1}).

{\bf (iii)} By (\ref{dgme}), we have the following equation:
\begin{align*}
D^{(k)}_{\theta}\gamma_t^{ij}-D^{(k)}_{\theta}\gamma_s^{ij}=&- \sum_{k_1=1}^d\int_{\theta\vee s}^tD^{(k)}_{\theta}\gamma_u^{ik_1} d M_u^{k_1j}-\sum_{k_1=1}^d\int_{\theta \vee s}^t \gamma_u^{ik_1} d \left(D^{(k)}_{\theta} M^{k_2j}_u\right)\nonumber\\
&-\sum_{k_2=1}^d\int_{\theta \vee s}^tD^{(k)}_{\theta}\gamma^{k_2j}_u d M^{k_2i}_u-\sum_{k_2=1}^d\int_{\theta\vee s}^t\gamma^{k_2j}_u d\left(D^{(k)}_{\theta} M^{k_2i}_u\right)\\
&+\sum_{k_1,k_2=1}^d\Big(Q^{k_1,i}_{k_2,j}\int_{\theta \vee s}^t D^{(k)}_{\theta}\gamma^{k_1k_2}_u du\Big).
\end{align*}
Then, by (\ref{elambda}), (\ref{egamma}), and (\ref{idgamma}), Burkholder-Davis-Gundy's, Jensen's, Minkowski's, and Cauchy-Schwarz's inequalities, we have
\begin{align*}
&\big\|D^{(k)}_{\theta}\gamma_t^{ij}-D^{(k)}_{\theta}\gamma_s^{ij}\big\|_{2p}^2\leq c_{d,p} \|h\|_{3,2}^2 \bigg[\sum_{k_1=1}^d\int_{\theta\vee s}^t \big\|D^{(k)}_{\theta}\gamma^{ik_1}_u\big\|_{2p}^2du\nonumber\\
&\hspace{20mm}+\sum_{k_2=1}^d\int_{\theta\vee s}^t\big\|\gamma_u^{ik_1}\big\|_{4p}^2\big\|D^{(k)}_{\theta}\xi^{k_2}_u \big\|_{4p}^2du+(t-s)\int_{\theta\vee s}^t  \big\|D^{(k)}_{\theta}\gamma^{k_1k_2}_u\big\|_{2p}^2 du\bigg]\\
&\hspace{38mm} \leq C(t-r)^2(t-s).
\end{align*}
This implies
\begin{align}\label{ddgamma0}
\left\|\|D\gamma_t^{ij}-D\gamma_s^{ij}\|_{H}\right\|_{2p}\leq C(t-r)^{\frac{3}{2}}(t-s)^{\frac{1}{2}}.
\end{align}
By (\ref{digmf}), we have
\begin{align*}
D\sigma^{ij}_t-D\sigma^{ij}_s=&\sum_{i_1,i_2=1}^d\left(\sigma^{ii_1}_tD\gamma^{i_1i_2}_t\sigma^{i_2j}_t-\sigma^{ii_1}_sD\gamma^{i_1i_2}_s\sigma^{i_2j}_s\right)\\
=&\sum_{i_1,i_2=1}^d\sigma^{ii_1}_t\left(D\gamma^{i_1i_2}_t-D\gamma^{i_1i_2}_s\right)\sigma^{i_2j}_t+\sum_{i_1,i_2=1}^d\left(\sigma^{ii_1}_t-\sigma^{ii_1}_s\right)D\gamma^{i_1i_2}_s\sigma^{i_2j}_t\\
&+\sum_{i_1,i_2=1}^d\sigma^{ii_1}_sD\gamma^{i_1i_2}_s\left(\sigma^{i_2j}_t-\sigma^{i_2j}_s\right).
\end{align*}
Thus (\ref{ddgamma}) follows from (\ref{tpcsi}), (\ref{eigamma}), (\ref{idgammah}), (\ref{dgamma}), and (\ref{ddgamma0}).

{\bf (iv)} Let $\theta=\theta_1\vee \theta_2$, by (\ref{d2xisde}), we have the following equation:
\begin{align*}
D^{(k_1,k_2)}_{\theta_1,\theta_2}\xi_t^i-&D^{(k_1,k_2)}_{\theta_1,\theta_2}\xi_s^i=-\sum_{j_1,j_2=1}^d\int_{\theta\vee s}^t\int_{\mathbb{R}^d} \partial_{j_2} h^{ij_1}(y-\xi_u) D^{(k_1,k_2)}_{\theta_1,\theta_2}\xi_u^{j_2} W^{j_1}(du, dy)\nonumber\\
&+\sum_{j_1,j_2,j_3=1}^d\int_{\theta\vee s}^t\int_{\mathbb{R}^d}\partial_{j_2,j_3} h^{ij_1}(y-\xi_u) D^{(k_1)}_{\theta_1}\xi_u^{j_2} D^{(k_2)}_{\theta_2}\xi_u^{j_3}W^{j_1}(du, dy).
\end{align*}
As a consequence, by (\ref{elambda}), (\ref{mdxi2}), Burkholder-Davis-Gundy's, Minkowski's, and Cauchy-Schwarz's inequalities, we have
\begin{align}\label{ddxi2bi}
\big\|D^{(i,j)}_{\theta_1,\theta_2}\xi_t^k-D^{(i,j)}_{\theta_1,\theta_2}\xi_s^k\big\|_{2p}^2\leq &c_p\bigg[\sum_{j_1=1}^d\|h\|_{3,2}^2\int_{\theta\vee s}^t\big\|D^{(i,j)}_{\theta_1,\theta_2}\xi_u^{j_1}\big\|_{2p}^2du\nonumber\\
&\quad +\sum_{j_1,j_2}^d\|h\|_{3,2}^2\int_{\theta\vee s}^t\big\|D^{(i)}_{\theta_1}\xi_u^{j_1}\big\|_{4p}^2\big\| D^{(j)}_{\theta_2}\xi_u^{j_2}\big\|_{4p}^2du\bigg]\nonumber\\
\leq &C(t-s)
\end{align}
Therefore, we obtain (\ref{ddxi2}) by integrating (\ref{ddxi2bi}) and Minkowski's  inequality.
\end{proof}

In the next lemma, we establish moment estimates for the functionals $H_{(i)}(\xi_t,1)$ and $H_{(i,j)}(\xi_t,1)$ introduced in (\ref{hfphif}) and (\ref{hhfphif}). Notice that these functionals are well defined, because $\xi_t\in\cap_{p\geq 1}\D^{3,p}(\R^d)$ and $\sigma^{ij}_t\in\cap_{p\geq 2}\D^{2,p}$ for all $1\leq i,j\leq d$.
\begin{lemma}\label{ehn}
Suppose that $h\in H_2^3(\R^d;\R^d\otimes \R^d)$, then  the following inequalities are satisfied:
\begin{align}
\max_{1\leq i\leq d}\left\|H_{(i)}(\xi_t,1)\right\|_{2p}\leq C(t-r)^{-\frac{1}{2}},\label{ehxi}\\
\max_{1\leq i,j\leq d}\left\|H_{(i,j)}(\xi_t,1)\right\|_{2p}\leq C(t-r)^{-1}.\label{ehxi2}
\end{align}
\end{lemma}
\begin{proof}
Due to Meyer's inequality (see e.g. Proposition 1.5.4 and 2.1.4 of Nualart \cite{springer-06-nualart}), it suffices to estimate
\[
\left\|\|\sigma^{ji}_t D\xi_t^j\|_{H}\right\|_{2p},\ \left\|\|D\left(\sigma^{ji}_t  D\xi_t^j\right)\|_{H^{\otimes 2}}\right\|_{2p},\ \textrm{and}\ \left\|\|D^2\left(\sigma^{ji}_t  D\xi_t^j\right)\|_{H^{\otimes 3}}\right\|_{2p}.
\]
By (\ref{medxih}) and Lemma \ref{ttpcsi} - \ref{elgamma}, we have
\begin{align*}
\left\|\|\sigma^{ji}_t  D\xi_t^j\|_{H}\right\|_{2p}\leq \left\|\sigma^{ji}_t\right\|_{4p}\left\|\|D\xi_t^j\|_H\right\|_{4p}\leq C(t-r)^{-\frac{1}{2}},
\end{align*}
\begin{align*}
\left\|\|D\left(\sigma^{ji}_t  D\xi_t^j\right)\|_{H^{\otimes 2}}\right\|_{2p}\leq&\left\|\|D\sigma^{ji}_t \otimes D\xi_t^j\|_{H^{\otimes 2}}\right\|_{2p}+\left\|\|\sigma^{ji}_t  D^2\xi_t^j\|_{H^{\otimes 2}}\right\|_{2p}\\
 \leq &\left\|\|D\sigma^{ji}_t\|_H \right\|_{4p}\left\|\|D\xi_t^j\|_H\right\|_{4p}+\left\|\sigma^{ji}_t \right\|_{4p}\left\|\|D^2\xi_t^j\|_{H^{\otimes 2}}\right\|_{4p}\\
 \leq &C(t-r)^{\frac{1}{2}},
\end{align*}
and
\begin{align*}
\left\|\|D^2\left(\sigma^{ji}_t  D\xi_t^j\right)\|_{H^{\otimes 3}}\right\|_{2p}\leq&\left\|\|D^2\sigma^{ji}_t \otimes D\xi_t^j\|_{H^{\otimes 2}}\right\|_{2p}\\
&+\left\|\|D\sigma^{ji}_t \otimes D^2\xi_t^j\|_{H^{\otimes 2}}\right\|_{2p}+\left\|\|\sigma^{ji}_t D^3\xi_t^j\|_{H^{\otimes 2}}\right\|_{2p}\\
\leq & C (t-r)
\end{align*}
The above inequalities hold for all $1\leq i,j\leq d$. Then, (\ref{ehxi}) and (\ref{ehxi2}) follows.
\end{proof}

The next lemma provides the moment estimate for the increment of $H_{(i)}(\xi_t,1)$.
\begin{lemma}\label{edh12}
Suppose that $h\in H_2^3(\R^d;\R^d\otimes \R^d)$. Then,
\begin{align}\label{edh1}
\max_{1\leq i\leq d}\left\|H_{(i)}(\xi_t,1)-H_{(i)}(\xi_s,1)\right\|_{2p}\leq C(s-r)^{-\frac{1}{2}}(t-r)^{-\frac{1}{2}}(t-s)^{\frac{1}{2}}.
\end{align}
\end{lemma}
\begin{proof}
Notice that, by definition, we have
\begin{align*}
H_{(i)}(\xi_t,1)-H_{(i)}(\xi_s,1)=&-\sum_{j=1}^d\delta\left(\sigma_t^{ji} D\xi_t^j\right)+\sum_{j=1}^d\delta\left(\sigma_s^{ji} D\xi_s^j\right)\\
=&-\sum_{j=1}^d\delta\left(\sigma_t^{ji} D\xi_t^j-\sigma_s^{ji} D\xi_s^j\right).
\end{align*}
Thus by Meyer's inequality again, it suffices to estimate
\begin{align*}
I_1:=\left\|\|\sigma_t^{ji} D\xi_t^j-\sigma_s^{ji} D\xi_s^j\|_H\right\|_{2p}\ \textrm{and}\ I_2:=\left\|\|D\left(\sigma_t^{ji} D\xi_t^j-\sigma_s^{ji} D\xi_s^j\right)\|_{H^{\otimes 2}}\right\|_{2p}.
\end{align*}
For $I_1$, we have
\begin{align*}
I_1\leq \left\|\|\left(\sigma_t^{ji}-\sigma_s^{ji}\right) D\xi_s^k\|_H\right\|_{2p}+\left\|\|\sigma_t^{ji} \left(D\xi_t^j-D\xi_s^j\right)\|_H\right\|_{2p}.
\end{align*}
Notice that by Lemmas \ref{ttpcsi} - \ref{medxigm}, we can write
\begin{align*}
&\left\|\|\left(\sigma_t^{ji}-\sigma_s^{ji}\right) D\xi_s^j\|_H\right\|_{2p}\leq \left\|\sigma_t^{ji}-\sigma_s^{ji}\right\|_{4p}\left\| \|D\xi_s^j\|_H\right\|_{4p}\\
&\hspace{20mm}\leq C(t-r)^{-\frac{1}{2}}(s-r)^{-\frac{1}{2}}(t-s)^{\frac{1}{2}}
\end{align*}
and
\begin{align*}
&\left\|\|\sigma_t^{ji} \left(D\xi_t^j-D\xi_s^j\right)\|_H\right\|_{2p}\leq \left\|\sigma_t^{ji}\right\|_{4p}\left\|\|D\xi_t^j-D\xi_s^j\|_H\right\|_{4p}\\
&\quad \leq C(t-r)^{-1}(t-s)^{\frac{1}{2}}\leq C(t-r)^{-\frac{1}{2}}(s-r)^{-\frac{1}{2}}(t-s)^{\frac{1}{2}}.
\end{align*}
Thus combining the above inequalities, we have the following estimate for $I_1$:
\begin{align}\label{edgdxi}
I_1\leq C(t-r)^{-\frac{1}{2}}(s-r)^{-\frac{1}{2}}(t-s)^{\frac{1}{2}}.
\end{align}
By Lemmas \ref{ttpcsi} - \ref{medxigm}, we have the following estimate for $I_2$:
\begin{align}\label{edgammadxi}
I_2\leq &\left\|D\sigma_t^{ji}\otimes D\xi_t^j-D\sigma_t^{ji}\otimes D\xi_s^j\right\|_{2p, H^{\otimes 2}}+\left\|\sigma_t^{ji} D^2\xi_s^j-\sigma_s^{ji}D^2\xi_s^j\right\|_{2p, H^{\otimes 2}}\nonumber\\
\leq &\left\|\|D\sigma_t^{ji}\|_H\right\|_{4p} \left\|\|\left(D\xi_t^j- D\xi_s^j\right)\|_H\right\|_{4p}+\left\|\|\left(D\sigma_t^{ji}-D\sigma_s^{ji}\right)\|_H\right\|_{4p}\left\| \|D\xi_s^j\|_H\right\|_{4p}\nonumber\\
&+\left\|\sigma_t^{ji} \right\|_{4p}\left\|\|D^2\xi_t^j- D^2\xi_s^j\|_{H^{\otimes 2}}\right\|_{4p}+\left\|\sigma_t^{ji} -\sigma_s^{ji} \right\|_{4p} \left\|\|D^2\xi_s^j\|_{H^{\otimes 2}}\right\|_{2p}\nonumber\\
\leq &C(t-s)^{\frac{1}{2}}.
\end{align}
Therefore, (\ref{edh1}) follows from (\ref{edgdxi}), (\ref{edgammadxi}) and  Meyer's inequality.
\end{proof}

The next lemma shows that $\xi$ is a $d$-dimensional Gaussian process in the whole probability space. Notice that, however, conditionally to $W$, the process $\xi$ is no longer Gaussian, because it is the solution to a nonlinear SDE. 
\begin{lemma}\label{lxigrv}
The process $\xi$ given by equation (\ref{sde})
 is a $d$-dimensional Gaussian process, with mean $x$ and covariance matrix
\begin{align}\label{cvmxi}
\Sigma_{s,t}= (t\wedge s-r) (I+\rho(0)),
\end{align}
where   $\rho(0)$ is defined in (\ref{rho}).
Moreover,  the probability density of $\xi_t$,  denoted by $p_{\xi_t}(y)$, is bounded by a Gaussian density:
\begin{align}\label{pdfid}
p_{\xi_t}(y)\leq \left(2\pi(t-r)\right)^{-\frac{d}{2}}\exp \Big(-\frac{k|x-y|^2}{t-r}\Big),
\end{align}
where 
\begin{align}\label{defk}
k=[2(d\|h\|_{2,3}^2+1)]^{-1}.
\end{align}
\end{lemma}
\begin{proof}
Since $B$ is a $d$-dimensional Brownian motion  and $W$ is a $d$-dimensional space-time white Gaussian random field independent of $B$, then $\xi=\{\xi_t,r\leq t\leq T\}$ is a square integrable $d$-dimensional martingale. The quadratic covariation of $\xi$ is given by
\begin{align}
\langle \xi^i, \xi^j \rangle_t=&\delta_{ij}(t-r)+\sum_{k=1}^d \int_r^t\int_{\mathbb{R}^d}h^{ik}(\xi_s-y)h^{jk}(\xi_s-y)dyds\nonumber\\
=&\left(\delta_{ij}+ \rho^{ij}(0)\right)(t-r).
\end{align}
Note that $\rho(0)$ is a symmetric nonnegative definite matrix. As a consequence, $I+\rho(0)$ is strictly positive definite, and thus nondegenerate. Therefore, we can find a nondegenerate matrix $M$, such that $M^* (I+\rho(0)) M=I$. Let $\eta=M \xi$, then $\eta=\{\eta_t, t\in[0,T]\}$ is a martingale with quadratic covariation
\begin{align*}
\langle\eta^i,\eta^j\rangle_t=(t-r) \sum_{k_1,k_2=1}^d M^{ik_1} M^{jk_2}\langle \xi^{k_1}, \xi^{k_2} \rangle_t=\delta_{ij}(t-r).
\end{align*}
 By Levy's martingale characterization, $\eta$ is a $d$-dimensional Brownian motion. Then, $\xi=M^{-1}\eta$ is a Gaussian process, with covariance matrix (\ref{cvmxi}).

Since for any $t>r$, $\Sigma_t:=\Sigma_{t,t}=(t-r)(I+\rho(0))$ is symmetric and positive definite, the probability density of the Gaussian random vector $\xi_t$ is given by
\begin{align}\label{pdrxi}
p_{\xi_t}(y)=\frac{1}{\sqrt{(2\pi)^d|\Sigma_t|}}\exp\Big(-\frac{1}{2}(y-x)^*\Sigma_t^{-1}(y-x)\Big).
\end{align}
Recall that $\rho(0)$ is symmetric and nonnegative definite. Then it has eigenvalues $\lambda_1\geq \lambda_2\geq  \cdots \ge \lambda_d\geq 0$. Let $\lambda$ be the diagonal matrix with diagonal elements $\lambda_1,\dots, \lambda_d$. There is an orthogonal  matrix $U$, such that $\rho(0)=U^* \lambda U$. Let $k$ be defined in (\ref{defk}). It follows that
\[
\lambda_1+1\leq \sum_{i,j=1}^d |\rho^{ij}(0)|+1\leq \|\rho\|_{\infty}+1\leq d\|h\|_{3,2}^2+1=\frac{1}{2k}.
\]
Thus for any nonzero $x\in\mathbb{R}^d$, we have
\begin{align*}
\frac{1}{2}x^*\Sigma_t^{-1}x-\frac{k}{t-r}x^*x=&\frac{1}{2}x^*\Big(\Sigma_t^{-1}-\frac{2k}{t-r} I\Big)x\\
=&\frac{1}{2(t-r)}x^*U^*\left( \left(I+\lambda\right)^{-1} -2k I\right)Ux\geq 0,
\end{align*}
because $\left(I+\lambda\right)^{-1} -2k I$ is a nonnegative diagonal matrix. Thus for any $x,y\in\mathbb{R}^d$, $t>r$, we have
\begin{align}\label{expi}
\exp\Big(-\frac{1}{2}(y-x)^*\Sigma_t^{-1}(y-x)\Big)\leq \exp \Big(-\frac{k|x-y|^2}{t-r}\Big),
\end{align}
 On the other hand, we have
\begin{align}\label{dcvi}
|\Sigma_t|=\left|U^*\left(I+\lambda\right)U(t-r)\right|\geq (t-r)^d.
\end{align}
Therefore, we obtain (\ref{pdfid}) by plugging (\ref{expi}) - (\ref{dcvi}) into (\ref{pdrxi}).
\end{proof}

Denote by $\mathbb{P}^W$, $\E^W$, and $\|\cdot\|_p^W$ the probability, expectation and $L^p$-norm conditional on $W$. The following two propositions are estimates for the conditional distribution of $\xi$.
\begin{proposition}\label{cmijt}
For any $0\leq r<t\leq T$, $c>0$, choose $\rho\in(0, c\sqrt{t-r}]$. Then, for any $p_1,p_2\geq 1$ and $y\in\R^d$, there exist $C>0$, depending on $p_1$, $p_2$, $c$, $\|h\|_{2}$, and $d$, such that
\begin{align}\label{cmiji}
\left\|\mathbb{P}^W(|\xi_t-y|\leq \rho)^{\frac{1}{p_1}}\right\|_{p_2}\leq C\exp\Big(-\frac{k |x-y|^2}{p(t-r)}\Big),
\end{align}
where $k$ is defined in (\ref{defk}) and $p=p_1\vee p_2$.
\end{proposition}
\begin{proof}
Let $p=p_1\vee p_2$. Then, by Jensen's inequality, we have
\begin{align*}
\left\|\mathbb{P}^W(|\xi_t-y|\leq \rho)^{\frac{1}{p_1}}\right\|_{p_2}=\left\|\left\|\1_{\{\xi_t-y|\leq \rho\}}\right\|^W_{p_1}\right\|_{p_2}\leq \left\|\1_{\{\xi_t-y|\leq \rho\}}\right\|_p,
\end{align*}
We consider two different cases.

{\bf (i)} Suppose that $2\rho\leq |x-y|$. If $|\xi_t-y|\leq\rho\leq c\sqrt{t-r}$, then
\begin{align*}
|\xi_t-x|\geq |x-y|-|\xi_t-y|\geq |x-y|-\rho\geq \frac{|x-y|}{2},
\end{align*}
and equivalently $\{|\xi_t-y|<\rho\}\subset\{|\xi_t-x|\geq \frac{|x-y|}{2}\}$. Then, by Lemma \ref{lxigrv}, we have
\begin{align}\label{emd2}
\left\|\mathbb{P}^W(|\xi_t-y|\leq \rho)^{\frac{1}{p_1}}\right\|_{p_2}= &\left\|\1_{\{|\xi_t-x|\geq \frac{|x-y|}{2}\}\cap\{|\xi_t-y|<\rho\}}\right\|_p\leq C\bigg[V_d\rho^d\sup_{|z-x|\geq\frac{|x-y|}{2}}p_{\xi_t}(z)\bigg]^{\frac{1}{p}}\nonumber\\
\leq&C\left[V_d c^d(2\pi)^{-\frac{d}{2}}\exp\Big(-\frac{k|x-y|^2}{ t-r}\Big) \right]^{\frac{1}{p}}
\end{align}
where $V_d=\frac{\pi^{\frac{d}{2}}}{\Gamma(1+\frac{d}{2})}$ is the volume of the unit sphere in $\R^d$.

{\bf (ii)} On the other hand, suppose that $2\rho>|x-y|$. Then $|x-y|\leq 2\rho\leq 2c\sqrt{t-r}$. Thus by Lemma \ref{lxigrv} again, we have
\begin{align}\label{emd3}
\left\|\mathbb{P}^W(|\xi_t-y|\leq \rho)^{\frac{1}{p_1}}\right\|_{p_2}\leq& C\big(V_d \rho^d(2\pi (t-r))^{-\frac{d}{2}} \big)^{\frac{1}{p}}\nonumber\\
\leq &C\big(V_dc^d(2\pi)^{-\frac{d}{2}}\big)^{\frac{1}{p}}\exp\Big(\frac{4kc^2}{p}-\frac{4kc^2}{p}\Big)\nonumber\\
\leq &C\big(V_dc^d(2\pi)^{-\frac{d}{2}}\big)^{\frac{1}{p}}e^{\frac{4kc^2}{p}}\exp\Big(-\frac{k|x-y|^2}{p(t-r)}\Big).
\end{align}
Therefore, (\ref{cmiji}) follows from (\ref{emd2}) - (\ref{emd3}).
\end{proof}

Denote by $p^W(r,x;t,y)$ the transition probability density of $\xi$ conditional on $W$. In other words, $p^W(r,x;t,y)$ is the conditional probability density of $\xi_t=\xi_t^{r,x}$.
\begin{proposition}\label{mectpd}
For any $0\leq r<t\leq T$, $p\geq 1$, and $y\in\R^d$, there exist $C>0$, depending on $T$, $d$, $\|h\|_{3,2}$, $p$, and $q$, such that
\begin{align}\label{edn}
\left\|p^W(r,x;t,y)\right\|_{2p}\leq C\exp\Big(-\frac{k|x-y|^2}{6pd(t-r)}\Big) (t-r)^{-\frac{d}{2}},
\end{align}
where $k$ is defined in (\ref{defk}).
\end{proposition}

\begin{proof}
In order to show this lemma, we apply a density formula based on the Riesz transformation (see Theorem \ref{bedf}), and use the estimate stated in Theorem \ref{tdpdift}. Choose $p_1\in(d,3pd]$, let $p_2=2p_1$, and $p_3=\frac{p_1p_2}{p_2-p_1}=p_2$. Then, by (\ref{de}) 
and H\"{o}lder's inequality, we have
\begin{align}\label{mcdxi1}
\left\|p^W_{\xi_t}(y)\right\|_{2p}\leq C\max_{1\leq i\leq d}\Big\{&\big\|\mathbb{P}^W\left(|\xi_t-y|<2\rho\right)^{\frac{1}{p_2}}\big\|_{6p} \Big\|\big|\|H_{(i)}(\xi_t, 1)\|^W_{p_1}\big|^{d-1}\Big\|_{6p}\nonumber\\
&\times\Big[\frac{1}{\rho}+\left\|\|H_{(i)}(\xi_t, 1)\|^W_{p_2}\right\|_{6p}\Big]\Big\},
\end{align}
By Jensen's inequality, we have for any $1\leq i\leq d$
\begin{align}
\Big\|\big|\left\|H_{(i)}(\xi_t, 1)\right\|^W_{p_1}\big|^{d-1}\Big\|_{6p}\leq\left\|H_{(i)}(\xi_t,1)\right\|_{6p\vee p_1}^{d-1}\leq \left\|H_{(i)}(\xi_t,1)\right\|_{6pd}^{d-1},
\end{align}
and
\begin{align}\label{mchkdq}
\left\|\left\|H_{(i)}(\xi_t, 1)\right\|^W_{p_2}\right\|_{6p}\leq\left\|H_{(i)}(\xi_t,1)\right\|_{6pd}.
\end{align}
Let $\rho=\frac{\sqrt{t-r}}{4}$. (\ref{edn}) is a consequence of (\ref{mcdxi1}) - (\ref{mchkdq}), Lemma \ref{ehn}, and Proposition \ref{cmijt}.
\end{proof}

\section{A conditional convolution representation}
In this section, we follow the idea of Li et. al. (see Section 3 of \cite{ptrf-12-li-wang-xiong-zhou}) to obtain a conditional convolution formulation of the SPDE (\ref{dqmvp}). Consider the following SPDE:
\begin{align}\label{crd}
u_t(x)=\int_{\mathbb{R}^d}\mu(z)p^W(0,z;t,x)dz+\int_0^t\int_{\mathbb{R}^d}p^W(r,z;t,x) u_r(z)V(dr,dz),
\end{align}
where $W$ and $V$ are the same random fields as in (\ref{dqmvp}), $p^W$ is the transition density of $\xi_t$ given by (\ref{sde}) conditional on $W$.
\begin{definition}\label{strsol}
A random field $u=\{u_t(x),t\in[0,T],x\in\R^d\}$ that is jointly measurable and adapted to the filtration generated by $W$ and $V$, is said to be a strong solution to the SPDE (\ref{crd}), if the stochastic integral in (\ref{crd}) is defined as Walsh's integral and the equality holds almost surely for almost every $t\in[0,T]$ and $x\in\R^d$.
\end{definition}

\begin{lemma}\label{eumcr}
Assume that $\kappa$ and $\mu$ are bounded. Then the SPDE (\ref{crd}) has a unique strong solution (in the sense of Definition \ref{strsol}), denoted by $u=\{u_t(x), 0\leq t\leq T,x\in\R^d\}$, such that 
\begin{align}\label{unibd}
\sup_{0\leq t\leq T}\sup_{x\in\mathbb{R}^d}\|u_t(x)\|_{2p}<\infty,
\end{align}
for any $p\geq 1$.
\end{lemma}

\begin{proof}
We prove the lemma by the Picard iteration. Let $u_0(t,x)\equiv \mu(x)$ and
\begin{align*}
u_{n}(t,x)=\int_{\mathbb{R}^d}\mu(z)p^W(0,z;t,x)dz+\int_0^t\int_{\mathbb{R}^d}p^W(r,z;t,x) u_{n-1}(r,z)V(dr,dz),
\end{align*}
for all $n\geq 1$ and $0\leq t\leq T$. Let $d_n$ be the difference of $u_n$, that is
\begin{align*}
d_n(t,x):=u_{n+1}(t,x)-u_n(t,x)=\int_0^t\int_{\mathbb{R}^d}p^W(r,z;t,x) \left(d_{n-1}(r,z)\right)V(dr,dz).
\end{align*}
We write
 \[
d^*_n(t):=\sup_{x\in\mathbb{R}^d}\left\|d_n(t,x)\right\|_{2p}^2.
\]
Then, by Burkholder-Davis-Gundy's and Minkowski's inequalities, we have
\begin{align}\label{medu}
d_n^*(t)\leq &c_p \|\kappa\|_{\infty}\sup_{x\in\mathbb{R}^d}\int_0^t\Big(\int_{\mathbb{R}^d}\left\|p^W(r,z;t,x)d_{n-1}(r,z)\right\|_{2p}dz\Big)^2dr.
\end{align}
By the Markov property, $p^{W}(r,z;t,x)$ depends on $\{W(s,z)-W(r,z), s\in(r,t], z\in\R^d\}$. On the other hand, $d_{n-1}(r,z)$ depends on $V$ and $\{W(s,z),s\in[0,r],z\in\R^d\}$. Thus, $p^{W}(r,z;t,x)$ and $d_{n-1}(r,z)$ are independent. That implies
\begin{align}\label{epp2d2}
\E\big(|p^W(r,z;t,x)d_{n-1}(r,z)|^{2p}\big)=\E\big(|p^W(r,z;t,x)|^{2p}\big)\E\big(|d_{n-1}(r,z)|^{2p}\big)
\end{align}
Then, by (\ref{medu}), (\ref{epp2d2}) and Proposition \ref{mectpd}, we have
\begin{align}\label{medu0}
d^*_n(t)\leq &c_p \|\kappa\|_{\infty}\int_0^td^*_{n-1}(r)\sup_{x\in\mathbb{R}^d}\Big(\int_{\mathbb{R}^d}\left\|p^W(r,z;t,x)\right\|_{2p}dz\Big)^2dr\leq C\int_0^td^*_{n-1}(r)dr,
\end{align}
where $C>0$ depends on $T$, $d$, $h$, $p$, and $\|\kappa\|_{\infty}$. Thus by iteration, we have
\begin{align}\label{sdspdef}
d^*_n(t)\leq C^n \int_0^t\int_0^{r_n}\cdots\int_0^{r_2}d_0^*(r_1)  dr_1\cdots dr_n,
\end{align}
To estimate $d_0^*$, we observe that
\begin{align}\label{duigdd}
d_0^*(t)=&\sup_{x\in\mathbb{R}^d}\Big\|\int_{\mathbb{R}^d}\left(\mu(z)-\mu(x)\right)p^W(0,z;t,x)dz+\int_0^t\int_{\mathbb{R}^d}p^W(r,z;t,x) \mu(z)V(dr,dz)\Big\|_{2p}^2\nonumber\\
\leq &\|\mu\|_{\infty}^2\sup_{x\in\mathbb{R}^d}\bigg(2\Big\|\int_{\mathbb{R}^d}p^W(0,z;t,x)dz\Big\|_{2p}+\Big\|\int_0^t\int_{\mathbb{R}^d}p^W(r,z;t,x) V(dr,dz)\Big\|_{2p}\bigg)^2\nonumber\\
:=&\|\mu\|_{\infty}^2\sup_{x\in\mathbb{R}^d}\left(2J_1+J_2\right)^2.
\end{align}
For $J_1$, by Minkowski's inequality and Proposition \ref{mectpd}, we have
\begin{align}\label{duigd1}
J_1\leq \sup_{x\in\mathbb{R}^d}\int_{\mathbb{R}^d}\left\|p^W(0,z;t,x)\right\|_{2p}dz\leq C.
\end{align}
For $J_2$, we apply Burkholder-Davis-Gundy's, Minkowski's inequalities and Proposition \ref{mectpd}. Then,
\begin{align}\label{duigd2}
J_2\leq&c_p\Big\|\int_0^t\int_{\mathbb{R}^d\times \mathbb{R}^d}\kappa(y,z) p^W(r,y;t,x) p^W(r,z;t,x) dydzdr\Big\|_p^{\frac{1}{2}}\\
\leq&c_p\|\kappa\|_{\infty}^{\frac{1}{2}}\bigg[\int_0^t\Big(\int_{\mathbb{R}^d}\left\|p^W(r,y;t,x) \right\|_{2p}dy\Big)^2dr\bigg]^{\frac{1}{2}}\leq  C
\end{align}
Combining (\ref{duigdd}) - (\ref{duigd2}), it follows that $d^*_0(t)<C$. As a consequence, we have
\begin{align}\label{sdspde0}
d^*_n(t)\leq &C \int_0^t\int_0^{r_n}\dots\int_0^{r_2} 1dr_1\dots dr_n=C\frac{t^n}{n!},
\end{align}
that is summable in $n$. Therefore, for any fixed $t\in[0,T]$ and $x\in\R^d$, $\{u_n(t,x)\}_{n\geq 0}$ is convergent in $L^{2p}(\Omega)$. Denote by $u_t(x)$ the limit of this sequence.

We claim that $u=\{u_t(x),t\in[0,T], x\in\R^d\}$ is a strong solution to (\ref{crd}). It suffices to show that as $n\to \infty$,
\begin{align}\label{l2cvg2}
\int_0^t\int_{\mathbb{R}^d}p^W(r,z;t,x) u_n(r,z)V(dr,dz)\to \int_0^t\int_{\mathbb{R}^d}p^W(r,z;t,x) u(r,z)V(dr,dz)
\end{align}
in $L^{2p}(\Omega)$ for almost every $t\in[0,T]$ and $x\in\R^d$. Actually, by Burkholder-Davis-Gundy's and Minkowski's inequalities, and the fact that $\{p^W(r,z;t,x),x,z\in\R^d\}$ and $\{u_n(r,z)-u(r,z),z\in\R^d\}$ are independent, we have
\begin{align*}
&\Big\|\int_0^t\int_{\mathbb{R}^d}p^W(r,z;t,x) \left(u_n(r,z)-u(r,z)\right)V(dr,dz)\Big\|_{2p}^2\\
\leq& \|\kappa\|_{\infty}\sup_{r\in[0,T]}\sup_{z\in\mathbb{R}^d}\left\|u_n(r,z)-u(r,z)\right\|_{2p}^2\int_0^t\Big(\int_{\mathbb{R}^d}\left\|p^W(r,z;t,x)\right\|_{2p}dz\Big)^2dr.
\end{align*}
The integral on the right-hand side of above inequality is finite, and because of (\ref{sdspde0}), as $n\to \infty$,
\begin{align*}
\sup_{r\in[0,T]}\sup_{z\in\mathbb{R}^d}\left\|u_n(r,z)-u(r,z)\right\|_{2p}^2\leq \sum_{k=n}^{\infty}d_k^*(T)\to 0,
\end{align*}
This implies that (\ref{l2cvg2}) is true.

In order to show the uniqueness, we assume that $v=\{v_t(x), t\in [0,T], x\in\R^d\}$ is another strong solution to (\ref{crd}). Let $d_t(x)=u_t(x)-v_t(x)$ for any $t\in[0,T]$ and $x\in\R^d$. Then,
\[
d_t(x)=\int_0^t\int_{\R^d}p^W(r,z;t,x)d_r(z)V(dr,dz).
\]
By Burkholder-Davis-Gundy's and Minkowski's inequalities and the fact that the families $\{d_r(x), x\in\R^d\}$ and $\{p^W(r,z;t,x),x,z\in\R^d\}$ are independent, we have
\begin{align}\label{unq}
\sup_{x\in\R^d}\|d_t(x)\|_{2p}^2\leq &\int_0^t\sup_{x\in\R^d}\|d_r(x)\|_{2p}^2\Big(\int_{\R^d}\left\|p^W(r,z;t,x)\right\|_{2p}dz\Big)^2dr\nonumber\\
\leq &C\int_0^t\sup_{x\in\R^d}\|d_r(x)\|_{2p}^2dr.
\end{align}
Due to Gr\"{o}nwall's lemma and the fact that $d_0\equiv 0$, (\ref{unq}) implies  $d(t,x)\equiv 0$, a.s. It follows that the solution to (\ref{crd}) is unique.

For the uniformly boundedness, let $u^*(t)=\sup_{x\in\mathbb{R}^d}\left\|u_t(x)\right\|_{2p}^2$.  We can show that
\begin{align*}
u^*(t)\leq &2\|\mu\|_{\infty}^2\Big(\sup_{x\in\mathbb{R}^d}\int_{\mathbb{R}^d}\left\|p^W(0,z;t,x)\right\|_{2p}dz\Big)^2\\
&+2\|\kappa\|_{\infty}\int_0^tu^*(r)\Big(\sup_{x\in\mathbb{R}^d}\int_{\mathbb{R}^d}\left\|p^W(r,z;t,x)\right\|_{2p}dz\Big)^2dr\\
\leq& c_1+c_2\int_0^tu^*(r)dr
\end{align*}
Then, (\ref{unibd}) is a consequence of Gr\"{o}nwall's lemma.
\end{proof}

\begin{proposition}
Assume that $\kappa$ and $\mu$ are bounded. Let $u=\{u_t(x),0< t\leq T, x\in\mathbb{R}^d\}$ be the unique strong solution to (\ref{crd}) in the sense of Definition \ref{strsol}. Then, $u$ is the strong solution to  (\ref{dqmvp}) in the sense of Definition \ref{def}.
\end{proposition}
\begin{proof}
Let $u=\{u_t(x),t\in[0,T],x\in\R^d\}$ be the unique solution to the SPDE (\ref{crd}), and write $Z(dt, dx)=u_t(x)V(dt,dx)$ for all $t\in[0,T]$ and $x\in\R^d$. Then, it suffices to show that $u$ satisfies the following equation:
\begin{align}\label{futf0}
\langle u_t, \phi \rangle=&\langle \mu, \phi\rangle+\int_0^t \langle u_s, A\phi\rangle ds+\int_0^t \int_{\mathbb{R}^d}\langle u_s, \nabla \phi^* h(y-\cdot) \rangle W(ds, dy)\nonumber\\
&+\int_0^t\int_{\mathbb{R}^d}\phi(x)Z(ds,dx),
\end{align}
for any $\phi\in C^2_b\left(\mathbb{R}^d\right)$.

Denote by 
\[
\E^W_{s,x}(\phi(\xi_t)):=\E\big(\phi(\xi_t)|W, \xi_s=x\big)=\int_{\R^d}\phi(z)p^W(s,x;t,z)dz.
\]
As $u$ is the strong solution to (\ref{crd}), the following equations are satisfied
\begin{align*}
\langle u_t, \phi\rangle=\left\langle \mu, \E^W_{0,\cdot}(\phi(\xi_t)) \right\rangle+\int_0^t\int_{\mathbb{R}^d}\E^W_{s,z}(\phi(\xi_t))Z(ds,dz),
\end{align*}
\begin{align*}
\int_0^t \langle u_s, A\phi\rangle ds=\int_0^t\left\langle \mu, \E^W_{0,\cdot}(A\phi(\xi_s)) \right\rangle ds+\int_0^t\int_0^s\int_{\mathbb{R}^d}\E^W_{r,z}(A\phi(\xi_s))Z(dr,dz)ds,
\end{align*}
and
\begin{align*}
\int_0^t \int_{\mathbb{R}^d}&\left\langle u_s, \nabla \phi^* h(y-\cdot) \right\rangle W(ds, dy)=\int_0^t \int_{\mathbb{R}^d}\left\langle \mu, \E^W_{0,\cdot}\big(\nabla \phi(\xi_s)^*h(y-\xi_s)\big) \right\rangle W(ds, dy)\\
&\ +\int_0^t\int_{\mathbb{R}^d}\int_0^s\int_{\mathbb{R}^d}\E^W_{r,z}\big((\nabla \phi(\xi_s)^*h(y-\xi_s)\big)Z(dr,dz) W(ds, dy).
\end{align*}
Thus by stochastic Fubini's theorem, we have
\begin{align}\label{futf}
&\langle u_t, \phi \rangle-\langle \mu, \phi\rangle-\int_0^t \langle u_s, A\phi\rangle ds-\int_0^t \int_{\mathbb{R}^d}\langle u_s, \nabla \phi^* h(y-\cdot) \rangle W(ds, dy)\\
=&\bigg\langle \mu, \E^W_{0,\cdot}\Big(\phi(\xi_t)-\phi(\xi_0)-\int_0^t A\phi(\xi_s)ds -\int_0^t \int_{\mathbb{R}^d} \nabla \phi(\xi_s)^*h(y-\xi_s)W(ds,dy)\Big)\bigg\rangle\nonumber\\
&+\int_0^t\int_{\mathbb{R}^d}\E^W_{s,z}\Big(\phi(\xi_t)-\int_s^tA\phi(\xi_r)dr-\int_s^t\int_{\mathbb{R}^d}\nabla \phi(\xi_r)^*h(y-\xi_r) W(dr, dy)\Big)Z(ds,dz).\nonumber
\end{align}
Notice that by It\^{o}'s formula, we have
\begin{align}\label{fxiito}
\phi(\xi_t^{s,x})=&\phi(x)+\int_s^t A \phi(\xi^{s,x}_r)dr+\int_s^r\nabla \phi(\xi^{s,x}_r)^*dB_r\nonumber\\
&+\int_s^t \int_{\mathbb{R}^d}\nabla \phi(\xi^{s,x}_r)^* h(y-\xi^{s,x}_r)W(dr,dy).
\end{align}
Then, (\ref{futf0}) follows from (\ref{futf}) and (\ref{fxiito}).
\end{proof}

\section{Proof of Theorem \ref{tjhc}}
In this section, we prove Theorem \ref{tjhc} by showing 
the the H\"{o}lder continuity of $u_t(x)$ in spatial and time variables separately:
\begin{proposition}\label{phcs}
Suppose that $h\in H^2_3\left(\mathbb{R}^d\right)$, $\|\kappa\|_{\infty}<\infty$, and $\mu\in L^1\left(\mathbb{R}^d\right)$ is bounded. Then, for any $0< s<t\leq T$, $x,y\in\R^d$ $\beta\in(0,1)$ and $p>1$, there exists a constant $C$ depending on $T$, $d$, $\|h\|_{3,2}$, $\|\mu\|_{\infty}$, $\|\kappa\|_{\infty}$, $p$, and $\beta$, such that the following inequalities are satisfied:
\begin{align}
\left\|u_t(y)-u_t(x)\right\|_{2p}\leq &Ct^{-\frac{1}{2}}(y-x)^{\beta},\label{mpdhcsv}\\
\left\|u_t(x)-u_s(x)\right\|_{2p}\leq &Cs^{-\frac{1}{2}}(t-s)^{\frac{1}{2}\beta}.\label{mpdhctv}
\end{align}
\end{proposition}
Then, Theorem \ref{tjhc} is simply a  corollary of Proposition \ref{phcs}. In order to prove Proposition \ref{phcs}, we need the following H\"{o}lder continuity results for the conditional transition density $p^W(r,z;t,x)$: 
\begin{lemma}\label{lmdtpirl}
Suppose that $h\in H_2^3(\mathbb{R}^d)$, $0\leq r<s<t\leq T$, $x,y\in\mathbb{R}^d$, and $\beta\in(0,1)$. Then, there exists $C>0$, depending on $T$, $d$, $\|h\|_{3,2}$, $p$ and $\beta$, such that the following inequalities are satisfied:
\begin{align}
\int_{\mathbb{R}^d}\left\|p^W(r,z;t,y)-p^W(r,z;t,x)\right\|_{2p}dz\leq &C(t-r)^{-\frac{1}{2}\beta}\left|y-x\right|^{\beta},\label{mdtpirl}\\
\int_{\mathbb{R}^d}\left\|p^W(r,z;t,x)-p^W(r,z;s,x)\right\|_{2p}dz\leq &C (s-r)^{-\frac{1}{2}\beta}(t-s)^{\frac{1}{2}\beta}.\label{idtmpd}
\end{align}
\end{lemma} 

Before the proof, let us firstly  derive a variant of the density formula (\ref{dfrk}). It will be used in the proof of (\ref{idtmpd}). Choose $\phi\in C^2_b\left(\mathbb{R}^n\right)$, such that $\1_{B(0,1)}\leq \phi\leq \1_{B(0,4)}$, and its first and second partial derivatives are all bounded by $1$. For any $x\in\R^d$ and $\rho>0$, we set $\phi^x_{\rho}:=\phi(\frac{\cdot -x}{\rho})$. Assume that $F$ satisfies all the properties in Theorem \ref{bedf}. Let $Q_n$ be the $n$-dimensional Poisson kernel (see (\ref{posnkn})). Then, the density of $F$ can be represented as follows:
\begin{align}\label{dfrkip}
p_F(x)=&\sum_{i,j_1,j_2=1}^n \E\left[\partial_{j_1} Q_n(F-x)\left\langle DF^{j_1}, DF^{j_2}\right\rangle_H \sigma^{j_2i} H_{(i)}(F, \phi_{\rho}^x(F))\right]\nonumber\\
=& \E\bigg[\Big\langle DQ_n(F-x), \sum_{i,j_2=1}^mH_{(i)}(F, \phi_{\rho}^x(F))\sigma^{j_2i} DF^{j_2} \Big\rangle_H\bigg]\nonumber\\
=&\sum_{i=1}^m \E\Big[Q_n(F-x)\sum_{j_2=1}^m\delta\left[H_{(i)}(F, \phi_{\rho}^x(F))\sigma^{j_2i}DF^{j_2}\right]\Big]\nonumber\\
=&-\sum_{i=1}^m \E\big[Q_n(F-x)H_{(i,i)}(F, \phi_{\rho}^x(F))\big].
\end{align}

\begin{proof}[Proof of Lemma \ref{lmdtpirl}]  
Let $\xi_t=\xi_t^{r,z}$ be defined in (\ref{sde}). \\
{\bf (i)} Choose $p_1\in(d,3pd]$, let $p_2=2p_1$, and $p_3=\frac{p_1p_2}{p_2-p_1}=p_2$. Then, by (\ref{dmvte}) and H\"{o}lder's inequality, for any fixed $z, x,y\in\mathbb{R}^d$ and $\rho>0$, we can show that
\begin{align*}
I(z):=&\|p^W(r,z;t,x)-p^W(r,z;t,y)\|_{2p}\\
\leq &C|y-x|\left\|\mathbb{P}^W\left(\xi_t-\tau\leq 4\rho\right)^{\frac{1}{p_2}}\right\|_{6p}\max_{1\leq i\leq d}\Big\{\Big\|\big|\|H_{(i)}(\xi_t;1)\|_{p_2}^W\big|^{d-1}\Big\|_{6p}\\
&\times\Big(\frac{1}{\rho^2}+\frac{2}{\rho}\left\|\|H_{(i)(\xi_t;1)}\|_{p_2}^{W}\right\|_{6p}+\left\|\|H_{(i,j)}(\xi_t;1)\|_{p_2}^W\right\|_{6p}\Big)\Big\},
\end{align*}
where $\tau=cx+(1-c)y$, for some $c\in(0,1)$ that depends on $z, x, y$.

Let $\rho=\frac{\sqrt{t-r}}{8}$. Similarly as proved in Proposition \ref{mectpd}, we can show that
\begin{align}\label{dmitgd}
I(z)\leq& C |y-x|(t-r)^{-\frac{d+1}{2}}\exp\Big(-\frac{k|\tau-z|^2}{(6p\vee p_2)(t-r)}\Big)\nonumber\\
\leq &C |y-x|(t-r)^{-\frac{d+1}{2}}\exp\Big(-\frac{k|\tau-z|^2}{6pd(t-r)}\Big),
\end{align}
where $k$ is defined in (\ref{defk}) and $C>0$ depends on $T$, $d$, $p$, and $\|h\|_{3,2}$.

Notice that even if we fix $x,y\in \R^d$, $\tau$ is still a function of $z$ that does not have an explicit formulation. Thus it is not easy to calculate the integral of $I$ directly. Without losing generality, assume that $x=0$, and $y=(y_1, 0,\dots, 0)$, where $y_1\geq 0$. Then $\tau=((1-c)y_1, 0, \dots, 0)$, where $c=c(z)\in(0, 1)$. Let $\widehat{k}=\frac{k}{6pd}$. For any $z=(z_1,\dots, z_d)\in \R^d$, we consider the following cases.

(a) If $z_1\leq 0$, then
\begin{align}\label{2exp1}
\exp\Big(-\frac{k|\tau-z|^2}{6pd(t-r)}\Big)\leq \exp\Big(-\frac{\widehat{k}|z|^2}{t-r}\Big).
\end{align}
(b) If $z_1\geq y_1$, then
\begin{align}\label{2exp2}
\exp\Big(-\frac{k|\tau-z|^2}{6pd(t-r}\Big)\leq \exp\Big(-\frac{\widehat{k}|y-z|^2}{t-r}\Big).
\end{align}
(c) If $0< z_1< y_1$, then
\begin{align}\label{2exp3}
\exp\Big(-\frac{k|\tau-z|^2}{6pd(t-r}\Big)\leq \exp\Big(-\frac{\widehat{k}|\tau_0-z|^2}{t-r}\Big),
\end{align}
where $\tau_0=(z_1, 0, \dots, 0)$. 

Therefore, combining (\ref{dmitgd}) - (\ref{2exp3}), we have
\begin{align}\label{imdsv}
\int_{\mathbb{R}^d}I(z)dz\leq C |y-x|(t-r)^{-\frac{d+1}{2}}\left(I_1+I_2+I_3\right),
\end{align}
where
\begin{align*}
I_1=\int_{-\infty}^0dz_1\int_{\mathbb{R}^{d-1}}\exp\Big(-\frac{\widehat{k}|z|^2}{t-r}\Big)dz_d\dots dz_2,\\
I_2=\int_{|y|}^{\infty}dz_1\int_{\mathbb{R}^{d-1}}\exp\Big(-\frac{\widehat{k}|y-z|^2}{t-r}\Big)dz_d\dots dz_2,\\
I_3=\int_0^{|y|}dz_1 \int_{\mathbb{R}^{d-1}}\exp\Big(-\frac{\widehat{k}|\tau_0-z|^2}{t-r}\Big)dz_d\dots dz_2.
\end{align*}
By a changing of variables, it is easy to show that
\begin{align}\label{imdsv12}
I_1+I_2=\int_{\mathbb{R}^d}\exp\Big(-\frac{\widehat{k}|z|^2}{t-r}\Big)dz=\widehat{k}^{-\frac{d}{2}}(t-r)^{\frac{d}{2}}.
\end{align}
For $I_3$, we compute the integral as follows:
\begin{align}\label{imdsv3}
I_3=&\int_0^{|y|}dz_1\int_{\mathbb{R}^{d-1}}\exp\Big(-\frac{\widehat{k}\left(z_2^2+\dots z_d^2\right)}{t-r}\Big)dz_d\dots dz_2\nonumber\\
=&\big(2\pi\widehat{k}^{-1}\big)^{\frac{d-1}{2}}(t-r)^{\frac{d-1}{2}}|y|.
\end{align}
Thus combining (\ref{imdsv}) - (\ref{imdsv3}), we have
\begin{align}\label{mdtpirl1}
\int_{\mathbb{R}^d}I(z)dz\leq &C\big[(t-r)^{-\frac{1}{2}}|y|+(t-r)^{-1}|y|^2\big] \nonumber\\
=&C\big[(t-r)^{-\frac{1}{2}}|y-x|+(t-r)^{-1}|y-x|^2\big].
\end{align}
It is easy to see that the inequality (\ref{mdtpirl1}) holds for all $x,y\in\R^d$.

On the other hand, by Proposition \ref{mectpd}, we have
\begin{align}\label{mdtpirl2}
\int_{\mathbb{R}^d}I(z)dz\leq\int_{\mathbb{R}^d}\|p^W(r,z;t,y)\|_{2p}+\|p^W(r, z;t,x)\|_{2p}dz\leq C.
\end{align}
Therefore by (\ref{mdtpirl1}) and (\ref{mdtpirl2}), for any $\beta_1,\beta_2\in(0,1)$, we have
\begin{align*}
\int_{\mathbb{R}^d}I(z)dz\leq C\big[(t-r)^{-\frac{1}{2}\beta_1}\left|y-x\right|^{\beta_1}+(t-r)^{-\beta_2}\left|y-x\right|^{2\beta_2}\big]
\end{align*}
Then, (\ref{mdtpirl}) follows by choosing $\beta=\beta_1=2\beta_2$.

{\bf (ii)} Let $\rho_1=\sqrt{t-r}$ and $\rho_2=\sqrt{s-r}$. By density formula (\ref{dfrkip}), we have
\begin{align}\label{dcdtv}
&\left|p^W(r,z;t,x)-p^W(r,z;s,x)\right|\nonumber\\
\leq &\sum_{i=1}^d\Big|\E^W\left\{\left[Q_d(\xi_t-x)-Q_d(\xi_s-x)\right]H_{(i,i)}(\xi_s,\phi_{\rho_2}^x(\xi_s))\right\}\Big|\nonumber\\
&+\sum_{i=1}^d\Big|\E^W\left\{Q_d(\xi_t-x)\left[H_{(i,i)}(\xi_t,\phi_{\rho_1}^x(\xi_t))-H_{(i,i)}(\xi_s,\phi_{\rho_2}^x(\xi_s))\right]\right\}\Big|\nonumber\\
=&I_1+I_2.
\end{align}

Estimation for $I_1$: Note that by  the local property of $\delta$ (see Proposition 1.3.15 of Nulart \cite{springer-06-nualart}), $H_{(i,i)}(\xi_s,\phi_{\rho_2}^x(\xi_s))$ vanishes except if $\xi_s\in B(x,4\rho_2)$. Choose $p_1\in (d, 2pd]$. Let $p_2=3p_1$ and $p_3=\frac{3p_1}{3p_1-2}$. Then, $\frac{2}{p_2}+\frac{1}{p_3}=1$. Thus, by H\"{o}lder's inequality, we have
\begin{align}\label{medcdtv1p}
\left\|I_1\right\|_{2p}\leq &d\big\|\|\1_{B(x, 4\rho_2)}(\xi_s)\|^W_{p_2}\big\|_{6p}\big\|\|Q_d(\xi_t-x)-Q_d(\xi_s-x)\|^W_{p_3}\big\|_{6p}\nonumber\\
&\times \max_{1\leq i\leq d}\big\|\|H_{(i,i)}(\xi_s,\phi_{\rho_2}^x(\xi_s))\|^W_{p_2}\big\|_{6p}.
\end{align}
By Proposition \ref{cmijt}, and the fact that $p_2=3p_1\leq 6pd$, the first factor satisfies the following inequality
\begin{align}\label{mpbxis}
\big\|\|\1_{B(x, 4\rho_2)}(\xi_s)\|^W_{p_2}\big\|_{6p}=\big\|\mathbb{P}^W(|\xi_s-x|<4\rho_2)^{\frac{1}{p_2}}\big\|_{6p}\leq C\exp\Big(-\frac{k|z-x|}{6pd(s-r)}\Big).
\end{align}
By Lemmas \ref{ehn} and \ref{ipfpd}, for all $1\leq i\leq d$, the last factor can be estimated as follows:
\begin{align}\label{h2xis}
\big\|\|H_{(i,i)}(\xi_s,\phi_{\rho_2}^x(\xi_s))\|^W_{p_2}\big\|_{6p}\leq &\frac{1}{\rho_2^2}+\frac{2}{\rho_2}\big\|\|H_{(i)}(\xi_s,1)\|^W_{p_2}\big\|_{6p}+\big\|\|H_{(i,i)}(\xi_s,1)\|^W_{p_2}\big\|_{6p}\nonumber\\
\leq & C(s-r)^{-1}.
\end{align}
We estimate the second factor by the mean value theorem. Let $\eta_1=|\xi_t-x|$ and $\eta_2=|\xi_s-x|$. Then, we can write
\begin{equation*}
Q_d(\xi_t-x)-Q_d(\xi_s-x)=\begin{cases}
A_2^{-1}\left(\log \eta_1 -\log \eta_2\right), &\text{if}\ d=2, \\
-A_d^{-1}\big[\eta_1^{-(d-2)}-\eta_2^{-(d-2)}\big], &\text{if}\ d\geq 3.
\end{cases}
\end{equation*}
Thus, by the mean value theorem, it follows that
\begin{align*}
\left|Q_d(\xi_t-x)-Q_d(\xi_s-x)\right|=\frac{c_d|\eta_1-\eta_2|}{|\zeta\eta_1+(1-\zeta)\eta_2|^{d-1}},
\end{align*}
where $c_d$ is a constant coming from the Poisson kernel, and $\zeta\in(0,1)$ is a random number that depends on $\eta_1$ and $\eta_2$. Notice that $f(x)=x^{-(d-1)}$ is a convex function on $(0,\infty)$, and $\mathbb{P}(\eta_1>0)=\mathbb{P}(\eta_2>0)=1$, then we have
\[
|\zeta\eta_1+(1-\zeta)\eta_2|^{-(d-1)}\leq |\zeta\eta_1|^{-(d-1)}+|(1-\zeta)\eta_2|^{-(d-1)},\ a.s.
\]
Let $q=\frac{p_1}{p_1-1}$, then $\frac{1}{q}+\frac{1}{p_2}=\frac{1}{p_3}$. As a consequence of H\"{o}lder's inequality, we have
\begin{align}\label{mdqts0}
&\big\|\|Q_d(\xi_t-x)-Q_d(\xi_s-x)\|^W_{p_3}\big\|_{6p}\leq c_d\bigg\|\Big\|\frac{|\eta_1-\eta_2|}{|\zeta\eta_1+(1-\zeta)\eta_2|^{d-1}}\Big\|^W_{p_3}\bigg\|_{6p}\\
&\qquad\leq C \big\|\|\eta_1-\eta_2\|_{p_2}^W\big\|_{12p}\Big\|\left\|\zeta\eta_1+(1-\zeta)\eta_2|^{-(d-1)}\right\|^W_q\Big\|_{12p}\nonumber\\
&\qquad\leq  C \|\eta_1-\eta_2\|_{12pd}\Big[\Big\|\big\|\zeta\eta_1^{-(d-1)}\big\|^W_q\Big\|_{12p}+\Big\|\big\|(1-\zeta)\eta_2^{-(d-1)}\big\|^W_q\Big\|_{12p}\Big]\nonumber\\
&\qquad\leq  C \big\||\xi_t-\xi_s|\big\|_{12pd}\Big[\Big\|\big\||\xi_t-y|^{-(d-1)}\big\|^W_q\Big\|_{12p}+\Big\|\big\||\xi_s-y|^{-(d-1)}\big\|^W_q\Big\|_{12p}\Big].\nonumber
\end{align}
The negative moments of $\xi_t-y$ can be estimated by (\ref{ehxi}), Jensen's inequality, and Lemma \ref{blerk}:
\begin{align}\label{mdqts1}
\Big\|\big\||\xi_t-x|^{-(d-1)}\big\|^W_{q}\Big\|_{12p}\leq &C\max_{1\leq i\leq d}\big\|\big|\|H_{i}(\xi_t,1)\|_{p_1}^W\big|^{d-1}\big\|_{12p}\nonumber\\
\leq &C \max_{1\leq i\leq d}\left\|H_{(i)}(\xi_t, 1)\right\|_{12pd}^{d-1}
\leq C(t-r)^{-\frac{d-1}{2}}.
\end{align} 
Then, by (\ref{mdqts0}) - (\ref{mdqts1}), we have
\begin{align}\label{mdqts}
\big\|\|Q_d(\xi_t-x)-Q_d(\xi_s-x)\|^W_{p_3}\big\|_{6p}\leq C(t-s)^{\frac{1}{2}}(s-r)^{-\frac{d-1}{2}}.
\end{align}
Thus combining (\ref{medcdtv1p}), (\ref{mpbxis}), (\ref{h2xis}) and (\ref{mdqts}), we have
\begin{align*}
\|I_1\|_{2p}\leq C\exp\Big(-\frac{k|z-x|}{6pd(s-r)}\Big)(s-r)^{-\frac{d+1}{2}}(t-s)^{\frac{1}{2}}.
\end{align*}
This implies
\begin{align}\label{medcdtv1}
\int_{\mathbb{R}^d}\|I_1\|_{2p}dz\leq C(s-r)^{-\frac{1}{2}}(t-s)^{\frac{1}{2}}.
\end{align}

Estimates for $I_2$: Recall that $\gamma_t=(\langle D\xi^i, D\xi^j\rangle_H)_{i,j=1}^d=\sigma_t^{-1}$. By computation analogue to (\ref{dfrkip}) going backward, we can show that
\begin{align}\label{medcdtv2}
&\E^W\left[Q_d(\xi_t-x)\left(H_{(i,i)}(\xi_t,\phi_{\rho_1}^x(\xi_t))-H_{(i,i)}(\xi_s,\phi_{\rho_2}^x(\xi_s))\right)\right]\nonumber\\
=& -\sum_{j_1,j_2=1}^d\E^W\left[\partial_{j_2} Q_d(\xi_t-x)\langle D\xi_t^{j_2}, D\xi_t^{j_1}\rangle_H H_{(i)}\left(\xi_t,\phi_{\rho_1}^x(\xi_t)\right)\sigma_t^{j_1i}\right]\nonumber\\
&+\sum_{j_1,j_2=1}^d\E^W\left[\partial_{j_2}Q_d(\xi_t-x)\langle D\xi_t^{j_2}, D\xi_s^{j_1}\rangle_H H_{(i)}\left(\xi_s,\phi_{\rho_2}^x(\xi^{r,z}_s)\right)\sigma_s^{j_1i}\right]\nonumber\\
=&-\E^W\left[\partial_iQ_d(\xi_t-x) \left(H_{(i)}\left(\xi_t,\phi_{\rho_1}^x(\xi_t)\right)- H_{(i)}\left(\xi_s,\phi_{\rho_2}^x(\xi_s)\right)\right)\right]\nonumber\\
&+ \sum_{j_1,j_2=1}^d\E^W\left[\partial_{j_2} Q_d(\xi_t-x)\langle D\xi_t^{j_2}-D\xi_s^{j_2}, D\xi_s^{j_1}\rangle_H H_{(i)}\left(\xi_s,\phi^x_{\rho_2}(\xi_s)\right)\sigma_s^{j_1i}\right]\nonumber\\
:=&J_1+J_2.
\end{align}
By Lemma \ref{ipfpd}, we have
\begin{align}\label{dhtv}
&\left|H_{(i)}\left(\xi_t,\phi_{\rho_1}^x(\xi_t)\right)-H_{(i)}\left(\xi_s,\phi_{\rho_2}^x(\xi_s)\right)\right|\leq \left|\partial_i\phi_{\rho_1}^x(\xi_t)-\partial_i\phi_{\rho_2}^x(\xi_s)\right|\\
&\hspace{20mm} +|\phi_{\rho_2}^x(\xi_s)|\left|H_{(i)}(\xi_t,1)-H_{(i)}(\xi_s,1)\right|+\left|H_{(i)}(\xi_t,1)\right|\left|\phi_{\rho_1}^x(\xi_t)-\phi_{\rho_2}^x(\xi_s)\right|.\nonumber
\end{align}
By the mean value theorem, for some  random numbers $c_1,c_2\in(0,1)$, we have
\begin{align}\label{dphitv}
\left|\phi_{\rho_1}^x(\xi_t)-\phi_{\rho_2}^x(\xi_s)\right|=&\big|\1_{B(x, 4\rho_1)}(\xi_t)\vee \1_{B(x, 4\rho_2)}(\xi_s)\big|\Big|\phi\Big(\frac{\xi_t-x}{\rho_1}\Big)-\phi\Big(\frac{\xi_s-x}{\rho_2}\Big)\Big|\nonumber\\
=&\big|\1_{B(x, 4\rho_1)}(\xi_t)\vee \1_{B(x, 4\rho_2)}(\xi_s)\big|\nonumber\\
&\times\Big|\nabla\phi \Big(c_1\frac{\xi_t-x}{\rho_1}+(1-c_1)\frac{\xi_s-x}{\rho_2}\Big)^*\cdot\Big(\frac{\xi_t-x}{\rho_1}-\frac{\xi_s-x}{\rho_2}\Big)\Big|\nonumber\\
\leq &\big|\1_{B(x, 4\rho_1)}(\xi_t)\vee \1_{B(x, 4\rho_2)}(\xi_s)\big|\Big|\frac{\xi_t-x}{\rho_1}-\frac{\xi_s-x}{\rho_2}\Big|,
\end{align}
and
\begin{align}\label{ddphitv}
&\left|\partial_i\phi_{\rho_1}^x(\xi_t)-\partial_i\phi_{\rho_2}^x(\xi_s)\right|=\Big|\rho_1^{-1}\partial_i\phi\Big(\frac{\xi_t-x}{\rho_1}\Big)-\rho_2^{-1}\partial_i\phi\Big(\frac{\xi_s-x}{\rho_2}\Big)\Big|\\
&\quad\leq\frac{1}{\rho_1}\Big|\nabla\partial_i\phi \Big(c_2\frac{\xi_t-x}{\rho_1}+(1-c_2)\frac{\xi_s-x}{\rho_2}\Big)^*\cdot\Big(\frac{\xi_t-x}{\rho_1}-\frac{\xi_s-x}{\rho_2}\Big)\Big|\\
&\qquad+\Big|\partial_i\phi_{\rho_2}^x(\xi_s)\Big|\Big|\frac{1}{\rho_1}-\frac{1}{\rho_2}\Big|\nonumber\\
&\quad\leq \frac{1}{\rho_1}\big(\1_{B(x, 4\rho_1)}(\xi_t)\vee \1_{B(x, 4\rho_2)}(\xi_s)\big)\Big|\frac{\xi_t-x}{\rho_1}-\frac{\xi_s-x}{\rho_2}\Big|+\1_{B(x, 4\rho_2)}(\xi_s)\Big|\frac{1}{\rho_1}-\frac{1}{\rho_2}\Big|.\nonumber
\end{align}
Choose $q\in(d,3pd]$, let $p_1=\frac{q}{q-1}$, $p_2=2q$, $p_3=4q$. Then, 
\[
\frac{1}{p_1}+\frac{2}{p_2}=\frac{1}{p_1}+\frac{1}{p_2}+\frac{2}{p_3}=1.
\]
 Then, by (\ref{dhtv}) - (\ref{ddphitv}), and H\"{o}lder's inequality, we have
\begin{align}\label{medcdtv21}
\|J_1\|_{2p}\leq&\rho_1^{-1}\big\|\|\partial_iQ_d(\xi_t-x) \|_{p_1}^W\big\|_{6p}\Big\|\big\|\1_{B(x, 4\rho_1)}(\xi_t)\vee \1_{B(x, 4\rho_2)}(\xi_s)\big\|_{p_2}^W\Big\|_{6p}\nonumber\\
&\times \bigg\|\Big\|\Big|\frac{\xi_t-x}{\rho_1}-\frac{\xi_s-x}{\rho_2}\Big|\Big\|_{p_2}^W\bigg\|_{6p}\nonumber\\
&+\big\|\|\partial_iQ_d(\xi_t-x) \|_{p_1}^W\big\|_{6p}\Big\|\big\|\1_{B(x, 4\rho_2)}(\xi_s)\big\|_{p_2}^W\Big\|_{6p}\big\|\|\rho_1^{-1}-\rho_2^{-1}\|_{p_2}\big\|_{6p}\nonumber\\
&+\big\|\|\partial_iQ_d(\xi_t-x)\|_{p_1}^W\big\|_{6p}\Big\|\big\|\1_{B(x, 4\rho_2)}(\xi_s)\big\|_{p_2}^W\Big\|_{6p}\big\|\|H_{(i)}(\xi_t,1)-H_{(i)}(\xi_s,1)\|_{p_2}^W\big\|_{6p}\nonumber\\
&+\big\|\|\partial_iQ_d(\xi_t-x) \|_{p_1}^W\big\|_{6p}\Big\|\big\|\1_{B(x, 4\rho_1)}(\xi_t)\vee \1_{B(x, 4\rho_2)}(\xi_s)\big\|_{p_2}^W\Big\|_{6p}\nonumber\\
&\times \bigg\|\Big\|\Big|\frac{\xi_t-x}{\rho_1}-\frac{\xi_s-x}{\rho_2}\Big|\Big\|_{p_3}^W\bigg\|_{12p}\quad \big\|\|H_{(i)}(\xi_t,1)\|_{p_3}^W\big\|_{12p}\nonumber\\
:=&L_1+L_2+L_3+L_4.
\end{align}
In order to estimate the moments of $\frac{\xi_t-x}{\rho_1}-\frac{\xi_s-x}{\rho_2}$, we rewrite this random vector in the following way:
\[
\frac{\xi_t-x}{\rho_1}-\frac{\xi_s-x}{\rho_2}=\frac{\xi_t-\xi_s}{\rho_1}+\left(\xi_s-z\right)\Big(\frac{1}{\rho_1}-\frac{1}{\rho_2}\Big)+(z-x)\Big(\frac{1}{\rho_1}-\frac{1}{\rho_2}\Big).
\]
It follows that
\begin{align*}
&\Big\|\Big|\frac{\xi_t-x}{\rho_1}-\frac{\xi_s-x}{\rho_2}\Big|\Big\|_{12p\vee p_3}\leq(t-r)^{-\frac{1}{2}}\big\||\xi_t-\xi_s|\big\|_{12pd}\\
&\hspace{15mm}+\frac{(t-r)^{\frac{1}{2}}-(s-r)^{\frac{1}{2}}}{(t-r)^{\frac{1}{2}}(s-r)^{\frac{1}{2}}}\big\||\xi_s-z|\big\|_{12pd}+|z-x|\frac{(t-r)^{\frac{1}{2}}-(s-r)^{\frac{1}{2}}}{(t-r)^{\frac{1}{2}}(s-r)^{\frac{1}{2}}}.
\end{align*}
According to Lemma \ref{lxigrv}, $\xi_t-\xi_s$ and $\xi_s-z$ are Gaussian random vectors with mean $0$, and covariance matrix $(t-s)(I+\rho(0))$ and $(s-r)(I+\rho(0))$ respectively. Therefore, we have
\begin{align}\label{mdgxi}
\Big\|\Big|\frac{\xi_t-x}{\rho_1}-\frac{\xi_s-x}{\rho_2}\Big|\Big\|_{12pd}\leq &c_{p,d}(t-r)^{-\frac{1}{2}}(t-s)^{\frac{1}{2}}+c_{p,d}\frac{(t-r)^{\frac{1}{2}}-(s-r)^{\frac{1}{2}}}{(t-r)^{\frac{1}{2}}(s-r)^{\frac{1}{2}}}(s-r)^{\frac{1}{2}}\nonumber\\
&+|z-x|\frac{(t-r)^{\frac{1}{2}}-(s-r)^{\frac{1}{2}}}{(t-r)^{\frac{1}{2}}(s-r)^{\frac{1}{2}}}\nonumber\\
\leq & C\big(|z-x|(s-r)^{-\frac{1}{2}}+1\big)(t-r)^{-\frac{1}{2}}(t-s)^{\frac{1}{2}}
\end{align}
Therefore, by (\ref{mdgxi}), Proposition \ref{cmijt} and Lemma \ref{blerk}, we have
\begin{align}\label{medcdtv211}
L_1+L_4\leq &C(t-r)^{-\frac{d}{2}}\Big[\exp\Big(-\frac{k|z-x|^2}{6pd(t-r)}\Big)+\exp\Big(-\frac{k|z-x|^2}{6pd(s-r)}\Big)\Big]\nonumber\\
&\times\big(1+|z-x|(s-r)^{-\frac{1}{2}}\big)(t-s)^{\frac{1}{2}},
\end{align}
and
\begin{align}\label{medcdtv212}
L_2+L_3\leq &C(t-r)^{-\frac{d}{2}}\exp\Big(-\frac{k|z-x|^2}{6pd(s-r)}\Big)(s-r)^{-\frac{1}{2}}(t-s)^{\frac{1}{2}}.
\end{align}
Plugging (\ref{medcdtv211}) and (\ref{medcdtv212}) into (\ref{medcdtv21}),  we have
\begin{align}\label{medcdtv210}
\int_{\mathbb{R}^d}\left\|J_1\right\|_{2p}dz\leq C (s-r)^{-\frac{1}{2}}(t-s)^{\frac{1}{2}}.
\end{align}
For $J_2$, notice that, by definition,
\begin{align*}
\langle D\xi_t^{j_2}-D\xi_s^{j_2}, D\xi_s^{j_1}\rangle_H=\sum_{k=1}^d\int_r^s \big(D^{(k)}_{\theta}\xi_t^{j_2}-D^{(k)}_{\theta}\xi_s^{j_2}\big)D^{(k)}_{\theta}\xi_s^{j_1}d\theta.
\end{align*}
By (\ref{sdexi1}), we have
\begin{align*}
D^{(k)}_{\theta}\xi_t^{j_2}-D^{(k)}_{\theta}\xi_s^{j_2}=\1_{[s,t]}(\theta)\delta_{j_2k}-\sum_{i=1}^d\1_{[r,t]}(\theta)\int_s^t D^{(k)}_{\theta}\xi_r^i dM^{ij_2}_r.
\end{align*}
By a argument similar to the one used in the proof of Lemma \ref{elgamma}, we can show that
\begin{align*}
\big\|\1_{[r,s]}(\theta)\big(D^{(k)}_{\theta}\xi_t^{j_2}-D^{(k)}_{\theta}\xi_s^{j_2}\big)\big\|_{2p}^2\leq C\1_{[r,s]}(\theta)(t-s).
\end{align*}
Therefore, by H\"{o}lder's and Minkowski's inequalities, we have
\begin{align}\label{mipdtv}
\left\|\langle D\xi_t^{j_2}-D\xi_s^{j_2}, D\xi_s^{j_1}\rangle_H\right\|_{2p}\leq &\sum_{k=1}^d\int_r^s \big\|\1_{[r,s]}(\theta)\big(D^{(k)}_{\theta}\xi_t^{j_2}-D^{(k)}_{\theta}\xi_s^{j_2}\big)\big\|_{4p}\big\|D^{(k)}_{\theta}\xi_s^{j_1}\big\|_{4p}d\theta\nonumber\\
\leq &C(s-r)(t-s)^{\frac{1}{2}}.
\end{align}
Choose $q\in(d,3pd]$. Let $p_1=\frac{q}{q-1}$, $p_2=2q$ and $p_3=6q$. Then $\frac{1}{p_1}+\frac{1}{p_2}+\frac{3}{p_3}=1$. Thus, by (\ref{mipdtv}), H\"{o}lder's inequality, Lemmas \ref{elgamma}, \ref{ehn}, \ref{blerk}, and Proposition \ref{cmijt},  we have
\begin{align*}
\left\|J_2\right\|_{2p}\leq &\sum_{j_1,j_2=1}^d\big\|\|\1_{B(x, 4\rho_2)}(\xi_s)\|^W_{p_2}\big\|_{6p}\big\|\|\partial_{j_2} Q_d(\xi_t-x)\|^W_{p_1}\big\|_{6p}\nonumber\\
&\times\big\|\|\langle D\xi_t^{j_2}-D\xi_s^{j_2}, D\xi_s^{j_1}\rangle_H \|^W_{p_3}\big\|_{18p}\big\|\|H_{(i)}(\xi_s,\phi_{\rho_2}^y(\xi_s))\|^W_{p_3}\big\|_{18p}\big\|\|\sigma_s^{j_1i}\|^W_{p_3}\big\|_{18p}\nonumber\\
\leq &C\exp\Big(-\frac{k|z-x|^2}{6pd(s-r)}\Big)(t-r)^{-\frac{d-1}{2}}(t-s)^{\frac{1}{2}}(s-r)^{-\frac{1}{2}}.
\end{align*}
As a consequence, we have
\begin{align}\label{medcdtv220}
\int_{\mathbb{R}^d}\left\|J_2\right\|_{2p}dz\leq C (t-s)^{\frac{1}{2}}.
\end{align}

Finally, combining (\ref{medcdtv1}), (\ref{medcdtv210}) and (\ref{medcdtv220}), we have
\begin{align}\label{idtmpd1}
\int_{\mathbb{R}^d}\left\|p^W(r,z;t,x)-p^W(r,z;s,x)\right\|_{2p}dz\leq C (s-r)^{-\frac{1}{2}}(t-s)^{\frac{1}{2}}.
\end{align}

On the other hand, by (\ref{edn}), we have
\begin{align}\label{idtmpd2}
\int_{\mathbb{R}^d}\|I_2\|_{2p}dz\leq \int_{\mathbb{R}^d}\|p^W(r,z;t,y)\|_{2p}+\|p^W(r,z;s,y)\|_{2p}\leq C.
\end{align}
Thus (\ref{idtmpd}) follows from (\ref{idtmpd1}) and (\ref{idtmpd2}).
\end{proof}

\begin{proof}[Proof of Proposition \ref{phcs}]
By the convolution representation (\ref{crd}), Burkholder-Davis-Gundy's, and Minkowski's inequalities, we have
\begin{align}\label{mpddm}
&\left\|u_t(y)-u_t(x)\right\|_{2p}\leq \Big\|\int_{\mathbb{R}^d}\mu(z)\left(p^W(0,z;t, y)-p^W(0,z;t,x)\right)dz\Big\|_{2p}\nonumber\\
&\qquad +\Big\|\int_0^t\int_{\mathbb{R}^d}u_r(z)\left(p^W(r,z;t,y)-p^W(r,z;t,x)\right)V(dz,dr)\Big\|_{2p}\nonumber\\
&\quad\leq\left\|\mu\right\|_{\infty}\int_{\mathbb{R}^d}\left\|p^W(0,z;t,y)-p^W(0,z;t,x)\right\|_{2p}dz\nonumber\\
&\qquad+\|\kappa\|_{\infty}^{\frac{1}{2}}\bigg(\int_0^t\Big(\int_{\mathbb{R}^d}\left\|u_r(z)\left(p^W(r,z;t,y)-p^W(r,z;t,x)\right)\right\|_{2p}dz\Big)^2dr\bigg)^{\frac{1}{2}}\nonumber\\
&\quad:=I_1+\|\kappa\|^{\frac{1}{2}}_{\infty}I_2.
\end{align}
Note that $I_1$ can be estimated by Lemma \ref{lmdtpirl}. For $I_2$, recall that $u(r,z)$ is independent of $p^W(r,z;t,y)$. Then, by Lemma \ref{eumcr} and \ref{lmdtpirl}, we have
\begin{align}\label{mpddm2}
I_2\leq&\bigg(\int_0^t\sup_{z\in\mathbb{R}^d}\left\|u_r(z)\right\|_{2p}^2\Big(\int_{\mathbb{R}^d}\left\|p^W(r,z;t,y)-p^W(r,z;t,x)\right\|_{2p}dz\Big)^2dr\bigg)^{\frac{1}{2}}\nonumber\\
\leq&C|y-x|^{\beta}\Big(\int_0^t(t-r)^{-\beta}dr\Big)^{\frac{1}{2}}\leq \frac{Ct^{\frac{1-\beta}{2}}}{\sqrt{1-\beta}}|y-x|^{\beta}.
\end{align}
Therefore (\ref{mpdhcsv}) follows from (\ref{mdtpirl}), (\ref{mpddm}) and (\ref{mpddm2}).

The proof of  (\ref{mpdhctv}) is quite similar. As did in (\ref{mpddm}), we can show that
\begin{align*}
&\left\|u_t(x)-u_s(x)\right\|_{2p}\leq \|\mu\|_{\infty}\int_{\mathbb{R}^d} \left\|p^W(0,z;t,x)-p^W(0,z;s,x)\right\|_{2p}dz\\
&\qquad+C\|\kappa\|_{\infty}^{\frac{1}{2}}\bigg[\int_s^t\sup_{z\in\mathbb{R}^d}\left\|u_r(z)\right\|_{2p}^2\Big(\int_{\mathbb{R}^d}\left\|p^W(r,z;t,x)\right\|_{2p}dz\Big)^2dr\bigg]^{\frac{1}{2}}\\
&\qquad+C\left\|\kappa\right\|_{\infty}^{\frac{1}{2}}\bigg[\int_0^s\sup_{z\in\mathbb{R}^d}\left\|u_r(z)\right\|_{2p}^2\Big(\int_{\mathbb{R}^d}\left\|\left(p^W(r,z;t,x)-p^W(r,z;s,x)\right)\right\|_{2p}dz\Big)^2dr\bigg]^{\frac{1}{2}}.
\end{align*}
Then, the estimate (\ref{mpdhctv}) follows from (\ref{idtmpd}), Proposition \ref{mectpd} and Lemma \ref{eumcr}.
\end{proof}

\begin{appendix} 

\section{Basic introduction on Malliavin calculus}

In this section, we present some preliminaries on the Malliavin calculus. We refer the readers to book of Nualart \cite{springer-06-nualart} for a detailed account on this topic.

Fix a time interval $[0,T]$. Let $B=\{B_t^1,\dots, B_t^d, 0\leq t\leq T\}$ be a standard $d$-dimensional Brownian motion on $[0,T]$. Denote by $\mathcal{S}$ the class of smooth random variables of the form
\begin{align}   \label{eq1}
G=g\left(B_{t_1}, \dots, B_{t_m}\right)=g\left(B_{t_1}^1,\dots,B_{t_1}^d, \dots,B_{t_m}^1,\dots,B_{t_m}^d\right),
\end{align}
where $m$ is any positive integer, $0\leq t_1<\dots<t_m\leq T$, and $g: \mathbb{R}^{md}\to \mathbb{R}$ is a smooth function that has all partial derivatives with at most polynomial growth. We make use of the notation $x=\left(x_i^k\right)_{1\le i\le m, 1\le k\le d}$ for any element $x\in  \R^{md}$.
The basic Hilbert space associated with  $B$ is  $H=L^2 \left([0,T]; \R^d\right)$.

\begin{definition}
For any $G\in\mathcal{S}$ given by (\ref{eq1}), the Malliavin derivative,  is the $H$-valued random variable $DG$  given by
\begin{align*}
D_{\theta}^{(k)} G=\sum_{i=1}^m \frac{\partial g}{\partial x_i^k}\left(B_{t_1}, \dots, B_{t_m}\right) \mathbf{1}_{[0,t_i]}(\theta), \quad 1\le k \le d,  \,\, \theta \in [0,T].
\end{align*}
\end{definition}
In the same way, for any $n\geq 1$, the iterated derivative  $D^n G$  of a random variable of the form (\ref{eq1}) is  a random variable with values in $H^{\otimes n}=L^2\left([0,T]^n; \R^{d^n}\right)$.
For each $p\ge 1$, the  iterated derivative  $D^n$ is a closable and unbounded operator  on $L^p(\Omega)$ taking values in $L^p(\Omega; H^{\otimes n})$.   For any $n\ge 1$, $p\ge 1$ and any Hilbert space $V$, we can introduce the Sobolev space $\mathbb{D}^{n,p}(V) $ of $V$-valued random variables as the closure of $\mathcal{S}$ with respect to the norm
\begin{align*}
\| G\|_{n,p,V}^2 =& \|G\|_{L^p(\Omega; V)}^2 + \sum_{k=1}^n    \|D^k G\|_{L^p(\Omega; H^{\otimes k} \otimes V)}^2\\
=&\big[\E\big(\|G\|_V^p\big)\big]^{\frac{2}{p}} + \sum_{k=1}^n   \big[\E \big(\|D^k G\|_{H^{\otimes k}\otimes V}^p\big)\big]^{\frac{2}{p}}.
\end{align*}

By definition, the divergence operator $\delta$ is the adjoint operator of $D$ in $L^2(\Omega)$. More precisely, $\delta$ is an unbounded operator on $L^2\left(\Omega; H\right)$, taking values in $L^2(\Omega)$. We denote by $\mathrm{Dom}(\delta)$ the domain of $\delta$. Then, for any $u=(u^1,\dots,u^d)\in \mathrm{Dom}(\delta)$, $\delta(u)$ is characterized by the duality relationship: for all for all $G\in \mathbb{D}^{1,2}=\D^{1,2}(\R)$.
\begin{align}
\E\left(\delta (u)G\right)= \E\left(\left\langle D G, u\right\rangle_H\right).
\end{align}

Let $F$ be an $n$-dimensional random vector, with components $F^i\in \mathbb{D}^{1,1}, 1\leq i\leq n$. We associate to $F$ an $n\times n$ random symmetric nonnegative definite matrix,  called the Malliavin matrix of $F$, denoted by $\gamma_F$. The entries of $\gamma_F$ are defined by
\begin{align}
\gamma^{ij}_F= \left\langle D F^i, D F^j\right\rangle_H=\sum_{k=1}^d\int_0^T D^{(k)}_{\theta}F^iD^{(k)}_{\theta}F^j d \theta.
\end{align}

Suppose that $F\in\cap_{p\geq 1}\mathbb{D}^{2,p}(\R^n)$, and its Malliavin matrix $\gamma_F$ is invertible. Denote by $\sigma_F$ the inverse of $\gamma_F$. Assume that $\sigma_F^{ij}\in \cap_{p\geq 1}\mathbb{D}^{1,p}$ for all $1\leq i,j\leq n$. Let $G\in\cap_{p\geq 1}\D^{1,2}$. Then $G\sigma_F^{ij}DF^k\in \mathrm{Dom}(\delta)$ for all $1\leq i,j,k\leq n$. Under the hypotheses, we define
\begin{align}\label{hfphif}
H_{(i)}(F,G)=-\sum_{j=1}^n\delta\left(G\sigma_F^{ji}DF^j\right),\quad 1\leq i\leq n.
\end{align}
If furthermore $H_{(i)}(F,G)\in\cap_{p\geq 1}\D^{1,p}$ for all $1\leq i\leq n$, then we define
\begin{align}\label{hhfphif}
H_{(i,j)}(F,G)=H_{(j)}\left(F,H_{(i)}(F, G)\right), \quad 1\leq i,j\leq n.
\end{align}

The following lemma is a Wiener functional version of Lemma 9 of Bally and Caramellino \cite{spa-11-bally-caramellino}.

\begin{lemma}\label{ipfpd}
Suppose that $F\in\cap_{p\geq 1}\mathbb{D}^{2,p}(\R^n)$, $(\gamma_F^{-1})^{ij}=\sigma_F^{ij}\in \cap_{p\geq 1}\mathbb{D}^{2,p}$ for all $1\leq i,j\leq n$, and $\phi\in C_b^1(\mathbb{R}^n)$. Then, for any $1\leq i\leq n$, we have
\begin{align}\label{ipfpd1}
H_{(i)}\left(F,\phi(F)\right)=&\partial_i\phi(F)+\phi(F)H_{(i)}(F,1).
\end{align}
Suppose that $F\in\cap_{p\geq 1}\mathbb{D}^{3,p}(\R^n)$ and $\phi\in C_b^2(\mathbb{R}^n)$. Then, for any $1\leq i,j\leq n$, we have
\begin{align}\label{ipfpd2}
H_{(i,j)}&\left(F,\phi(F)\right)=\partial_{ij}\phi(F)+\partial_i\phi(F)H_{(j)}(F,1)\nonumber\\
&\quad +\partial_{j}\phi(F) H_{(i)}(F,1)+\phi(F) H_{(i,j)}(F,1).
\end{align}
\end{lemma}
\begin{proof}
For any $F\in\cap_{p\geq 1}\mathbb{D}^{2,p}(\R^n)$ and $\phi\in C_b^1(\mathbb{R}^n)$, it is easy to check that $\phi(F)\in\cap_{p\geq 1}\D^{1,p}$. Then, $H_{(i)}(F,\phi (F))$ is well defined. For any $G\in\mathbb{D}^{1,2}$, by the duality of $D$ and $\delta$, we have
\begin{align}\label{edphigmdg}
\E\left(H_{(i)}\left(F,\phi(F)\right)G\right)=&-\sum_{j=1}^n\E\left(\delta\left(\phi(F)\sigma_F^{ji}DF^{j}\right)G\right)\nonumber\\
=&-\sum_{j=1}^n\E\left(\phi(F)\sigma_F^{ji}\left\langle DF^{j}, DG\right\rangle_H\right).
\end{align}
On the other hand, by the product rule for the operator $D$, we have
\begin{align*}
&\E\left(\phi(F)H_{(i)}(F,1)G\right)=-\sum_{j=1}^m\E\left(\left\langle\sigma_F^{ji}DF^{j},D\left(\phi(F)G\right)\right\rangle_H\right)\nonumber\\
&\quad=-\sum_{j=1}^m\E\left(\phi(F)\sigma_F^{ji}\left\langle DF^{j},DG\right\rangle_H\right)\nonumber-\sum_{j_1,j_2=1}^m\E\left(  G\partial_{j_2}\phi(F)\sigma_F^{j_1i}\left\langle DF^{j_1}, DF^{j_2}\right\rangle_H\right).
\end{align*}
Note that $\sigma_F$ is the inverse of $\gamma_F=\big(\langle DF^{i}, DF^{j}\rangle_H\big)_{i,j=1}^n$, then
\begin{align}\label{amie}
\sum_{j_1,j_2=1}^m\E\left(G\partial_{j_2}\phi(F)\sigma_F^{j_1i}\left\langle DF^{j_1}, DF^{j_2}\right\rangle_H\right)=\E\left(G\partial_i\phi(F)\right).
\end{align}
Then, (\ref{ipfpd1}) follows from (\ref{edphigmdg}) - (\ref{amie}). The equality (\ref{ipfpd2}) can be proved similarly.
\end{proof}

The next theorem is a density formula using the Riesz transformation. The formula was first introduced by Malliavin and Thalmaier (see Theorem 4.23 of \cite{springer-06-malliavin-thalmaier}), then further studied by Bally and Caramenillo \cite{spa-11-bally-caramellino}.

For any integer $n\geq 2$, let $Q_n$ be the $n$-dimensional Poisson kernel. That is,
\begin{equation}\label{posnkn}
Q_n(x)=
\begin{cases}
A_2^{-1}\log |x|, &n=1,\\
-A_n^{-1}|x|^{2-n}, &n>2,
\end{cases}
\end{equation}
 where $A_n$ is the area of the unit sphere in $\mathbb{R}^n$. Then, $\partial_i Q_n(x)=c_nx_i\left|x\right|^{-n}$, where $c_2=A_2^{-1}$ and $c_n=(\frac{n}{2}-1)A_n^{-1}$ for $n>2$.

The theorem below is the density formula for a class of differentiable random variables. 

\begin{theorem}\label{bedf}(Proposition 10 of  Bally and Caramenillo \cite{spa-11-bally-caramellino})
Let $F\in\cap_{p\geq 1}\mathbb{D}^{2,p}(\R^n)$. Assume that $(\gamma_F^{-1})^{ij}=\sigma_F^{ij}\in\cap_{p\geq 1}\D^{1,p}$  for all $1\leq i,j\leq n$. Then, the law of $F$ has a density $p_F$. 

More precisely, for any $x\in \R^n$ and $r>0$, let $B(x,r)$ be the sphere on $\R^n$ centered at $x$ with radius $r$. Suppose that $\phi\in C^1_b(\R^d)$, such that $\1_{B(0,1)}\leq \phi\leq \1_{B(0,2)}$, and $|\nabla \phi|\leq 1$. Define $\phi_{\rho}^x:= \phi(\frac{\cdot-x}{\rho})$ for any $\rho>0$ and $x\in\mathbb{R}^n$. Then,
\begin{align}\label{dfrk}
p_F(x)=&\sum_{i=1}^n\E\big( \partial_i Q_n (F-x)H_{(i)}(F,1)\big)\nonumber\\
=&\sum_{i=1}^n\E\big( \partial_i Q_n (F-x)H_{(i)}(F,\phi_{\rho}^x(F))\big)\nonumber\\
=&\sum_{i=1}^n\E\big( \1_{B_{(x,2\rho)}}(F)\partial_i Q_n (F-x)H_{(i)}(F,\phi_{\rho}^x(F))\big).
\end{align}
\end{theorem}

The next theorem provides the estimates for the density and its increment.
\begin{theorem}\label{tdpdift}
Suppose that $F$ satisfies the conditions in Theorem \ref{bedf}. Then, for any $p_2>p_1>n$, let $p_3=\frac{p_1p_2}{p_2-p_1}$, there exists a constant $C$ that depends on $p_1$, $p_2$ and $n$, such that
\begin{align}\label{de}
p_F(x)\leq &C \mathbb{P}(|F-x|<2\rho)^{\frac{1}{p_3}}\max_{1\leq i\leq n} \Big[\left\|H_{(i)}(F, 1)\right\|_{p_1}^{n-1}\Big(\frac{1}{\rho}+\left\|H_{(i)}(F, 1)\right\|_{p_2}\Big)\Big].
\end{align}
If furthermore, $F\in\cap_{p\geq 1}\mathbb{D}^{3,p}(\R^n)$, for any $x_1,x_2\in\mathbb{R}^n$, we can find $y=cx_1+(1-c)x_2$ for some $c\in (0,1)$ that depends on $x_1$, $x_2$. Then, there exist a constant $F$ the constant $C$ that depends on $p_1$, $p_2$, and $m$, such that
\begin{align}\label{dmvte}
&\left|p_F(x_1)-p_F(x_2)\right|\leq C|x_1-x_2| \mathbb{P}(|F-y|<4\rho)^{\frac{1}{p_3}}\nonumber\\
&\quad \times\max_{1\leq i,j\leq n} \Big[\left\|H_{(i)}(F, 1)\right\|_{p_1}^{n-1}\Big(\frac{1}{\rho^2}+\frac{2}{\rho}\left\|H_{(i)}(F, 1)\right\|_{p_2}+\left\|H_{(i,j)}(F, 1)\right\|_{q_2}\Big)\Big].
\end{align}
\end{theorem}

\textbf{Remark}: The inequalities stated in Theorem \ref{tdpdift} are an improved version of those estimates by Bally and Caramillino (see Theorem 8 of \cite{spa-11-bally-caramellino}). We refer to Nualart and Nualart (see Lemma 7.3.2 of \cite{cambridge-18-nualart-nualart}) for a related result. For the sake of completeness, we present below a proof of Theorem  \ref{tdpdift}. The proof follows the same idea as in Theorem 8 of \cite{spa-11-bally-caramellino}. The only difference occurs when choosing the radius of the ball in the estimate for the Poisson kernel. If we optimize the radius, then the exponent of $\|H_{(i)}(F,1)\|_p$ is $n-1$, instead of $\frac{q_1(n-1)}{q_1-n}>n-1$ in \cite{spa-11-bally-caramellino}.  In order to prove Theorem \ref{tdpdift}, we first give the estimate for the Poisson kernel:
\begin{lemma}\label{blerk}
Suppose that $F$ satisfy the conditions in  Theorem \ref{bedf}. For any $p>n$, let $q=\frac{p}{p-1}$. Then, there exists a constant $C>0$ depends on $m$ and $p$, such that
\begin{align}\label{rke}
\sup_{x\in\mathbb{R}^n}\left\|\partial_iQ_n(F-x)\right\|_q\leq \sup_{x\in\mathbb{R}^n}\left\||F-x|^{-(n-1)}\right\|_q\leq C \max_{1\leq i\leq n}\left\|H_{(i)}(F, 1)\right\|_p^{n-1}.
\end{align}
\end{lemma}

\begin{proof}
Assume that 
\[
\|p_F\|_{\infty}:=\sup_{x\in\R^d}p_F(x)<\infty.
\]
Denote by $\displaystyle M=\sup_{1\leq i\leq n}\|H_{(i)}(F,1)\|_p$. Then by H\"{o}lder's inequality, for all $x\in\R^d$, we have
\begin{align*}
p_F(x)=&\sum_{i=1}^n\E\big( \partial_i Q_n (F-x)H_{(i)}(F,1)\big)\leq \sum_{i=1}^m\|\partial_i Q_n (F-x)\|_q\|H_{(i)}(F,1)\|_p\\
\leq & n\sup_{x\in\mathbb{R}^n}\left\||F-x|^{-(n-1)}\right\|_qM,\nonumber
\end{align*}
which implies
\begin{align}\label{rke1}
\|p_F\|_{\infty}\leq n\sup_{x\in\mathbb{R}^n}\left\||F-x|^{-(n-1)}\right\|_qM.
\end{align}

In order to estimate $\||F-x|^{-(n-1)}\|_q$, choose any $\rho>0$. Then for all $x\in\R^n$,
\begin{align}\label{rke2}
\E(|F-x|^{-(n-1)q})=&\int_{\R^d}|y-x|^{-(n-1)q}p_F(y)dy\nonumber\\
=&\int_{|y-x|\leq \rho}|y-x|^{-(n-1)q}p_F(y)dy+\int_{|y-x|>\rho}|y-x|^{-(n-1)q}p_F(y)dy\nonumber\\
\leq &\|p_F\|_{\infty}\int_0^{\rho}r^{-(n-1)q}r^{n-1}dr+\rho^{-(n-1)q}\nonumber\\
=&k_{n,q}\|p_F\|_{\infty}\rho^{1-(n-1)(q-1)}+\rho^{-(n-1)q},
\end{align}
where $k_{n,q}=[1-(n-1)(q-1)]^{-1}$. The last equality is due to the fact that $1-(n-1)(q-1)>0$.

Combining (\ref{rke1}) and (\ref{rke2}), we have
\begin{align}\label{rke3}
\|p_F\|_{\infty}\leq\Big[ nk_{n,q}^{\frac{1}{q}}\|p_F\|_{\infty}^{\frac{1}{q}}\rho^{\frac{1-(n-1)(q-1)}{q}}+\rho^{-(n-1)}\Big]M.
\end{align}
By optimizing the right-hand side of (\ref{rke3}), we choose
\[
\rho=\rho^*:=\Big[\frac{(n-1)q}{n}\Big]^{\frac{q}{n}}\|p_F\|_{\infty}^{-\frac{1}{n}}.
\]
Plugging $\rho^*$ into (\ref{rke3}), we obtain
\begin{align*}
\|p_F\|_{\infty}\leq \bigg(nk_{n,q}^{\frac{1}{q}}\Big[\frac{(n-1)q}{n}\Big]^{\frac{1-(n-1)(q-1)}{n}}+\Big[\frac{(n-1)qM}{n}\Big]^{-\frac{q(n-1)}{n}}\bigg)M|p_F\|_{\infty}^{\frac{n-1}{n}}.
\end{align*}
Then, it follows that
\begin{align}\label{dinm}
\|p_F\|_{\infty}\leq CM^{n}=C\max_{1\leq i\leq n}\|H_{(i)}(F,1)\|_p^n
\end{align}
where $C$ is a constant that depends on $p$ and $n$. Thus (\ref{rke}) follows from (\ref{rke3}) and (\ref{dinm}).

The result can be generalized to the case without the assumption $\|p_F\|_{\infty}<\infty$ by the same argument as in Theorem 5 of \cite{spa-11-bally-caramellino}.
\end{proof}

\begin{proof}[Proof of Theorem \ref{tdpdift}]
	Choose $p_2>p_1>n$, let $p_3=\frac{p_1p_2}{p_2-p_1}$ and $q=\frac{p_1}{p_1-1}$. Then $\frac{1}{q}+\frac{1}{p_2}+\frac{1}{p_3}=1$. Thus by density formula (\ref{dfrk}) and H\"{o}lder's inequality, we have
	\begin{align}\label{de1}
p_F(x)\leq \sum_{i=1}^n\| \1_{B_{(x,2\rho)}}(F)\|_{p_3}\|\partial_i Q_n (F-x)\|_{q}\|H_{(i)}(F,\phi_{\rho}^x(F))\|_{p_2}.	
	\end{align}
	Then, (\ref{de}) is a consequence of (\ref{de1}), Lemma \ref{ipfpd} and \ref{blerk}. The inequality (\ref{dmvte}) can be proved similarly.
\end{proof}

 \end{appendix}

\end{document}